\documentclass[reqno,a4paper]{amsart}
\usepackage[english]{babel}

\usepackage{graphicx,mathrsfs,tikz,tikz-cd,latexsym,ifthen,amsmath,amsfonts,amssymb,amsthm,stmaryrd,fancyhdr,amscd,amsbsy,amstext,tensor, enumerate,enumitem,empheq,mathtools,mathpazo,verbatim,a4wide, epigraph} 
\mathtoolsset{showonlyrefs,showmanualtags}

\usepackage{xr-hyper}
\usepackage{color}
\definecolor{MyDarkBlue}{rgb}{0.15,0.25,0.45}
\usepackage{hyperref}
\hypersetup{anchorcolor=black, linktocpage=true,
colorlinks=true,
citecolor=red,
linkcolor=MyDarkBlue,
urlcolor=black,
breaklinks=true,
plainpages=true
}

\usepackage[hyperpageref]{backref} 

\usepackage[utf8x]{inputenc}

\usepackage[toc,page]{appendix}
\usepackage[mathscr]{euscript} 
\allowdisplaybreaks

\newif\ifpersonal
\ifpersonal
\newcommand{\personal}[1]{\textcolor[rgb]{0,0,1}{(Personal: #1)}}
\newcommand{\todo}[1]{\textcolor{red}{(Todo: #1)}}
\else
\newcommand*{\personal}[1]{\ignorespaces}
\newcommand*{\todo}[1]{\ignorespaces}
\fi

\usetikzlibrary{positioning,shapes,shadows,arrows}

\newcommand{\calA}{\mathcal A}
\newcommand{\calB}{\mathcal B}
\newcommand{\calC}{\mathcal C}
\newcommand{\calD}{\mathcal D}
\newcommand{\calE}{\mathcal E}
\newcommand{\calF}{\mathcal F}
\newcommand{\calG}{\mathcal G}
\newcommand{\calH}{\mathcal H}

\newcommand{\calL}{\mathcal L}
\newcommand{\calM}{\mathcal M}

\newcommand{\calO}{\mathcal O}
\newcommand{\calP}{\mathcal P}
\newcommand{\calQ}{\mathcal Q}

\newcommand{\calS}{\mathcal S}
\newcommand{\calT}{\mathcal T}
\newcommand{\calU}{\mathcal U}

\newcommand{\calX}{\mathcal X}

\newcommand{\R}{\mathbb R}

\newcommand{\C}{\mathbb C}
\newcommand{\Z}{\mathbb Z}
\newcommand{\Q}{\mathbb Q}
\newcommand{\T}{\mathbb T}
\newcommand{\PP}{\mathbb P}
\newcommand{\A}{\mathbb A}
\newcommand{\G}{\mathbb G}

\newcommand{\LL}{\mathbb L}
\newcommand{\V}{\mathbb V}

\newcommand{\scrO}{\mathscr O}

\newcommand{\scrS}{\mathscr S}

\newcommand{\sfB}{\mathsf B}
\newcommand{\sfH}{\mathsf H}

\newcommand{\sfK}{\mathsf K}
\newcommand{\sfR}{\mathsf R}

\newcommand{\sfT}{\mathsf T}

\newcommand{\id}{\mathsf{id}}
\newcommand{\Spec}{\mathsf{Spec}}
\newcommand{\fib}{\mathsf{fib}}
\DeclareMathOperator*{\colim}{colim}

\newcommand{\Hom}{\mathsf{Hom}}
\newcommand{\calHom}{\mathcal{H}\mathsf{om}}
\newcommand{\calEnd}{\mathcal{E}\mathsf{nd}}
\newcommand{\Ext}{\mathsf{Ext}}
\newcommand{\calExt}{\mathcal{E}\mathsf{xt}}
\newcommand{\Aut}{\mathsf{Aut}}
\newcommand{\End}{\mathsf{End}}

\newcommand{\Ind}{\mathsf{Ind}}
\newcommand{\Fun}{\mathsf{Fun}}

\newcommand{\ev}{\mathsf{ev}}
\DeclareMathOperator*{\fcolim}{``colim''}

\newcommand{\dR}{\mathsf{dR}}

\newcommand{\B}{\mathsf{B}}
\newcommand{\Dol}{\mathsf{Dol}}
\newcommand{\Del}{\mathsf{Del}}
\newcommand{\nil}{\mathsf{nil}}

\newcommand{\PSh}{\mathsf{PSh}}

\newcommand{\dStn}{\mathsf{dStn}_\C}
\newcommand{\dAn}{\mathsf{dAn}_\C}
\newcommand{\dAnSt}{\mathsf{dAnSt}}

\newcommand{\cC}{\mathcal C}

\newcommand{\cS}{\mathcal S}
\newcommand{\cT}{\mathcal T}

\newcommand{\Cat}{\mathsf{Cat}}
\newcommand{\Alg}{\mathsf{Alg}}
\newcommand{\CAlg}{\mathsf{CAlg}}
\newcommand{\Mod}{\textrm{-} \mathsf{Mod}}
\newcommand{\dAff}{\mathsf{dAff}}

\newcommand{\PrL}{\mathcal P \mathsf{r}^{\mathsf{L}}}
\newcommand{\monPrLst}{{\mathcal P\mathsf r}^{\mathsf L, \, \otimes}_{\mathsf{st}}}
\newcommand{\cTdisc}{\cT_{\mathsf{disc}}}
\newcommand{\RTop}{{}^{\mathsf R} \mathsf{Top}}

\newcommand{\Map}{\mathsf{Map}}

\newcommand{\catPerf}{\mathsf{Perf}}
\newcommand{\catPerffilt}{\mathsf{Perf}^{\mathsf{filt}}}
\newcommand{\catPerfgr}{\mathsf{Perf}^{\mathsf{gr}}}
\newcommand{\catCoh}{\mathsf{Coh}}
\newcommand{\catCohb}{\mathsf{Coh}^{\mathsf{b}}}

\newcommand{\catAPerf}{\mathsf{APerf}}

\newcommand{\catQCoh}{\mathsf{QCoh}}
\newcommand{\catDb}{\mathsf{D}^\mathsf{b}}

\newcommand{\bfRep}{\mathbf{Rep}}
\newcommand{\bfPerf}{\mathbf{Perf}}
\newcommand{\bfAPerf}{\mathbf{APerf}}

\newcommand{\bfAnPerf}{\mathbf{AnPerf}}
\newcommand{\bfPerfext}{\mathbf{Perf}^{\,\mathsf{ext}}}
\newcommand{\bfCoh}{\mathbf{Coh}}
\newcommand{\bfAnCoh}{\mathbf{AnCoh}}

\newcommand{\bfCohext}{\mathbf{Coh}^{\mathsf{ext}}}
\newcommand{\bfAnCohext}{\mathbf{AnCoh}^{\mathsf{ext}}}
\newcommand{\bfBun}{\mathbf{Bun}}
\newcommand{\bfAnBun}{\mathbf{AnBun}}
\newcommand{\bfHiggs}{\mathbf{Higgs}}
\newcommand{\bfMap}{\mathbf{Map}}
\newcommand{\bfAnMap}{\mathbf{AnMap}}
\newcommand{\BGL}{\mathsf{BGL}}
\newcommand{\trunc}[1]{\tensor*[^{\mathsf{cl}}]{#1}{}}
\newcommand{\op}{^{\mathsf{op}}}
\newcommand{\an}{^{\mathsf{an}}}
\newcommand{\Sym}{\mathsf{Sym}}
\newcommand{\dSt}{\mathsf{dSt}}
\newcommand{\Sh}{\mathsf{Sh}}
\newcommand{\dGeom}{\mathsf{dGeom}}
\newcommand{\dGeomqc}{\mathsf{dGeom}^{\mathsf{qc}}}
\newcommand{\Corr}{\mathsf{Corr}}
\newcommand{\TwoSeg}{2\textrm{-}\mathsf{Seg}}
\DeclareMathOperator*{\flim}{``lim''}

\makeatletter
\newcommand{\ostar}{\mathbin{\mathpalette\make@circled\star}}
\newcommand{\make@circled}[2]{%
	\ooalign{$\m@th#1\smallbigcirc{#1}$\cr\hidewidth$\m@th#1#2$\hidewidth\cr}%
}
\newcommand{\smallbigcirc}[1]{%
	\vcenter{\hbox{\scalebox{0.77778}{$\m@th#1\bigcirc$}}}%
}
\makeatother

\newcommand{\triend}{\parbox{2mm}{\hfill} \hfill\mbox{\hspace{0.2mm}}\hfill$\triangle$}
\newcommand{\ocend}{\parbox{2mm}{\hfill} \hfill\mbox{\hspace{0.2mm}}\hfill$\oslash$}

\newtheorem{theorem}{Theorem}
\newtheorem{proposition}[theorem]{Proposition}
\newtheorem{lemma}[theorem]{Lemma}
\newtheorem{corollary}[theorem]{Corollary}
\newtheorem{conjecture}[theorem]{Conjecture}
\newtheorem{corollary*}{Corollary}
\newtheorem*{theorem*}{Theorem}
\newtheorem*{proposition*}{Proposition}
\newtheorem*{conjecture*}{Conjecture}

\numberwithin{equation}{section}
\numberwithin{theorem}{section}

\theoremstyle{remark}
\newtheorem{ex}[theorem]{Example}
\newenvironment{example}{\begin{ex}}{\triend\end{ex}}

\theoremstyle{remark}
\newtheorem{rem}[theorem]{Remark}
\newenvironment{remark}{\begin{rem}}{\triend\end{rem}}

\theoremstyle{definition}
\newtheorem{defin}[theorem]{Definition}
\newenvironment{definition}{\begin{defin}}{\ocend\end{defin}}
\newtheorem{construction}[theorem]{Construction}

\title[Two-dimensional categorified Hall algebras]{Two-dimensional categorified Hall algebras}

\author[M.~Porta]{Mauro Porta}
\address[Mauro Porta]{Institut de recherche mathématique avancée (IRMA), UMR 7501, Université de Strasbourg, 7 rue René-Descartes, 67084 Strasbourg Cedex, France}
\curraddr{}
\email{\href{mailto:porta@math.unistra.fr}{porta@math.unistra.fr}}

\author[F.~Sala]{Francesco Sala}
\address[Francesco Sala]{Università di Pisa, Dipartimento di Matematica, Largo Bruno Pontecorvo 5, 56127 Pisa (PI), Italy}
\address{Kavli IPMU (WPI), UTIAS, The University of Tokyo, Kashiwa, Chiba 277-8583, Japan}
\curraddr{}
\email{\href{mailto:francesco.sala@unipi.it}{francesco.sala@unipi.it}}

\thanks{The work of the second-named author is partially supported by World Premier International 
	Research Center Initiative (WPI), MEXT, Japan, by JSPS KAKENHI Grant number JP17H06598 and 
	by JSPS KAKENHI Grant number JP18K13402.}


\begin{document}
	
	\begin{flushright}
		IPMU--19--0025
	\end{flushright}
	
	\vskip 1cm

\begin{abstract}
	
	In the present paper, we introduce two-dimensional \textit{categorified Hall algebras} of smooth curves and smooth surfaces. A categorified Hall algebra is an associative monoidal structure on the stable $\infty$-category $\catCohb(\R\calM)$ of complexes of sheaves with bounded coherent cohomology on a derived moduli stack $\R\calM$. In the surface case, $\R\calM$ is a suitable derived enhancement of the moduli stack $\calM$ of coherent sheaves on the surface. This construction categorifies the K-theoretical and cohomological Hall algebras of coherent sheaves on a surface of Zhao and Kapranov-Vasserot. In the curve case, we define three categorified Hall algebras associated with suitable derived enhancements of the moduli stack of Higgs sheaves on a curve $X$, the moduli stack of vector bundles with flat connections on $X$, and the moduli stack of finite-dimensional local systems on $X$, respectively. In the Higgs sheaves case we obtain a categorification of the K-theoretical and cohomological Hall algebras of Higgs sheaves on a curve of Minets and Sala-Schiffmann, while in the other two cases our construction yields, by passing to $\sfK_0$, new K-theoretical Hall algebras, and by passing to $\sfH_\ast^{\mathsf{BM}}$, new cohomological Hall algebras. Finally, we show that the Riemann--Hilbert and the non-abelian Hodge correspondences can be lifted to the level of our categorified Hall algebras of a curve.
	
\end{abstract}

\maketitle\thispagestyle{empty}

\tableofcontents

\bigskip\section{Introduction}

In this work we introduce two-dimensional \textit{categorified Hall algebras} of smooth curves and smooth surfaces.
A categorified Hall algebra is an associative monoidal structure ``à la Hall'' on the dg-category $\catCohb(\R\calM)$\footnote{We mean the bounded derived category of complexes of sheaves with coherent cohomology. A more classical notation would be $\mathsf D^{\mathsf b}_{\mathsf{coh}}(\R \calM)$. In the main body of the paper we will construct directly stable $\infty$-categories, without passing through explicit dg-enhancements. Moreover, \textit{associative monoidal structure} is to be technically understood as $\mathbb E_1$-monoidal structure.} on a derived moduli stack $\R\calM$.
In the surface case, $\R\calM$ is a suitable derived enhancement of the moduli stack $\calM$ of coherent sheaves on the surface.
This construction categorifies the K-theoretical Hall algebra of zero-dimensional coherent sheaves on a surface $S$ \cite{Zhao_Hall} and the K-theoretical and cohomological Hall algebras of coherent sheaves on $S$ \cite{KV_Hall}.
In the curve case, we define three categorified Hall algebras associated with suitable derived enhancements of the moduli stack of Higgs sheaves on a curve $X$, the moduli stack of flat vector bundles on $X$, and the moduli stack of local systems on $X$, respectively.
In the Higgs sheaves case, we obtain a categorification of the K-theoretical and cohomological Hall algebras of Higgs sheaves on a curve \cite{Minets_Hall, Sala_Schiffmann}, while in the other two cases we obtain, as a by-product, the construction of the corresponding K-theoretical and cohomological Hall algebras. 
While the underlying K-theoretical and cohomological Hall algebras can also be obtained via perfect obstruction theories and are insensitive to the derived enhancements we use here, our categorified versions depend in a substantial way on the existence of a sufficiently natural derived enhancement.
To the best of our knowledge, it is not possible to obtain such categorifications using perfect obstruction theories.

Before providing precise statements of our results, we shall briefly recall the literature about K-theoretical and cohomological Hall algebras.

\subsection{Review of the Hall convolution product}

Let $\calA$ be an abelian category and denote by $\calM_\calA$ the corresponding \textit{moduli stack of objects}: $\calM_\calA$ is a geometric derived stack over $\C$ parameterizing \textit{families} of objects in $\calA$.
In particular, its groupoid of $\C$-points $\calM_\calA(\C)$ coincides with the groupoid of objects of $\calA$.
Similarly, we can consider the moduli stack $\calM^{\mathsf{ext}}_\calA$ parameterizing \textit{families} of short exact sequences in $\calA$ and form the following diagram:
\begin{align}\label{eq:convolution}
	\begin{tikzcd}[ampersand replacement = \&]
		\& \calM_\calA^{\mathsf{ext}} \arrow{dl}[swap]{p} \arrow{dr}{q} \\
		\calM_\calA\times \calM_\calA \& \& \calM_\calA
	\end{tikzcd}\ , \quad \begin{tikzcd}[ampersand replacement = \&, column sep = tiny]
		\& 0 \to E_1 \to E \to E_2 \to 0 \arrow[mapsto]{dl} \arrow[mapsto]{dr} \\
		(E_1, E_2) \& \& E
	\end{tikzcd}
\end{align}
When the maps $p$ and $q$ are sufficiently well behaved, passing to (an oriented) Borel--Moore homology\footnote{Examples of oriented Borel--Moore homology theories are the $G_0$-theory (i.e., the Grothendieck group of coherent sheaves), Chow groups, elliptic cohomology.} yields a product map
\begin{align}
q_\ast \circ p^\ast \colon \sfH_\ast^{\mathsf{BM}}( \calM_\calA ) \otimes \sfH_\ast^{\mathsf{BM}}( \calM_\calA ) \longrightarrow \sfH_\ast^{\mathsf{BM}}( \calM_\calA ) \ ,
\end{align}
which can then been proven to be associative.
In what follows, we refer to the above multiplicative structure as a ``cohomological Hall algebra'' (\textit{CoHA} for short) attached to $\calA$.

The existence of the above product does not come for free.
Typically, one needs a certain level of regularity for $p$ (e.g.\ smooth or lci).
In turn, this imposes severe restrictions on the abelian category $\calA$.
For instance, if $\calA$ has cohomological dimension one, then $p$ is smooth, but this is typically false when $\calA$ has cohomological dimension two.
Quite recently, there has been an increasing amount of research around two-dimensional CoHAs (see e.g. \cite{SV_elliptic, SV_Cherednik, SV_generators, SV_Yangians, YZ_CoHA, YZ_Yangian, KV_Hall}).
We will give a thorough review of the historical development in \S\ref{ss:historical}, but for the moment let us say that the first goal of this paper is to provide an approach to the construction of the convolution product \textit{à la Hall} that can work uniformly in the two-dimensional setting.
The key of our method is to consider a suitable natural derived enhancements $\R \calM_\calA$ and $\R \calM^{\mathsf{ext}}_\calA$ of the moduli stacks $\calM_\calA$ and $\calM^{\mathsf{ext}}_\calA$, respectively.

The use of derived geometry is both natural and expected, and made an early explicit appearance in \cite{Negut_shuffle}.
The effectiveness of this method can be easily understood via the following two properties:
\begin{enumerate}
	\item the map
	\begin{align}
		\R p \colon \R \calM^{\mathsf{ext}}_\calA \longrightarrow \R \calM_\calA \times \R \calM_\calA
	\end{align}
	has better regularity properties than its underived counterpart.
	When $\calA$ has cohomological dimension two, $\R p$ is typically lci, while $p$ is not.
	
	\item \label{item:property-(2)} Oriented Borel--Moore homology theories are insensitive to the derived structure, hence yielding natural isomorphisms\footnote{This is best seen in the case of the $G_0$-theory - cf.\ Proposition~\ref{prop:quasi_compact_G_truncation}.}
	\begin{align}
		\sfH_\ast^{\mathsf{BM}}( \calM_\calA ) \simeq \sfH_\ast^{\mathsf{BM}}( \R \calM_\calA ) \ .
	\end{align}
	\end{enumerate}
These two properties constitute the main leitmotiv of the current paper.
The upshot is that we can use the map $\R p$ in order to construct the Hall product in a much more general setting.

As announced at the beginning, the use of derived geometry has another pleasant consequence: it allows us to categorify the CoHAs considered above.
The precise formulation of this construction, as well as the study of its first properties, is the second goal of this work.
More specifically, we show that the (derived) convolution diagram induces an associative monoidal structure
\begin{align}
	\ostar_{\mathrm{Hall}} \colon \catCohb( \R \calM_\calA ) \otimes \catCohb( \R\calM_\calA ) \longrightarrow \catCohb( \R\calM_\calA )
\end{align}
on the dg-category of complexes of sheaves with bounded coherent cohomology on $\R \calM_\calA$.
We refer to this monoidal dg-category as the \textit{two-dimensional categorified Hall algebra} (\textit{Cat-HA} in the following) \textit{of $\calA$}.

From the Cat-HA we can extract a certain number of CoHAs.
Most notably, we recover a CoHA structure on the spectrum of $G$-theory.
Notice that this would be impossible if we limited ourselves to consider $\catCohb( \R \calM_\calA )$ as a triangulated category --- see e.g.\ \cite{Schlichting,Toen_Vezzosi_Ktheory}.
As a closing remark, let us emphasize that, unlike oriented Borel--Moore homology theories, our Cat-HA is very sensitive to the derived structure of $\R \calM_\calA$.
In other words, property \eqref{item:property-(2)} above fails in the categorified setting:
\begin{align}
	 \catCohb( \R \calM_\calA ) \quad \text{and} \quad \catCohb( \calM_\calA )
\end{align}
are no longer equivalent.
Furthermore, the same difficulties encountered when trying to construct the CoHA out of $\calM_\calA$ prevent, in an even harsher way, endowing $\catCohb( \calM_\calA )$ with an associative monoidal structure.
Indeed, if one simply cares about the construction of the CoHA, it would be possible to bypass the use of derived geometry by using one of his shadows, i.e.\ perfect obstruction theories.
However, the complexity of the higher coherences involved in the construction of the Cat-HA leads us to believe that an approach to categorification via perfect obstruction theories is highly unlikely. 

\subsection{Main results} 

We can summarize the main contributions of this paper as follows: on the one hand, we construct many examples of two-dimensional categorified Hall algebras (Cat-HAs) attached to curves and surfaces.
On the other hand, we show that from these new Cat-HAs one can extract the known constructions of K-theoretical Hall algebras of surfaces and of Higgs sheaves on a curve.
As a byproduct, our approach provides K-theoretical and cohomological Hall algebras associated to flat vector bundles and local systems on a curve.

\subsubsection*{Categorified Hall algebras}

Let $X$ be a smooth proper $\C$-scheme.
In \S \ref{s:coh} we introduce a derived enhancement $\bfCoh(X)$ of the (classical) geometric derived stack of coherent sheaves on $X$.
Informally, its functor of points assigns to every affine derived $\C$-scheme $S$ the space of $S$-flat perfect complexes on $X \times S$.
We show in Proposition \ref{prop:coh_geometric} that $\bfCoh(X)$ is a geometric derived stack which is locally of finite presentation.

Similarly, we introduce the derived stack $\bfCohext(X)$ which, roughly speaking, parameterizes \textit{extensions} of $S$-flat of perfect complexes on $X \times S$.
These derived stacks can be organized in the following convolution diagram
\begin{align}\label{eq:convolution-derived}
	\begin{tikzcd}[ampersand replacement = \&]
		{} \& \bfCohext(X) \arrow{dl}[swap]{p} \arrow{dr}{q} \\
		\bfCoh(X)\times \bfCoh(X) \& \& \bfCoh(X)
	\end{tikzcd}
\end{align}
of the form \eqref{eq:convolution}.
The main input to our construction is the computation of the tor-amplitude of the cotangent complex of $p$:

\begin{proposition}[see Proposition \ref{prop:extremal_projections_scheme}]
	The relative cotangent complex $\mathbb L_p$ of $p \colon \bfCohext(X) \to \bfCoh(X)$ has tor-amplitude within $[-1,n-1]$, where $n$ is the dimension of $X$.
\end{proposition}

When $X$ is a surface, the cotangent complex of $p$ has tor-amplitude within $[-1,1]$.
This is to say that $p$ is derived lci, and in particular we obtain a well-defined functor
\begin{align}
	\ostar \coloneqq q_\ast \circ p^\ast \colon \catCohb( \bfCoh(X) ) \otimes \catCohb( \bfCoh(X) ) \longrightarrow \catCohb( \bfCoh(X) ) \ .
\end{align}
This implies:

\begin{theorem}[see Proposition~\ref{prop:coh_2_Segal}]
	Let $X$ be a smooth and proper complex surface.
	Then the functor $\ostar$ can be promoted to an $\mathbb E_1$-monoidal structure on the stable $\infty$-category $\catCohb( \bfCoh(X) )$.
\end{theorem}

We refer to $\catCohb( \bfCoh(X) )$ together with its $\mathbb E_1$-monoidal structure $\ostar$ as the \textit{categorified Hall algebra of the surface $X$}.
In a nutshell, the construction goes as follows.
The convolution diagram considered above is part of a richer combinatorial structure that can be seen as a simplicial object in derived stacks
\begin{align}
	\calS_\bullet \bfCoh(X) \colon \mathbf \Delta \longrightarrow \dSt \ .
\end{align}
In low dimensions, the simplexes of $\calS_\bullet \bfCoh(X)$ can be described as follows:
\begin{align}
	\scrS_0 \bfCoh(X) \simeq \Spec(\C) , \quad \scrS_1 \bfCoh(X) \simeq \bfCoh(X), \quad \scrS_2 \bfCoh(X) \simeq \bfCohext(X) \ ,
\end{align}
and the simplicial maps induce the maps $p$ and $q$ above.
The simplicial object $\calS_\bullet \bfCoh(X)$ is known as the \textit{Waldhausen construction} of $\bfCoh(X)$, and one can summarize its main properties by saying that it is a \textit{$2$-Segal object} in the sense of Dyckerhoff--Kapranov \cite{Dyckerhoff_Kapranov_Higher_Segal}.

Its relevance for us is that \cite[Theorem 11.1.6]{Dyckerhoff_Kapranov_Higher_Segal} provides a canonical $\infty$-functor
\begin{align}
	\TwoSeg( \dSt ) \longrightarrow \Alg_{\mathbb E_1}( \Corr^\times( \dSt ) ) \ . 
\end{align}
In other words, we can attach to every $2$-Segal object an $\mathbb E_1$-monoid object in the category of correspondences in derived stacks.
In order to convert these data into the higher coherences of the Cat-HA, we make use of the Gaitsgory--Rozenblyum correspondence machine \cite{Gaitsgory_Rozenblyum_Study_I}.
Let
\begin{align}
	\Corr^\times( \dGeom )_{\mathsf{rps}, \mathsf{lci}} \hookrightarrow \Corr^\times( \dSt )
\end{align}
be the subcategory whose objects are derived geometric (i.e.\ higher Artin) stacks, and whose class of horizontal (resp.\ vertical) morphisms is the class of maps representable by proper schemes (resp.\ lci morphisms).
Then the universal property of the $(\infty,2)$-category of correspondences of Gaitsgory--Rozenblyum provides us with a lax monoidal functor
\begin{align}
	\catCohb \colon \Corr^\times( \dGeom )_{\mathsf{rps}, \mathsf{lci}} \longrightarrow \Cat_\infty^{\mathsf{st}} \ , 
\end{align}
with values in the $\infty$-category of stable $\infty$-categories.
Being lax monoidal, this functor preserves $\mathbb E_1$-monoid objects, therefore delivering the Cat-HA.

\begin{remark}
	If $X$ is projective and $H$ is an ample divisor, similar results hold for the stack $\bfCoh^{\mathsf{ss}, \, p(m)}(X)$ of Giese\-ker $H$-semistable coherent sheaves on $X$ with reduced Hilbert polynomial equal to a fixed monic polynomial $p(m)\in \Q[m]$. Moreover, if $X$ is quasi-projective, the results above hold for the stack $\bfCoh_{\mathsf{prop}}^{\leqslant d}(X)$ of coherent sheaves on $X$ with proper support and dimension of the support less than or equal to an integer $d$. Finally, if the surface is toric, minimal variations in our construction (discussed in \S\ref{ss:equivariant}) allow us to consider the toric-equivariant setting.
	
	One can also extend the above construction to obtain Cat-HAs associated to derived moduli stacks of Simpson's semistable properly supported sheaves with fixed reduced Hilbert polynomial on a smooth (quasi-)projective surface. An analysis of these Cat-HAs has been carried out in \cite{Coha_mckay} when the surface is the minimal resolution of a Kleinian singularity.
\end{remark}

As we said before, our second main source of examples is the two-dimensional Cat-HA that can be attached to smooth projective complex curves $X$.
There are three types of such examples, coming respectively from local systems, flat vector bundles, and Higgs sheaves on $X$.\footnote{Recall that a Higgs sheaf is a pair $(E, \varphi\colon E\to \Omega_X^1\otimes E)$, where $E$ is a coherent sheaf on $X$ and $\varphi$ a morphism of $\scrO_{X}$-modules, called a \textit{Higgs field}. Here, $\Omega_X^1$ is the sheaf of 1-forms of $X$. On the other hand, by a flat vector bundle we mean a vector bundle endowed with a flat connection. Finally, recall that a local system can be interpreted as a finite-dimensional representation of the fundamental group of $X$.}
A uniform treatment of these Cat-HAs is made possible by Simpson's formalism of \textit{shapes}.
These are derived stacks attached to the curve $X$, written
\begin{align}
	X_\B \ , \quad X_\dR\ , \quad X_\Dol \ . 
\end{align}
We refer to  the compendium \cite{Porta_Sala_Shapes} for the precise definition of these derived stacks.
However, let us say straight away that their usefulness lies in the fact that coherent sheaves on $X_\B$ (resp.\ $X_\dR$, $X_\Dol$) canonically coincide with local systems (resp.\ flat vector bundles, Higgs sheaves) on $X$.
Using these shapes, we can easily make sense of the derived enhancements
\begin{align}
	\bfCoh( X_\B )\ , \quad \bfCoh( X_\dR )\ , \quad \bfCoh( X_\Dol )
\end{align}
of the classical stacks of local systems, flat vector bundles and Higgs sheaves on $X$, respectively.

The construction of the convolution diagram (and of the $2$-Segal object) can be carried out on in this setting without any additional difficulty.
The key computation of the tor-amplitude of the map $p$ in this context is discussed in \S\ref{ss:ext-shapes}.
Every case has to be analyzed on its own, because the proof relies on specific features of the type of sheaves that are considered.
From here, the same method discussed for surfaces yields:

\begin{theorem}[see Theorem \ref{thm:coh_algebra_in_correspondence}]
	Let $X$ be a smooth projective complex curve.
	The convolution diagram induces an $\mathbb E_1$-monoidal structure on the stable $\infty$-categories
	\begin{align}
		 \catCohb( \bfCoh(X_\B) )\ , \quad \catCohb( \bfCoh(X_\dR) )\ , \quad \catCohb( \bfCoh(X_\Dol) ) \ .
	\end{align}
\end{theorem}

We refer to these $\mathbb E_1$-monoidal categories as the \textit{Betti}, \textit{de Rham} and \textit{Dolbeault Cat-HAs}.
We denote their underlying tensor products as $\ostar_\B$, $\ostar_\dR$ and $\ostar_\Dol$, respectively.
Our formalism also allows considering the natural $\C^\ast$-action on $\bfCoh(X_\Dol)\simeq \sfT^\ast \bfCoh(X)$ that ``scales the fibers" and so we introduce the corresponding $\C^\ast$-equivariant version of the Dolbeault Cat-HA (cf.\ \S\ref{ss:equivariant}).

It is a natural question to try to relate these three Cat-HAs attached to a curve.
Our first result in this direction, concerning the de Rham and the Betti Cat-HAs, is of analytic nature.
It can be informally stated by saying that the Riemann--Hilbert correspondence respects the Hall convolution structure:

\begin{theorem}[Cat-HA version of the derived Riemann--Hilbert correspondence]\label{thm:Cat-HA-RH}
	Let $X$ be a smooth projective complex curve.
	Then:
	\begin{enumerate}\itemsep0.2cm
		\item the convolution diagrams for the analytifications $\bfCoh(X_\dR)\an$ and $\bfCoh(X_\B)\an$
		induce an $\mathbb E_1$-monoidal structure on the stable $\infty$-categories $\catCohb(\bfCoh(X_\dR)\an)$ and $\catCohb(\bfCoh(X_\B)\an)$, written
		\begin{align}
			(\catCohb( \bfCoh(X_\dR)\an ), \ostar_\dR\an)\ , \quad ( \catCohb( \bfCoh(X_\B)\an ), \ostar_\B\an)\  . 
		\end{align}
		
		\item There is a natural diagram of stable $\mathbb E_1$-monoidal $\infty$-categories and monoidal functors
		\begin{align}
			\begin{tikzcd}[ampersand replacement = \&]
				(\catCohb( \bfCoh(X_\B) ), \ostar_\B ) \arrow{d} \& (\catCohb( \bfCoh(X_\dR) ), \ostar_\dR ) \arrow{d} \\
				(\catCohb( \bfCoh(X_\B)\an ), \ostar_\B\an ) \arrow{r} \& (\catCohb( \bfCoh(X_\dR)\an ), \ostar_\dR\an ) 
			\end{tikzcd}\ ,
		\end{align}
		where the vertical functors are induced by analytification and the horizontal functor is induced by the Riemann--Hilbert transformation $\eta_{\mathsf{RH}} \colon X_\dR\an \to X_\B\an$ of \cite{Porta_Derived}.
		Furthermore, the horizontal functor is an equivalence.
	\end{enumerate}
\end{theorem}

The new ingredient in this theorem is the use of derived complex analytic geometry, first introduced by J.\ Lurie in \cite{DAG-IX} and further expanded by the first-named author in \cite{Porta_GAGA,Porta_Yu_Representability,Porta_Holstein_Mapping}.
The key point is to prove that the Riemann--Hilbert correspondence of \cite{Porta_Derived} can be lifted to the $\mathbb E_1$-monoidal setting, and this is achieved by the natural transformation $\eta_{\mathsf{RH}}$ already mentioned in the above statement.

The relation between the de Rham and the Dolbeault categorified Hall algebras is more subtle. In order to state it, one has to use another shape of Simpson, the \textit{Deligne shape} $X_\Del\to \A^1$. Then the derived stack $\bfCoh_{/\A^1}(X_\Del)$ is the derived moduli stack of \textit{Deligne's $\lambda$ connections} on $X$. Such a stack \textit{interpolates the de Rham moduli stack with the Dolbeault moduli stack}: it naturally lives over $\A^1$ and one has
\begin{align}
	\bfCoh_{/\A^1}(X_\Del) \times_{\A^1} \{0\} \simeq \bfCoh(X_\Dol) \quad \text{and} \quad \bfCoh_{/\A^1}(X_\Del) \times_{\A^1} \{1\} \simeq \bfCoh(X_\dR) \ . 
\end{align}
We restrict ourselves to the open substack $\bfCoh_{/\A^1}^\ast(X_\Del)\subset \bfCoh_{/\A^1}(X_\Del)$ for which the fiber at zero is the derived moduli stack $\bfCoh^{\mathsf{ss}, \,0}(X_\Dol)$ of semistable Higgs bundles on $X$ of degree zero.
As before, this yields:

\begin{theorem}[Weak Cat-HA version of the non-abelian Hodge correspondence]\label{thm:Cat-HA-naH}
	Let $X$ be a smooth projective complex curve.
	Then the stable $\infty$-category $\catCohb_{\C^\ast}( \bfCoh_{/\A^1}^\ast(X_\Del) )$ has a natural $\mathbb E_1$-monoidal structure.
	In addition, it is a module over $\catPerffilt \coloneqq \catPerf( [\A^1_\C / \G_m] )$ and we have monoidal functors:
	\begin{align}
		\Phi \colon \catCohb_{\C^\ast}( \bfCoh_{/\A^1}^\ast(X_\Del) ) & \otimes_{\catPerffilt} \catPerf_\C \longrightarrow \catCohb( \bfCoh(X_\dR) )\ ,\\[3pt]
		\Psi \colon \catCohb_{\C^\ast}( \bfCoh_{/\A^1}^\ast(X_\Del) ) & \otimes_{\catPerffilt} \catPerfgr \longrightarrow \catCohb_{\C^\ast}( \bfCoh^{\mathsf{ss}, \, 0}(X_\Dol) )\ ,
	\end{align}
	where $\catPerfgr \coloneqq \catPerf( \sfB \G_m )$.
\end{theorem}

\begin{conjecture}[Cat-HA version of the non-abelian Hodge correspondence]
The morphisms $\Phi$ and $\Psi$ are equivalences.	
\end{conjecture}

\subsubsection*{Decategorification}

Now, we investigate what algebras we obtain after decategorifying our Cat-HAs, i.e., after passing to the Grothendieck group. 
First, we introduce the finer invariant $\catCohb_{\mathsf{pro}}$, which is more adapted to the study of non-quasi-compact stacks.
Among its features, there is the fact that for every derived stack $Y$ there is a canonical equivalence (cf.\ Proposition \ref{prop:quasi_compact_G_truncation})
\begin{align}
	K( \catCohb_{\mathsf{pro}}(Y) ) \simeq K( \catCohb_{\mathsf{pro}}(\trunc{Y}) ) ,
\end{align}
a property that fails if we simply use $\catCohb$ instead of $\catCohb_{\mathsf{pro}}$.
The construction of $\catCohb_{\mathsf{pro}}$ relies on the machinery developed in \S\ref{sec:ind_quasi_compact}.

First, our construction provides a categorification of the K-theoretical Hall algebras of surfaces \cite{Zhao_Hall,KV_Hall} and the K-theoretical Hall algebras of Higgs sheaves on curves \cite{Minets_Hall, Sala_Schiffmann} (see \S\ref{ss:historical} for a review of these algebras).
\begin{theorem}
	Let $X$ be a smooth quasi-projective complex surface. There exists an algebra isomorphism between $\pi_0 K( \catCohb_{\mathsf{pro}}( \bfCoh_{\mathsf{prop}}^{\leqslant d}(X) ))$ and the K-theoretical Hall algebra of $X$ as defined in \cite{Zhao_Hall, KV_Hall}. Thus, the CoHA tensor structure on the stable $\infty$-category $\catCohb_{\mathsf{pro}}( \bfCoh^{\leqslant d}(X) )$ is a categorification of the latter. 
	
	Finally, if in addition $X$ is toric, similar results holds in the equivariant setting.
\end{theorem}

Now, let $X$ be a smooth projective complex curve. Our techniques provide a categorification of the Dolbeault K-theoretical Hall algebra of $X$ \cite{Minets_Hall,Sala_Schiffmann}:
\begin{proposition}
	Let $X$ be a smooth projective complex curve.  There exists an algebra isomorphism between $\pi_0 K( \catCohb_{\mathsf{pro},\,\C^\ast}( \bfCoh(X_\Dol)  )$ and the K-theoretical Hall algebra of Higgs sheaves on $X$ introduced in \cite{Sala_Schiffmann, Minets_Hall}. Thus, the CoHA tensor structure on the stable $\infty$-category $\catCohb_{\mathsf{pro},\,\C^\ast }( \bfCoh(X_\Dol) )$ is a categorification of the latter. 
\end{proposition}

One of the consequences of our construction is the categorification\footnote{A categorification of $\mathbf{U}_{q, t}^+(\ddot{\mathfrak{gl}}_1)$ has been also obtained by Negu\c{t}. In \cite{Negut_Hecke}, by means of (smooth) Hecke correspondences, he defined functors on the bounded derived category of the smooth moduli space of Gieseker-stable sheaves on a smooth projective surface, which after passing to $G$-theory, give rise to an action of the elliptic Hall algebra on the K-theory of such smooth moduli spaces.} of a positive nilpotent part of the quantum toroidal algebra $\mathbf{U}_{q, t}^+(\ddot{\mathfrak{gl}}_1)$. This is also known as the \textit{elliptic Hall algebra} of Burban and Schiffmann \cite{BS_elliptic}.
\begin{proposition}
	There exists a $\Z[q,t]$-algebra isomorphism 
	\begin{align}
		\pi_0 K( \catCohb_{\mathsf{pro},\,\C^\ast\times \C^\ast}( \bfCoh_{\mathsf{prop}}^{\leqslant 0}(\C^2) ) ) \simeq \mathbf{U}_{q, t}^+(\ddot{\mathfrak{gl}}_1)\ .	
	\end{align}
	Here, the $\C^\ast\times\C^\ast$-action on $\bfCoh_{\mathsf{prop}}^{\leqslant 0}(\C^2)$ is  induced by the torus action on $\C^2$.
\end{proposition}

In the Betti case, Davison \cite{Davison_character} defined the Betti cohomological Hall algebra of $X$ by using the Kontsevich-Soibelman CoHA formalism and a suitable choice of a quiver with potential. In \cite{Padurariu_CoHA}, the author generalizes such a formalism in the G-theory case. Thus, by combining the two one obtains a Betti K-theoretical Hall algebra. We expect that this is equivalent to our realization of the Betti K-theoretical Hall algebra. Finally, our approach defines the de Rham K-theoretical Hall algebra of $X$. 

By using the formalism of Borel--Moore homology of higher stacks developed in \cite{KV_Hall} and their construction of the Hall product via perfect obstruction theories, we obtain equivalent realizations of the COHA of a surface \cite{KV_Hall} and of the Dolbeault CoHA of a curve \cite{Sala_Schiffmann, Minets_Hall}. Moreover, we define the de Rham cohomological Hall algebra of a curve.

\subsection{DG-Coherent categorification}

At this stage, we would like to clarify what kind of ``categorification'' we provide and compare our approach to the other approaches to categorification known in the literature.

Let us start by recalling two well-known categorifications of the quantum group $\mathbf{U}_q(\mathfrak{n}_{\calQ})$, where $\mathfrak{n}_\calQ$ is the positive nilpotent part of a simply laced Kac--Moody algebra $\mathfrak{g}_\calQ$ and $\calQ$ is the corresponding quiver. The first one is provided by Lusztig in \cite{Lusztig_Canonical_bases, Lusztig_Perverse_sheaves}, and we shall call it the \textit{perverse categorification} of $\mathbf{U}_q(\mathfrak{n}_{\calQ})$. Denote by $\mathscr{R}ep(\calQ, \mathbf{d})$ the moduli stack of representations of the quiver $\calQ$ of dimension $\mathbf{d}$. Then, -- in modern terms -- Lusztig introduced a graded additive subcategory $\mathsf{C}( \mathscr{R}ep(\calQ, \mathbf{d}) )$ of the bounded derived category $\catDb( \mathscr{R}ep(\calQ, \mathbf{d}) )$ of constructible complexes whose split Grothendieck group is isomorphic to the $\mathbf{d}$-weight subspace of $\mathbf{U}_q(\mathfrak{n}_\calQ)$. By using a \textit{diagrammatic approach}, Khovanov--Lauda \cite{Lauda_Categorification, KL_Diagrammatic, KL_Diagrammatic_2, KL_Categorification} and Rouquier \cite{Rouquier_2-Kac--Moody} provided another categorification $\mathbf{U}_q(\mathfrak{n}_{\calQ})$, which is a 2-category: we call this the \textit{diagrammatic categorification} of $\mathbf{U}_q(\mathfrak{n}_{\calQ})$. In addition, they showed that $\mathbf{U}_q(\mathfrak{n}_\calQ)$ is the Grothendieck group of the monoidal category of all projective graded modules over the quiver-Hecke algebra $R$ of $\calQ$. A relation between these two categorifications of the same quantum group was established by Rouquier \cite{Rouquier_Quiver} and Varagnolo--Vasserot \cite{VV_Canonical}: they proved that there exists an equivalence of additive graded monoidal categories between $\oplus_{\mathbf{d}}\,\mathsf{C}( \mathscr{R}ep(\calQ, \mathbf{d}) )$ 
 and the category of all finitely generated graded projective $R$-modules.

Let $\calQ$ be the affine Dynkin quiver $A_1^{(1)}$. In \cite{SVV_Categorification}, the authors constructed another categorification of the quantum group  $\mathbf{U}_q(\mathfrak{n}_{\calQ})$, which they call the \textit{coherent categorification} of it. They showed that there exists a monoidal structure on the homotopy category of the $\C^\ast$-equivariant dg-category $\catCohb_{\C^\ast}( \bfRep(\Pi_{A_1}) )$, where $\Pi_{A_1}$ is the so-called preprojective algebra of the finite Dynkin quiver $A_1$ and $\bfRep(\Pi_{A_1})$ is a suitable derived enhancement of the moduli stack $\mathscr{R}ep(\Pi_{A_1})$ of finite-dimensional representations of $\Pi_{A_1}$. Here, there is a canonical $\C^\ast$-action on $\mathscr{R}ep(\Pi_{A_1})$ which lifts to the derived enhancement. By passing to the $G$-theory we obtain another realization of the algebra $\mathbf{U}_q(\mathfrak{n}_{\calQ})$. In \textit{loc.cit.} the authors started to investigate the relation between the perverse categorification and the coherent categorification of $\mathbf{U}_q(\mathfrak{n}_{\calQ})$ when $\calQ=A_1$.

Since in our paper we do not work with monoidal structures on triangulated categories, but rather with $\mathbb E_1$-monoidal structures on dg-categories, our construction provides the \textit{dg-coherent} categorification of the K-theoretical Hall algebras of surfaces \cite{Zhao_Hall,KV_Hall}, of the K-theoretical Hall algebras of Higgs sheaves on curves \cite{Minets_Hall, Sala_Schiffmann}, and of the de Rham and Betti K-theoretical Hall algebras of curves. At this point, one can wonder if there are perverse categorifications of these K-theoretical Hall algebras. 

Since in general there is no clear guess on what moduli stack to consider on the perverse side, it is not clear how to define the Lusztig's category.\footnote{Note that in the case treated in \cite{SVV_Categorification}, the moduli stack considered on the perverse side is $\mathscr{R}ep(A_1^{(1)})$, while on the coherent side is $\mathscr{R}ep(\Pi_{A_1})$. One evident relation between these two stacks is that the quiver appearing on the former stack is the affinization of the quiver on the latter stack.} The only known case so far is the perverse categorification of $\mathbf{U}_{q, t}^+(\ddot{\mathfrak{gl}}_1)$, i.e., of the K-theoretical Hall algebra of zero-dimensional sheaves on $\C^2$, due to Schiffmann. The latter is isomorphic to the K-theoretical Hall algebra of the preprojective algebra $\Pi_{\mathsf{one}\textrm{-}\mathsf{loop}}$ of the one-loop quiver -- see \S\ref{ss:historical}.

In  \cite{Schiffmann_Canonical}, Schiffmann constructed perverse categorifications of certain quantum loop and toroidal algebras, in particular of $\mathbf{U}_{q, t}^+(\ddot{\mathfrak{gl}}_1)$. In this case, he defined the Lusztig's category $\mathsf{C}( \bfCoh(X) )$ for the bounded derived category $\catDb( \bfCoh(X) )$ of constructible complexes on the moduli stack $\bfCoh(X)$ of coherent sheaves on an elliptic curve and he proved that the split Grothendieck group of $\mathsf{C}( \bfCoh(X) )$ is isomorphic to $\mathbf{U}_{q, t}^+(\ddot{\mathfrak{gl}}_1)$. Thus, for this quantum group we have both a perverse and a dg-coherent categorification. Although, it would be natural to ask what is the relation between them, it seems that the question is not well-posed since the former categorification comes from an additive category, while the latter from a dg-category. 

A viewpoint which can help us to correctly formulate a question about these two different categorifications is somehow provided by \cite{SV_Langlands}. In that paper, the authors pointed out how the two different realizations of $\mathbf{U}_{q, t}^+(\ddot{\mathfrak{gl}}_1)$ should be reinterpreted as a $G$-theory version of the geometric Langlands correspondence (see e.g.~\cite{AG_Langlands} and references therein):\footnote{One usually expects on the left-hand-side $\mathsf{D}\textrm{-}\mathsf{mod}(\bfBun(X, n))$, but this is indeed by definition $\catQCoh(\bfBun(X, n)_\dR)$.}
\begin{align}
	\catQCoh(\bfBun(X, n)_\dR) \simeq \mathsf{IndCoh}_{\mathsf{Nilp}_{\mathsf{Glob}}}(\bfBun(X_\dR, n))\ ,
\end{align}
where the Lusztig's category arises from the left-hand side, while a K-theoretical Hall algebra arises from the right-hand side. Here, $X$ is a smooth projective complex curve and $n$ a non-negative integer.\footnote{One may notice that the K-theoretical Hall algebra considered in \cite{SV_Langlands} is the one associated with $\Pi_{\mathsf{one}\textrm{-}\mathsf{loop}}$ while our construction provides a de Rham K-theoretical Hall algebra of $X$: the relation between them should arise from the observation that the moduli stack of finite-dimensional representations of $\Pi_{\mathsf{one}\textrm{-}\mathsf{loop}}$ is some sort of ``formal neighborhood'' of the trivial D-module in $\bfBun(X_\dR)$.}

By interpreting \cite{SV_Langlands} as a decategorified version of what we are looking for, we may speculate the following:
\begin{conjecture}
	Let $X$ be a smooth projective complex curve. Then there exist an $\mathbb E_1$-monoidal structure on the dg-category $\catCohb( \bfBun(X)_\dR )$ and an $\mathbb E_1$-monoidal equivalence between $\catCohb( \bfBun(X)_\dR )$ and the categorified Hall algebra\footnote{Or a version of it in which the complexes have fixed singular supports - see \cite{AG_Langlands} for the definition of singular support in this context.} $\catCohb( \bfCoh(X_\dR) )$.
\end{conjecture}
It follows that the right ``dg-enhancement'' of Lusztig's construction should be $\catCohb( \bfBun(X)_\dR )$. In addition, one should expect that, when $X$ is an elliptic curve, by passing to the $G$-theory one recovers $\mathbf{U}_{q, t}^+(\ddot{\mathfrak{gl}}_1)$. 

Finally, one may wonder if there is a diagrammatic description of our categorified Hall algebras in the spirit of Khovanov--Lauda and Rouquier. Let $Y$ be either a smooth proper complex scheme $S$ of dimension $\leq 2$ or one of the Simpson's shapes of a smooth projective complex curve $X$. Then $\bfCoh(Y)$ admits a stratification
\begin{align}
	\bfCoh(Y)=\bigsqcup_{\mathbf{v}\in \Lambda}\, \bfCoh(Y, \mathbf{v})
\end{align}
such that the Hall product is graded with respect to it, where $\Lambda$ is the numerical Grothendieck group of $S$ in the first case and of $X$ in the second case. 

Now, we define\footnote{We thank Andrea Appel for helping us spelling out the description of $\calU$.} the following $(\infty, 2)$-category $\calU$: it is the subcategory inside the $(\infty, 2)$-category of dg-categories such that
\begin{itemize}
	\item its objects are the dg-categories $\catCohb( \bfCoh(Y, \mathbf{v}) )$,
	\item the 1-morphisms between $\catCohb( \bfCoh(Y, \mathbf{v}) )$ and $\catCohb( \bfCoh(Y, \mathbf{v}') )$ are the functors of the form
	\begin{align}
	-\ostar E \ ,
	\end{align} 
	for $E\in \catCohb( \bfCoh(Y, \mathbf{v}'-\mathbf{v}) )$. Here, $\ostar$ denotes the Hall tensor product.
\end{itemize}
The study of 2-morphisms in $\calU$ should lead to an analogue of KLR algebras in this setting, which will be investigated in a future work.

\subsection{Historical background on CoHAs} \label{ss:historical}

For completeness, we include a review of the literature around two-dimensional CoHAs.

The first instances\footnote{To the best of the authors' knowledge, the first circle of ideas around two-dimensional CoHAs can be found in an unpublished manuscript by Grojnowski \cite{Grojnowski_affinizing}.} of two-dimensional CoHAs can be traced back to the works of Schiffmann and Vasserot \cite{SV_elliptic, SV_Langlands}.
Seeking for a ``geometric Langlands dual algebra'' of the (classical) Hall algebra of a curve\footnote{By the (classical) Hall algebra of a curve we mean the Hall algebra associated with the abelian category of coherent sheaves on a smooth projective curve defined over a finite field. As explained in \cite{Schiffmann_Lectures_II}, conjecturally this algebra can be realized by using the Lusztig's category (such a conjecture is true in the genus zero and one case, for example).}, the authors were lead to introduce a convolution algebra structure on the (equivariant) $G_0$-theory of the cotangent stack $\sfT^\ast \mathscr{R}ep( \calQ_g )$.
Here $\mathscr{R}ep( \calQ_g )$ is the stack of finite-dimensional representations of the quiver $\calQ_g$ with one vertex and $g$ loops.
When $g=1$, the corresponding associative algebra is isomorphic to a positive part of the elliptic Hall algebra.
A study of the representation theory of the elliptic Hall algebra by using its CoHA description was initiated in \cite{SV_elliptic} and pursued by Negu\c{t} \cite{Negut_qAGT} in connection with gauge theory and deformed vertex algebras. 

The extension of this construction to any quiver and, at the same time, to Borel--Moore homology theory and more generally to any oriented Borel--Moore homology theory was shown e.g. in \cite{YZ_CoHA}. Note that $\sfT^\ast \mathscr{R}ep(\calQ)$ is equivalent to the stack of finite-dimensional representations of the \textit{preprojective algebra $\Pi_\calQ$ of $\calQ$}. For this reason, sometimes this CoHA is called the \textit{CoHA of the preprojective algebra of $\calQ$}.

In the Borel--Moore homology case, Schiffmann and Vasserot gave a characterization of the generators of the CoHA of the preprojective algebra of $\calQ$ in \cite{SV_generators}, while a relation with the (Maulik-Okounkov) Yangian was established in \cite{SV_Yangians, Davison_DT, YZ_Yangian}. Again, when $\calQ=\calQ_1$, a connection between the corresponding two-dimensional CoHA and vertex algebras was provided in \cite{SV_Cherednik, Negut_AGT} (see also \cite{RSYZ}).

In \cite{KS_Hall}, Kontsevich and Soibelman introduced another CoHA, in order to provide a mathematical definition of Harvey and Moore's algebra of BPS states \cite{Harvey_Moore}. It goes under the name of \textit{three-dimensional} CoHA since it is associated with Calabi--Yau categories of global dimension three (such as the category of representations of the Jacobi algebra of a quiver with potential, the category of coherent sheaves on a CY 3-fold, etc.). As shown by Davison in \cite[Appendix]{RS_Hall} (see also \cite{YZ_2Hall}), using a dimensional reduction argument, the CoHA of the preprojective algebra of a quiver described above can be realized as a Kontsevich--Soibelman one.

For certain choices of the quiver $\calQ$, the cotangent stack $\sfT^\ast \mathscr{R}ep(\calQ)$ is a stack parameterizing coherent sheaves on a surface. Thus the corresponding algebra can be seen as an example of a CoHA associated to a surface. This is the case when the quiver is the one-loop quiver $\calQ_1$: indeed, $\sfT^\ast \mathscr{R}ep(\calQ_1)$ coincides with the stack $\mathscr{C}oh_0(\C^2)$ parameterizing zero-dimensional sheaves on the complex plane $\C^2$. In particular, the elliptic Hall algebra can be seen as an algebra attached to zero-dimensional sheaves on $\C^2$.

Another example of two-dimensional CoHA is the \textit{Dolbeault CoHA} of a curve. Let $X$ be a smooth projective curve and let $\mathscr{H}iggs(X)$ be the stack\footnote{Note that the truncation of the derived stack $\bfCoh(X_\Dol)$ is isomorphic to $\mathscr{H}iggs(X)$.} of Higgs sheaves on $X$. Then the Borel--Moore homology of the stack $\mathscr{H}iggs(X)$ of Higgs sheaves on $X$ is endowed with the structure of a convolution algebra. Such an algebra has been introduced by the second-named author and Schiffmann in \cite{Sala_Schiffmann}. In \cite{Minets_Hall}, independently Minets has introduced the Dolbeault CoHA in the rank zero case. Thanks to the Beauville--Narasimhan--Ramanan correspondence, the Dolbeault CoHA can be interpreted as the CoHA of torsion sheaves on $\sfT^\ast X$ such that their support is proper over $X$. In particular, Minets' algebra is an algebra attached to zero-dimensional sheaves on $\sfT^\ast X$. Such an algebra coincides with Negu\c{t}'s \textit{shuffle algebra} \cite{Negut_shuffle} of a surface $S$ when $S=\sfT^\ast X$. 

Negu\c{t}'s algebra of a smooth surface $S$ is defined by means of \textit{Hecke correspondences} depending on zero-dimensional sheaves on $S$, and its construction comes from a generalization of the realization of the elliptic Hall algebra in \cite{SV_elliptic} via Hecke correspondences. Zhao \cite{Zhao_Hall} constructed the cohomological Hall algebra of the moduli stack of zero-dimensional sheaves on a smooth surface $S$ and fully established the relation between this CoHA and Negu\c{t}'s algebra of $S$. A stronger, independently obtained result is due to  Kapranov and Vasserot \cite{KV_Hall}, who defined the CoHA associated to a category of coherent sheaves on a smooth surface $S$ with proper support of a fixed dimension.

\subsection{Outline}

In \S\ref{s:coh} we introduce our derived enhancement of the classical stack of coherent sheaves on a smooth complex scheme. We also define derived moduli stacks of coherent sheaves on the Betti, de Rham, and Dolbeault shapes of a smooth scheme. In \S\ref{s:coh_ext} we introduce the derived enhancement of the classical stack of extensions of coherent sheaves on both a smooth complex scheme and on a Simpson's shape of a smooth complex scheme. In addition, we define the convolution diagram \eqref{eq:convolution-derived} and provide the tor-amplitude estimates for the map $p$. \S\ref{s:categorifiedHall} is devoted to the construction of the categorified Hall algebra associated with the moduli stack of coherent sheaves on either a smooth scheme or a Simpson's shape of a smooth scheme: in \S\ref{ss:convolution_structure} we endow such a stack of the structure of a 2-Segal space à la Dyckerhoff--Kapranov, while in \S\ref{ss:categorification} by applying the functor $\catCohb_{\mathsf{pro}}$, we obtain one of our main results, i.e., an $\mathbb E_1$-monoidal structure on $\catCohb_{\mathsf{pro}}( \bfCoh(Y) )$ when $Y$ is either a smooth curve or surface, or a Simpson's shape of a smooth curve; finally, \S\ref{ss:equivariant} is devoted to the equivariant case of the construction of the categorified Hall algebra. In \S\ref{s:cohas}, we show how our approach provides equivalent realizations of the known K-theoretical Hall algebras of surfaces and of Higgs sheaves on a curve. In \S\ref{s:Hodge_filtration} and \S\ref{s:RH} we discuss \textit{Cat-HA} versions of the non-abelian Hodge correspondence and of the Riemann--Hilbert correspondence, respectively. In particular, in \S\ref{s:RH} we develop the construction of the categorified Hall algebra in the analytic setting and we compare the two resulting categorified Hall algebras. Finally, Appendix~\ref{sec:ind_quasi_compact} is devoted to the study of the G-theory of non-quasi-compact stacks and the construction of $\catCohb_{\mathsf{pro}}$.

\subsection*{Acknowledgements}

First, we would like to thank Mattia Talpo for suggesting us to discuss about the use of derived algebraic geometry in the theory of cohomological Hall algebras. This was the starting point of our collaboration. 

Part of this work was developed while the second-named author was visiting the Université de Strasbourg and was completed while the first-named author was visiting Kavli IPMU. We are grateful to both institutions for their hospitality and wonderful working conditions. We would like to thank Andrea Appel, Kevin Costello, Ben Davison, Adeel Khan, Olivier Schiffmann, Philippe Eyssidieux, Tony Pantev and Bertrand To\"en for enlightening conversations.
We especially thank Andrea Gagna and Ivan Di Liberti for discussions about the correct $2$-categorical notion of the $\Ind$-construction.
Finally, we thank the anonymous referee for useful suggestions and comments.

The results of the present paper have been presented by the second-named author at the Workshop ``Cohomological Hall algebras in Mathematics and Physics`` (Perimeter Institute for Theoretical Physics, Canada; February 2019). The second-named author is grateful to the organizers of this event for the invitation to speak.

\subsection{Notations and convention}\label{ss:notation}

In this paper we freely use the language of $\infty$-categories. Although the discussion is often independent of the chosen model for $\infty$-categories, whenever needed we identify them with quasi-categories and refer to \cite{HTT} for the necessary foundational material.

The notations $\cS$ and $\Cat_\infty$ are reserved for the $\infty$-categories of spaces and of $\infty$-categories, respectively. If $\cC \in \Cat_\infty$ we denote by $\cC^\simeq$ the maximal $\infty$-groupoid contained in $\cC$. We let $\Cat_\infty^{\mathsf{st}}$ denote the $\infty$-category of stable $\infty$-categories with exact functors between them.
We also let $\PrL$ denote the $\infty$-category of presentable $\infty$-categories with left adjoints between them.
We let $\Pr^{\mathsf L, \omega}$ the $\infty$-category of compactly generated presentable $\infty$-categories with morphisms given by left adjoints that commute with compact objects.
Similarly, we let $\PrL_{\mathsf{st}}$ (resp.\ $\Pr^{\mathsf L, \omega}_{\mathsf{st}}$) denote the $\infty$-categories of stably presentable $\infty$-categories with left adjoints between them (resp.\ left adjoints that commute with compact objects).
Finally, we set
\begin{align}
	\Cat_\infty^{\mathsf{st}, \, \otimes} \coloneqq \CAlg( \Cat_\infty^{\mathsf{st}} ) \ , \quad \mathcal P \mathsf r_{\mathsf{st}}^{\mathsf L,\, \otimes} \coloneqq \CAlg( \PrL_{\mathsf{st}} ) \ .
\end{align}
Given an $\infty$-category $\cC$ we denote by $\PSh(\cC)$ the $\infty$-category of $\cS$-valued presheaves.
We follow the conventions introduced in \cite[\S2.4]{Porta_Yu_Higher_analytic_stacks_2014} for $\infty$-categories of sheaves on an $\infty$-site.

Since we only work over the field of complex numbers $\C$, we reserve the notation $\CAlg$ for the $\infty$-category of simplicial commutative rings over the field of complex numbers $\C$. We often refer to objects in $\CAlg$ simply as \textit{derived commutative rings}.
We denote its opposite by $\dAff$, and we refer to it as the $\infty$-category of \textit{affine derived schemes}. 

In \cite[Definition~1.2.3.1]{Lurie_SAG} it is shown that the étale topology defines a Grothendieck topology on $\dAff$. We denote by $\dSt\coloneqq \Sh(\dAff, \tau_{\text{ét}})^\wedge$ the hypercompletion of the $\infty$-topos of sheaves on this site. We refer to this $\infty$-category as the \textit{$\infty$-category of derived stacks}. For the notion of derived \textit{geometric} stacks, we refer to \cite[Definition~2.8]{Porta_Yu_Higher_analytic_stacks_2014}.

Let $A \in \CAlg$ be a derived commutative ring.
We let $A \Mod$ denote the stable $\infty$-category of $A$-modules,  equipped with its canonical symmetric monoidal structure provided by \cite[Theorem~3.3.3.9]{Lurie_Higher_algebra}.
Furthermore, we equip it with the canonical \textit{t}-structure whose connective part is its smallest full subcategory closed under colimits and extensions and containing $A$.
Such a \textit{t}-structure exists in virtue of \cite[Proposition 1.4.4.11]{Lurie_Higher_algebra}.
Notice that there is a canonical equivalence of abelian categories $A \Mod^\heartsuit \simeq \pi_0(A) \Mod^\heartsuit$. 

We say that an $A$-module $M \in A \Mod$ is \textit{perfect} if it is a compact object in $A \Mod$. 	
We denote by $\catPerf(A)$ the full subcategory of $A \Mod$ spanned by perfect complexes\footnote{It is shown in \cite[Proposition~7.2.4.2]{Lurie_Higher_algebra} that an $\catPerf(A)$ coincides with the smallest full stable subcategory of $A \Mod$ closed under retracts and containing $A$.	In particular, $\catPerf(A)$ is a stable $\infty$-category which is furthermore idempotent complete.}. On the other hand, we say that an $A$-module $M \in A \Mod$ is \textit{almost perfect}\footnote{Suppose that $A$ is almost of finite presentation over $\C$. In other words, suppose that $\pi_0(A)$ is of finite presentation in the sense of classical commutative algebra and that each $\pi_i(A)$ is coherent over $\pi_0(A)$. Then \cite[Proposition~7.2.4.17]{Lurie_Higher_algebra} shows that an $A$-module $M$ is almost perfect if and only if $\pi_i(M) = 0$ for $i \ll 0$ and each $\pi_i(M)$ is coherent over $\pi_0(A)$.}  if $\pi_i(M) = 0$ for $i \ll 0$ and for every $n \in \Z$ the object $\tau^{\le n}(M)$ is compact in $A \Mod^{\le n}$. We denote by $\catAPerf(A)$ the full subcategory of $A \Mod$ spanned by sheaves of almost perfect modules.

Given a morphism $f \colon A \to B$ in $\CAlg$ we obtain an $\infty$-functor $f^\ast \colon A \Mod \longrightarrow B \Mod$, which preserves (almost) perfect modules.
We can assemble these data into an $\infty$-functor
\begin{align}
	\catQCoh \colon \dAff\op \longrightarrow \monPrLst \ . 
\end{align}
Since the functor $f^\ast$ preserves (almost) perfect modules, we obtain well defined subfunctors
\begin{align}
	\catPerf\, , \ \catAPerf \colon \dAff\op \longrightarrow \Cat_\infty^{\mathrm{st}, \, \otimes} \ . 
\end{align}

Given a derived stack $X \in \dSt$, we denote by $\catQCoh(X)$, $\catAPerf(X)$ and $\catPerf(X)$ the stable $\infty$-categories of quasi coherent, almost perfect, and perfect complexes respectively. One has
\begin{align}
	\catQCoh(X) &\simeq \varprojlim_{\Spec(A) \to X} \catQCoh(\Spec(A)) \ , \quad \catAPerf(X) \simeq \varprojlim_{\Spec(A) \to X} \catAPerf(\Spec(A))\ , \quad\text{and}\\ 
	\catPerf(X) &\simeq \varprojlim_{\Spec(A) \to X} \catPerf(\Spec(A)) \ .
\end{align}
The $\infty$-category $\catQCoh(X)$ is presentable. In particular, using \cite[Proposition~1.4.4.11]{Lurie_Higher_algebra} we can endow $\catQCoh(X)$ with a canonical \textit{t}-structure. 

Let $f \colon X \to Y$ be a morphism in $\dSt$. We say that $f$ is \textit{flat} if the induced functor $	f^\ast \colon \catQCoh(Y) \to \catQCoh(X)$ is $t$-exact.

Let $X\in\dSt$. We denote by $\catCoh(X)$ the full subcategory of $\scrO_X\Mod$ spanned by $\calF\in \scrO_X\Mod$  for which there exists an atlas $\{f_i\colon U_i\to X\}_{i\in I}$ such that for every $i\in I$, $n\in \Z$, the $\scrO_{U_i}$-modules $\pi_n(f^\ast_i \calF)$ are coherent sheaves. We denote by $\catCoh^\heartsuit(X)$ (resp.\ $\catCohb(X)$, $\catCoh^{+}(X)$, and $\catCoh^{-}(X)$) the full subcategory of $\catCoh(X)$ spanned by objects cohomologically concentrated in degree 0 (resp. locally cohomologically bounded, bounded below, bounded above).

\bigskip\section{Derived moduli stacks of coherent sheaves}\label{s:coh}

Our goal in this section is to define derived enhancements of the classical stacks of coherent sheaves on a proper complex algebraic variety $X$, of Higgs sheaves on $X$, of vector bundles with flat connections on $X$, and of finite-dimensional representations of the fundamental group $\pi_1(X)$ of $X$.

\subsection{Relative flatness}

We start by defining the objects that this derived enhancement will classify.

\begin{definition} \label{defin:coherent_perfect_complex}
	Let $f \colon X \to S$ be a morphism of derived stacks.
	We say that a quasi-coherent sheaf $\calF \in \catQCoh(X)$ has \textit{tor-amplitude within $[a,b]$ relative to $S$} (resp.\ \textit{tor-amplitude $\le n$ relative to $S$}) if for every $\calG \in \catQCoh^\heartsuit(S)$ one has
	\begin{align}  
		\pi_i( \calF \otimes f^\ast \calG ) = 0 \quad i \notin [a,b] \quad \text{(resp.\ } i \notin [0,n] \text{)}  \ . 
	\end{align} 
	We let $\catQCoh_S^{\le n}(X)$ (resp.\ $\catAPerf_S^{\le n}(X)$) denote the full subcategory of $\catQCoh(X)$ spanned by those quasi-coherent sheaves (resp.\ sheaves of almost perfect modules) $\calF$ on $X$ having tor-amplitude $\le n$ relative to $S$.
	We write
	\begin{align} 
		\catCoh_S(X) \coloneqq \catAPerf_S^{\le 0}(X) \ , 
	\end{align}
	and we refer to $\catCoh_S(X)$ as the $\infty$-category of \textit{flat families of coherent sheaves on $X$ relative to $S$}.
\end{definition}

\begin{remark} \label{rem:Coh_not_stable}
	The $\infty$-category $\catCoh_S( X \times S )$ is \textit{not} stable.
	This is because in general the cofiber of a map between sheaves of almost perfect modules in tor-amplitude $\le 0$ is only in tor-amplitude $[1,0]$.
	When $S$ is underived, the cofiber sequences $\calF' \to \calF \to \calF''$ in $\catAPerf(X \times S)$ whose three terms are all coherent correspond to short exact sequences of coherent sheaves.
	In particular, the map $\calF' \to \calF$ is a monomorphism and the map $\calF \to \calF''$ is an epimorphism.
\end{remark}

\begin{remark} \label{rem:tor_amplitude_underived}
	Let $A \in \CAlg_{\C}$ be a derived commutative ring and let $M \in A \Mod$.
	Then $M$ has tor-amplitude $\le n$ if and only if $M \otimes_A \pi_0(A)$ has tor-amplitude $\le n$.
		In particular, if $A$ is underived and $M \in A \Mod^\heartsuit$, then $M$ has tor-amplitude $\le 0$ if and only if $M$ is flat in the sense of the usual commutative algebra.
\end{remark}

We start by studying the functoriality of $\catCoh_S(Y)$ in $S$:

\begin{lemma} \label{lem:base_change_tor_amplitude}
	Let
		\begin{align}
		\begin{tikzcd}[ampersand replacement = \&]
		X_T \arrow{r}{g'} \arrow{d}{f'} \& X \arrow{d}{f} \\
		T \arrow{r}{g} \& S
		\end{tikzcd}
		\end{align}
	be a pullback square in $\dSt$.
	Assume that $T$ and $S$ are affine derived schemes.
	If $\calF \in \catQCoh(X)$ has tor-amplitude $[a,b]$ relative $S$, then $g^{\prime \ast}(\calF) \in \catQCoh(X_T)$ has tor-amplitude within $[a,b]$ relative to $T$.
\end{lemma}

\begin{proof}
	Let $\calF \in \catQCoh(X)$ be a quasi-coherent sheaf of tor-amplitude within $[a,b]$ relative to $S$ and let $\calG \in \catQCoh^\heartsuit(T)$.
	Since $g$ is representable by affine schemes, so does $g'$.
	Therefore, \cite[Proposition~\ref*{shapes-prop:representable_by_affine}]{Porta_Sala_Shapes} implies that $g'_\ast$ is $t$-exact and conservative.
	Therefore, $g^{\prime \ast}(\calF) \otimes f^{\prime \ast}(\calG)$ is in cohomological amplitude $[a,b]$ if and only if $g'_\ast( g^{\prime \ast}(\calF) \otimes f^{\prime \ast}(\calG) )$ is.
	Combining \cite[Propositions~\ref*{shapes-prop:representable_by_affine}-(1) and \ref*{shapes-prop:categorically_quasi_compact}-(2)]{Porta_Sala_Shapes}, we see that
	\begin{align}
		g'_\ast( g^{\prime \ast}(\calF) \otimes f^{\prime \ast}(\calG) ) \simeq \calF \otimes g'_\ast( f^{\prime \ast}( \calG ) ) \ , 
	\end{align}
	and using \cite[Proposition~\ref*{shapes-prop:representable_by_affine}-(2)]{Porta_Sala_Shapes} we can rewrite the last term as
	\begin{align}
		\calF \otimes g'_\ast( f^{\prime \ast}( \calG ) ) \simeq \calF \otimes f^\ast( g_\ast(\calG) ) \ .
	\end{align}
	Since $g_\ast$ is $t$-exact, we have $g_\ast(\calG) \in \catQCoh^\heartsuit(S)$. The conclusion now follows from the fact that $\calF$ has tor-amplitude within $[a,b]$.
\end{proof}

\begin{construction}\label{ss:construction}
	Let $X\in \dSt$ and consider the derived stack
	\begin{align}
		\bfAPerf(X) \colon \dAff\op \longrightarrow \cS
	\end{align}
	sending an affine derived scheme $S \in \dAff$ to the maximal $\infty$-groupoid $\catAPerf(X \times S)^\simeq$ contained in the stable $\infty$-category $\catAPerf(X \times S)$ of almost perfect modules over $X \times S$. 
	
	Lemma~\ref{lem:base_change_tor_amplitude} implies that the assignment sending $S \in \dAff$ to the full subspace $\catCoh_S(X \times S)^\simeq$ of $\catAPerf(X \times S)^\simeq$ spanned by flat families of coherent sheaves on $X$ relative to $S$ defines a substack
	\begin{align}
		\bfCoh(X)\colon \dAff\op \longrightarrow \cS
	\end{align}
	of $\bfAPerf(X)$.
	We refer to $\bfCoh(X)$ as the \textit{derived stack of coherent sheaves on $X$}.
	\end{construction}

In this paper we are mostly interested in this construction when $X$ is a scheme or one of its Simpson's shapes $X_\B$, $X_\dR$ or $X_\Dol$.
We provide the following useful criterion to recognize coherent sheaves:

\begin{lemma} \label{lem:flat_effective_epimorphism}
	Let $f \colon X \to S$ be a morphism in $\dSt$.
	Assume that there exists a flat effective epimorphism $u \colon U \to X$.
	Then $\calF \in \catQCoh(X)$ has tor-amplitude within $[a,b]$ relative to $S$ if and only if $u^\ast(\calF)$ has tor-amplitude within $[a,b]$ relative to $S$.
\end{lemma}

\begin{proof}
	Let $\calG \in \catQCoh^\heartsuit(S)$.
	Then since $u$ is a flat effective epimorphism, we see that the pullback functor
	\begin{align}  
		u^\ast \colon \catQCoh(X) \longrightarrow \catQCoh(U) 
	\end{align} 
	is $t$-exact and conservative.
	Therefore $\pi_i( \calF \otimes f^\ast \calG ) \simeq 0$ if and only if
	\begin{align}  
		u^\ast( \pi_i(\calF \otimes f^\ast \calG) ) \simeq \pi_i( u^\ast(\calF) \otimes u^\ast f^\ast \calG ) \simeq 0 \ . 
	\end{align} 
	The conclusion follows.
\end{proof}

As a consequence, we see that, for morphisms of geometric derived stacks, the notion of tor-amplitude within $[a,b]$ relative to the base introduced in Definition \ref{defin:coherent_perfect_complex} coincides with the most natural one:

\begin{lemma} \label{lem:relative_tor_amplitude_geometric_case}
	Let $X$ be a geometric derived stack, let $S = \Spec(A) \in \dAff$ and let $f \colon X \to S$ be a morphism in $\dSt$.
	Then $\calF \in \catQCoh(X)$ has tor-amplitude within $[a,b]$ relative to $S$ if and only if there exists a smooth affine covering $\{u_i \colon U_i = \Spec(B_i) \to X\}$ such that $f_{i\ast} u_i^\ast(\calF)$ has tor-amplitude within $[a,b]$ as $A$-module,\footnote{Cf.\ \cite[Definition~7.2.4.21]{Lurie_Higher_algebra} for the definition of tor-amplitude within $[a,b]$.} where $f_i \coloneqq f \circ u_i$.
\end{lemma}

\begin{proof}
	Applying Lemma \ref{lem:flat_effective_epimorphism}, we can restrict ourselves to the case where $X = \Spec(B)$ is affine.
	In this case, we first observe that $f_\ast \colon \catQCoh(X) \to \catQCoh(S)$ is $t$-exact and conservative.
	Therefore, $\pi_i( \calF \otimes f^\ast \calG ) \simeq 0$ if and only if $\pi_i( f_\ast( \calF \otimes f^\ast \calG ) ) \simeq 0$.
	The projection formula yields
	\begin{align}  
		f_\ast( \calF \otimes f^\ast \calG ) \simeq f_\ast(\calF) \otimes \calG \ , 
	\end{align} 
	and therefore the conclusion follows.
\end{proof}

\subsection{Deformation theory of coherent sheaves}

Let $X$ be a derived stack.
We study the deformation theory of the stack $\bfCoh(X)$.
Since we are also interested in the case where $X$ is one of Simpson's shapes, we first recall the following definition:

\begin{defin}
	A morphism $u \colon U \to X$ in $\dSt$ is a \textit{flat effective epimorphism} if:
	\begin{enumerate}\itemsep=0.2cm
		\item it is an effective epimorphism, i.e.\ the map $\pi_0(U) \to \pi_0(X)$ is an epimorphism of discrete sheaves;
		\item it is flat, i.e.\ the pullback functor $u^\ast \colon \catQCoh(X) \to \catQCoh(U)$ is $t$-exact.
	\end{enumerate}
\end{defin}

We have the following stability property:

\begin{lemma} \label{lem:universally_flat_effective_epi}
	Let $X \to S$ be a morphism in $\dSt$ and let $U \to X$ be a flat effective epimorphism.
	If $T \to S$ is representable by affine derived schemes, then $U \times_S T \to X \times_S T$ is a flat effective epimorphism.
\end{lemma}

\begin{proof}
	Combine \cite[Proposition~6.2.3.5]{HTT} and \cite[Proposition~\ref*{shapes-prop:local_tor_amplitude}-(2)]{Porta_Sala_Shapes}.
\end{proof}

\begin{example} \label{eg:applications}
	\hfill
	\begin{enumerate}\itemsep0.2cm
		\item If $X$ is a geometric derived stack and $u \colon U \to X$ is a smooth atlas, then $u$ is a flat effective epimorphism.
		
		\item Let $X$ be a connected $\C$-scheme of finite type and let $x \colon \Spec(\C) \to X$ be a closed point.
		Then the induced map $\Spec(\C) \to X_\B$ is a flat effective epimorphism.
		See \cite[Proposition~\ref*{shapes-prop:Betti_properties}-(\ref*{shapes-prop:Betti_properties:flat_atlas})]{Porta_Sala_Shapes}.
		
		\item Let $X$ be a smooth $\C$-scheme. The natural map $\lambda_X \colon X \to X_\dR$ is a flat effective epimorphism.
		See \cite[Proposition~\ref*{shapes-prop:de_Rham_properties}-(\ref*{shapes-prop:de_Rham_properties:effective_epi}) and -(\ref*{shapes-prop:de_Rham_properties:flat_atlas})]{Porta_Sala_Shapes}.
		
		\item Let $X$ be a geometric derived stack. The natural map $\kappa_X \colon X \to X_\Dol$ is a flat effective epimorphism.
		See  \cite[Lemma~\ref*{shapes-lem:effective-Dol}]{Porta_Sala_Shapes}.
	\end{enumerate}
\end{example}

\begin{lemma} \label{lem:universally_flat_atlas}
	Let $u \colon U \to X$ be a flat effective epimorphism.
	Then the square
	\begin{align}
		\begin{tikzcd}[ampersand replacement = \&]
			\bfCoh(X) \arrow{r} \arrow{d} \& \bfCoh(U) \arrow{d} \\
			\bfAPerf(X) \arrow{r} \& \bfAPerf(U)
		\end{tikzcd}
	\end{align}
	is a pullback square.
\end{lemma}

\begin{proof}
	We have to prove that for every $S \in \dAff$, a sheaf of almost perfect modules $\calF \in \catAPerf(X \times S)$ is flat relative to $S$ if and only if its pullback to $U \times S$ has the same property.
	Since $u \colon U \to X$ is a flat effective epimorphism, so is $S \times U \to S \times X$ by Lemma~\ref{lem:universally_flat_effective_epi}.
	At this point, the conclusion follows from Lemma~\ref{lem:flat_effective_epimorphism}.
\end{proof}

Since Example~\ref{eg:applications} contains our main applications, we will always work under the assumption that there exists a flat effective epimorphism $U \to X$, where $U$ is a geometric derived stack locally almost of finite type.
The above lemma allows us therefore to carry out the main verifications in the case where $X$ itself is geometric and locally almost of finite type.

We start with infinitesimal cohesiveness and nilcompleteness.
Recall that $\bfAPerf(X)$ is infinitesimally cohesive and nilcomplete for every derived stack $X \in \dSt$:

\begin{lemma} \label{lem:aperf_infinitesimally_cohesive_nilcomplete}
	Let $X \in \dSt$ be a derived stack.
	Then $\bfAPerf(X)$ is infinitesimally cohesive and nilcomplete.
\end{lemma}

\begin{proof}
	Combine Propositions~\ref*{shapes-prop:inf_cohesive}-(\ref*{shapes-prop:inf_cohesive:mapping_stack}) and \ref*{shapes-prop:nilcomplete}-(\ref*{shapes-prop:nilcomplete:mapping_stack}) with Theorem~\ref*{shapes-thm:perf_inf_cohesive_nilcomplete} in \cite{Porta_Sala_Shapes}.
\end{proof}

In virtue of the above lemma, our task is reduced to proving that the map $\bfCoh(X) \to \bfAPerf(X)$ is infinitesimally cohesive and nilcomplete.
Thanks to Lemma~\ref{lem:universally_flat_atlas}, the essential case is when $X$ is affine:

\begin{lemma} \label{lem:coh_inf_cohesive_nilcomplete}
	Let $X\in \dAff$ be an affine derived scheme.
	Then the morphism
	\begin{align}
		\bfCoh(X) \longrightarrow \bfAPerf(X) 
	\end{align}
	is infinitesimally cohesive and nilcomplete.
	As a consequence, $\bfCoh(X)$ is infinitesimally cohesive and nilcomplete.
\end{lemma}

\begin{proof}
	We start dealing with infinitesimal cohesiveness.
	Let $S = \Spec(A)$ be an affine derived scheme and let $M \in \catQCoh(S)^{\ge 1}$ be a quasi-coherent complex.
	Let $S[M] \coloneqq \Spec(A \oplus M)$ and let $d \colon S[M] \to S$ be a derivation.
	Finally, let $S_d[M[-1]]$ be the pushout
	\begin{align}
		\begin{tikzcd}[ampersand replacement = \&]
		S[M] \arrow{r}{d} \arrow{d}{d_0} \& S \arrow{d}{f_0} \\ S \arrow{r}{f} \& S_d[M[-1]] 
		\end{tikzcd}\ ,
	\end{align}
	where $d_0$ denotes the zero derivation.
	Since the maximal $\infty$-groupoid functor $(-)^\simeq \colon \Cat_\infty \to \cS$ commutes with limits, it is enough to prove that the square
	\begin{align}
		\begin{tikzcd}[ampersand replacement = \&]
			\catCoh_{S_d[M[-1]]}( X \times S_d[M[-1]] ) \arrow{r} \arrow{d} \& \catCoh_S( X \times S ) \times_{\catCoh_{S[M]}( X \times S[M] )} \catCoh_S( X \times S ) \arrow{d} \\
			\catAPerf(X \times S_d[M[-1]]) \arrow{r} \& \catAPerf( X \times S ) \times_{\catAPerf(X \times S[M])} \catAPerf( X \times S )
		\end{tikzcd}
	\end{align}
	is a pullback.
	Using \cite[Theorem~16.2.0.1 and Proposition~16.2.3.1(6)]{Lurie_SAG}, we see that the bottom horizontal map is an equivalence.
	As the vertical arrows are fully faithful, we deduce that the top horizontal morphism is fully faithful as well.
	It is therefore enough to check that the top horizontal functor is essentially surjective.
	Let $\varphi, \varphi_0 \colon X\times S \to X \times S_d[M[-1]]$ be the two morphisms induced by $f$ and $f_0$, respectively.
	Let $\calF \in \catAPerf(X \times S_d[M[-1]])$ be such that $\varphi^\ast(\calF), \varphi_0^\ast(\calF) \in \catCoh_S( X \times S )$.
	We want to prove that $\calF\in \catCoh_{S_d[M[-1]]}( X\times S_d[M[-1]] )$.
	This question is local on $X$, so we can assume that $X$ is affine.
	Let $p \colon X \times S \to S$ and $q \colon X \times S_d[M[-1]] \to S_d[M[-1]]$ be the natural projections.
	Then
	\begin{align}
		f^\ast q_\ast(\calF) \simeq p_\ast \varphi^\ast(\calF) \quad \textrm{and} \quad f_0^\ast q_\ast(\calF) \simeq p_\ast \varphi_0^\ast(\calF)
	\end{align}
	have tor-amplitude $\le 0$.
	Since $p$ is affine, $p_\ast$ is $t$-exact, and therefore the modules $p_\ast \varphi^\ast(\calF)$ and $p_\ast \varphi_0^\ast(\calF)$ are eventually connective.
	The conclusion now follows from \cite[Proposition~16.2.3.1-(3)]{Lurie_SAG}.
	
	We now turn to nilcompleteness.
	Let $S \in \dAff$ be an affine derived scheme and let $S_n \coloneqq \mathsf t_{\le n}(S)$ be its $n$-th truncation.
	We have to prove that the diagram
	\begin{align}
		\begin{tikzcd}[ampersand replacement = \&]
			\catCoh_S(X\times S) \arrow{r} \arrow{d} \& \lim_n \catCoh_{S_n}( X \times S_n ) \arrow{d} \\
			\catAPerf(X\times S) \arrow{r} \& \lim_n \catAPerf( X \times S_n )
		\end{tikzcd}
	\end{align}
	is a pullback.
	Combining \cite[Propositions~19.2.1.5 and 2.7.3.2-(c)]{Lurie_SAG} we see that the bottom horizontal map is an equivalence.
	As the vertical maps are fully faithful, we deduce that the top horizontal map is fully faithful as well.
	Thus, it is enough to check that the top horizontal map is essentially surjective.
	Given $\calF \in \catAPerf(X \times S)$ denote by $\calF_n$ its image in $\catAPerf(X \times S_n)$.
	We wish to show that if each $\calF_n$ belongs to $\catCoh_{S_n}( X \times S_n )$ then $\calF$ belongs to $\catCoh_S(X \times S)$.
	Since the squares
	\begin{align}
		\begin{tikzcd}[ampersand replacement = \&]
		X \times S \arrow{r} \arrow{d} \& X \times S_n \arrow{d} \\ S \arrow{r} \& S_n
		\end{tikzcd}
	\end{align}
	are derived pullback, by using derived base change it suffices to check that the equivalence
	\begin{align}
		\catQCoh^{\mathsf{acn}}(S) \longrightarrow \lim_n \catQCoh^{\mathsf{acn}}(S_n)
	\end{align}
	respects tor-amplitude $\le 0$, where $\catQCoh^{\mathsf{acn}}(Y)$ denotes the full subcategory of $\catQCoh(Y)$ spanned by those quasi-coherent sheaves $\calF$ such that $\pi_i(\calF) = 0$ for $i \ll 0$.
	This follows at once from \cite[Proposition~2.7.3.2-(c)]{Lurie_SAG}.
\end{proof}

\begin{corollary} \label{cor:coh_inf_cohesive_nilcomplete}
	Let $X \in \dSt$ be a derived stack.
	Assume that there exists a flat effective epimorphism $u \colon U \to X$, where $U$ is a geometric derived stack.
	Then the map
	\begin{align}
		\bfCoh(X) \longrightarrow \bfAPerf(X) 
	\end{align}
	is infinitesimally cohesive and nilcomplete.
	In particular, $\bfCoh(X)$ is infinitesimally cohesive and nilcomplete.
\end{corollary}

\begin{proof}
	Combining \cite[Propositions~\ref*{shapes-prop:inf_cohesive}-(\ref*{shapes-prop:inf_cohesive:limits}) and \ref*{shapes-prop:nilcomplete}-(\ref*{shapes-prop:nilcomplete:limits})]{Porta_Sala_Shapes}, we see that infinitesimally cohesive and nilcomplete morphisms are stable under pullbacks.
	Therefore, the first statement is a consequence of Lemmas~\ref{lem:universally_flat_atlas} and \ref{lem:coh_inf_cohesive_nilcomplete}.
	The second statement follows from Lemma \ref{lem:aperf_infinitesimally_cohesive_nilcomplete}.
\end{proof}

We now turn to study the existence of the cotangent complex of $\bfCoh(X)$.
This is slightly trickier, because $\bfAPerf(X)$ does not admit a (global) cotangent complex.
Nevertheless, it is still useful to consider the natural map $\bfCoh(X) \to \bfAPerf(X)$.
Observe that it is $(-1)$-truncated by construction.
In other words, for every $S \in \dAff$, the induced map
\begin{align}
	\bfCoh(X)(S) \longrightarrow \bfAPerf(X)(S) 
\end{align}
is fully faithful.
This is very close to asserting that the map is formally étale, as the following lemma shows:

\begin{lemma} \label{lem:checking_formally_etale}
	Let $F \to G$ be a morphism in $\dSt$.
	Assume that:
	\begin{enumerate}\itemsep0.2cm
		\item \label{item:checking_formally_etale-(1)} for every $S \in \dAff$ the map $F(S) \to G(S)$ is fully faithful;
		\item \label{item:checking_formally_etale-(2)} for every $S \in \dAff$, the natural map
		\begin{align}
				F(S) \longrightarrow F(S_{\mathsf{red}}) \times_{G(S_{\mathsf{red}})} G(S)
		\end{align}
		induces a surjection at the level of $\pi_0$.
	\end{enumerate}
	Then $F \to G$ is formally étale.
\end{lemma}

\begin{proof}
	First, consider the square
	\begin{align}
		\begin{tikzcd}[ampersand replacement = \&]
		F(S) \arrow{r} \arrow{d} \& F(S_{\mathsf{red}}) \arrow{d} \\
		G(S) \arrow{r} \& G(S_{\mathsf{red}})
		\end{tikzcd}\ .
	\end{align}
	Assumption (1) implies that the vertical maps are $(-1)$-truncated, hence so is the map $F(S) \to F(S_{\mathsf{red}}) \times_{G(S_{\mathsf{red}})} G(S)$ as well.
	Assumption (2) implies that it is also surjective on $\pi_0$, hence it is an equivalence.
	In other words, the above square is a pullback.
	
	We now show that $F \to G$ is formally étale.
	Let $S = \Spec(A)$ be an affine derived scheme.
	Let $\calF \in \catQCoh^{\le 0}$ and let $S[\calF] \coloneqq \Spec(A \oplus \calF)$ be the split square-zero extension of $S$ by $\calF$.
	Consider the lifting problem
	\begin{align}
		\begin{tikzcd}[ampersand replacement = \&]
		S \arrow{r} \arrow{d} \& F \arrow{d} \\
		S[M] \arrow{r} \arrow[dashed]{ur} \& G 
		\end{tikzcd} \ .
	\end{align}
	The solid arrows induce the following commutative square in $\cS$:
	\begin{align}
		\begin{tikzcd}[ampersand replacement = \&]
		F(S[M]) \arrow{r} \arrow{d} \& F(S) \arrow{d} \\
		G(S[M]) \arrow{r} \& G(S) 
		\end{tikzcd} \ .
	\end{align}
	To prove that $F \to G$ is formally étale is equivalent to proving that the square is a pullback.
	
	Observe that the above square is part of the following naturally commutative cube:
	\begin{align}
		\begin{tikzcd}[ampersand replacement = \&]
		\& F((S[M])_{\mathsf{red}}) \arrow{rr} \arrow{dd} \& \& F(S_{\mathsf{red}}) \arrow{dd} \\
		F(S[M]) \arrow{ur} \arrow[crossing over]{rr} \arrow{dd} \& \& F(S) \arrow{ur} \\
		\& G((S[M])_{\mathsf{red}}) \arrow{rr} \& \& G(S_{\mathsf{red}}) \\
		G(S[M]) \arrow{ur} \arrow{rr} \& \& G(S) \arrow{ur} \arrow[leftarrow, crossing over]{uu}
		\end{tikzcd} \ .
	\end{align}
	The horizontal arrows of the back square are equivalences, and therefore the back square is a pullback.
	The argument we gave at the beginning shows that the side squares are pullbacks.
	Therefore, the conclusion follows.
\end{proof}

To check condition (1) of the above lemma for $F = \bfCoh(X)$ and $G = \bfAPerf(X)$, we need the following variation of the local criterion of flatness.

\begin{lemma} \label{lem:local_criterion_flatness}
	Let $f \colon X \to S$ be a morphism in $\dSt$ and let $\calF \in \catAPerf(X)$.
	Assume that:
	\begin{enumerate}\itemsep0.2cm
		\item $S$ is an affine derived scheme;
		\item there is a flat effective epimorphism $u \colon U \to X$, where $U$ is a geometric derived stack;
		\item for every pullback square
		\begin{align}
			\begin{tikzcd}[ampersand replacement = \&]
			X_s \arrow{r}{j_s} \arrow{d} \& X \arrow{d} \\
			\Spec(K) \arrow{r}{s} \& S 
			\end{tikzcd} \ ,
		\end{align}
		where $K$ is a field, $j_s^\ast(\calF) \in \catAPerf(X_s)$ has tor-amplitude within $[a,b]$ relative to $\Spec(K)$.
	\end{enumerate}
	Then $\calF$ has tor-amplitude within $[a,b]$ relative to $S$.
\end{lemma}

\begin{proof}
	Let $U_s \coloneqq \Spec(K) \times_S U$.
	Since $u \colon U \to X$ is a flat effective epimorphism, Lemma~\ref{lem:universally_flat_effective_epi} implies that the same goes for $u_s \colon U_s \to X_s$.
	Therefore, Lemma \ref{lem:flat_effective_epimorphism} allows is to replace $X$ by $U$.
	Applying this lemma one more time, we can further assume $U$ is an affine derived scheme.
	At this point, the conclusion follows from the usual local criterion for flatness, see \cite[Proposition~6.1.4.5]{Lurie_SAG}.
\end{proof}

\begin{corollary} \label{cor:formally_etale}
	Let $X \in \dSt$ be a derived stack and assume there exists a flat effective epimorphism $u \colon U \to X$, where $U$ is a geometric derived stack.
	Then the natural map $\bfCoh(X) \to \bfAPerf(X)$ is formally étale.
\end{corollary}

\begin{proof}
	We apply Lemma \ref{lem:checking_formally_etale}.
	We already remarked that assumption \eqref{item:checking_formally_etale-(1)} is satisfied, essentially by construction.
	Let now $S \in \dAff$ and let
	\begin{align}
		 j \colon X \times S_{\mathsf{red}} \longrightarrow X \times S 
	\end{align}
	be the natural morphism.
	Let $\calF \in \catAPerf(X \times S)$.
	Then Lemma~\ref{lem:local_criterion_flatness} implies that $\calF$ is flat relative to $S$ if and only if $j^\ast(\calF)$ is flat relative to $S_{\mathsf{red}}$.
	This implies that assumption \eqref{item:checking_formally_etale-(2)} of Lemma~\ref{lem:checking_formally_etale} is satisfied as well, and the conclusion follows.
\end{proof}

Since in many cases $\bfPerf(X)$ admits a global cotangent complex, it is useful to factor the map $\bfCoh(X) \to \bfAPerf(X)$ through $\bfPerf(X)$.
The following lemma provides a useful criterion to check when this is the case:

\begin{lemma} \label{lem:checking_perfectness_on_fibers}
	Let $f \colon X \to S$ be a morphism of derived stacks.
	Let $\calF \in \catAPerf(X)$ be an almost perfect complex and let $a \le b$ be integers.
	Assume that:
	\begin{enumerate}\itemsep0.2cm
		\item $S$ is an affine derived scheme;
		\item there exists a flat effective epimorphism $u \colon U \to X$, where $U$ is a geometric derived stack locally almost of finite type;
		\item for every ladder of pullback squares
		\begin{align}
			\begin{tikzcd}[ampersand replacement = \&]
				U_s \arrow{d}{i_s} \arrow{r}{u_s} \& X_s \arrow{r} \arrow{d}{j_s} \& \Spec(K) \arrow{d}{s} \\
				U \arrow{r}{u} \& X \arrow{r}{f} \& S
			\end{tikzcd} \ ,
		\end{align}
		where $K$ is a field, $u_s^\ast j_s^\ast(\calF) \in \catAPerf(U_s)$ has tor-amplitude within $[a,b]$.
	\end{enumerate}
	Then $u^\ast(\calF) \in \catAPerf(U)$ has tor-amplitude within $[a,b]$ and therefore $\calF$ belongs to $\catPerf(X)$.
\end{lemma}

\begin{proof}
	Since $u \colon U \to X$ is a flat effective epimorphism, Lemmas \ref{lem:flat_effective_epimorphism} and \ref{lem:universally_flat_effective_epi} allow to replace $X$ by $U$.
	In other words, we can assume $X$ to be a geometric derived stack locally almost of finite type from the very beginning.
	Applying Lemma~\ref{lem:flat_effective_epimorphism} a second time to an affine atlas of $X$, we can further assume $X$ is an affine derived scheme, say $X = \Spec(B)$.
	
	Given a geometric point $x \colon \Spec(K) \to X$, we let $B_{(x)}$ denote the localization
	\begin{align}
		B_{(x)} \coloneqq \colim_{x \in U \subset X} \scrO_X(U) \ ,
	\end{align}
	where the colimit ranges over all the open Zariski neighborhoods of the image of $x$ inside $X$.
	It is then enough to prove that for each such geometric point, $\calF \otimes_B B_{(x)}$ is in tor-amplitude $[a,b]$.
	
	Given $x \colon \Spec(K) \to X$ let $s \coloneqq f \circ x \colon \Spec(K) \to S$.
	By assumption $j_s^\ast(\calF) \in \catAPerf(X_s)$ is in tor-amplitude $[a,b]$.
	Let $\overline{x} \colon \Spec(K) \to X_s$ be the induced point.
	Then $x = j_s \circ \overline{x}$, and therefore $x^\ast(\calF) \simeq \overline{x}^\ast( j_s^\ast(\calF) )$ is in tor-amplitude $[a,b]$.
	Let $\kappa$ denote the residue field of the local ring $\pi_0( B_{(x)} )$.
	Since the map $\kappa \to K$ is faithfully flat, we can assume without loss of generality that $K = \kappa$.
	In this way, we are reduced to the situation of Lemma~\ref{lem:local_criterion_flatness} with $X = S$.
	
	Finally, we remark that since $u$ is an effective epimorphism, the diagram
	\begin{align}
		\begin{tikzcd}[ampersand replacement = \&]
		\catPerf(X) \arrow{r}{u^\ast} \arrow{d} \& \catPerf(U) \arrow{d} \\
		\catAPerf(X) \arrow{r}{u^\ast} \& \catAPerf(U)
		\end{tikzcd}
	\end{align}
	is a pullback square.
	Therefore, an almost perfect complex $\calF \in \catAPerf(X)$ is perfect if and only if $u^\ast(F)$ is.
	The proof is complete.
\end{proof}

\begin{corollary} \label{cor:coherent_over_smooth_are_perfect}
	Let $X$ be a derived stack and assume there exists a flat effective epimorphism $u \colon U \to X$, where $U$ is a smooth geometric derived stack.
	Then for every $S \in \dAff$, the subcategory $\catCoh_S(X \times S) \subseteq \catAPerf(X \times S)$ is contained in $\catPerf(X \times S)$.
	In particular, the natural map $\bfCoh(X) \to \bfAPerf(X)$ induces a formally étale map
	\begin{align}
		\bfCoh(X) \longrightarrow \bfPerf(X) \ .
	\end{align}
\end{corollary}

\begin{proof}
	Since $u$ is a flat effective epimorphism, Lemmas~\ref{lem:flat_effective_epimorphism} and \ref{lem:universally_flat_effective_epi} imply that it is enough to prove the corollary for $U = X$.
	In this case, we have to check that if $\calF \in \catAPerf(X \times S)$ is flat relative to $S$, then it belongs to $\catPerf(X \times S)$.
	The question is local on $X$, and therefore we can further assume that $X$ is affine and connected.
	As $X$ is smooth, it is of pure dimension $n$ for some integer $n$.
	It follows that every $\calG \in \catCoh^\heartsuit(X)$ has tor-amplitude $\le n$ on $X$.
	At this point, the first statement follows directly from Lemma~\ref{lem:checking_perfectness_on_fibers}.
	As for the second statement, the existence of the factorization follows from what we have just discussed.
	Corollary \ref{cor:formally_etale} implies that $\bfCoh(X) \to \bfPerf(X)$ is formally \'etale.
\end{proof}

\begin{corollary} \label{cor:coh_global_cotangent_complex_smooth_case}
	Let $X$ be a derived stack and let $u \colon U \to X$ be a flat effective epimorphism, where $U$ is a smooth geometric derived stack.
	If $\bfPerf(X)$ admits a global cotangent complex, then so does $\bfCoh(X)$.
\end{corollary}

\begin{proof}
	This is a direct consequence of Corollary \ref{cor:coherent_over_smooth_are_perfect}.
\end{proof}

We define $\bfBun(X)$ as 
\begin{align}
	\bfBun(X)\coloneqq \coprod_{n\geq 0}\, \bfMap(X, \BGL_n)\ .
\end{align}
It is an open substack of $\bfCoh(X)$. We call it the \textit{derived stack of vector bundles on $X$}.

\subsection{Coherent sheaves on schemes}\label{ss:coherentsheavesonschemes}

We now specialize to the case where $X$ is an underived complex scheme of finite type.
Our goal is to prove that if $X$ is proper, then $\bfCoh(X)$ is geometric, and provide some estimates on the tor-amplitude of its cotangent complex.
Observe that in this case, $X$ has universally finite cohomological dimension.
Corollary~\ref{cor:coh_inf_cohesive_nilcomplete} shows that $\bfCoh(X)$ is infinitesimally cohesive and nilcomplete.
In virtue of Lurie's representability theorem \cite[Theorem~18.1.0.2]{Lurie_SAG}, in order to prove that $\bfCoh(X)$ is geometric it is enough to check that it admits a global cotangent complex and that its truncation is geometric.
Recall that if $X$ is smooth and proper, then $\bfPerf(X)$ admits a global cotangent complex, see for instance \cite[Corollary~\ref*{shapes-cor:cotangent_complex_Perf}]{Porta_Sala_Shapes}.
Therefore, Corollary~\ref{cor:coh_global_cotangent_complex_smooth_case} implies that under these assumptions the same is true for $\bfCoh(X)$.
We can relax the smoothness by carrying out a more careful analysis as follows:

\begin{lemma} \label{lem:coh_cotangent_complex}
	Let $X$ be a proper, underived complex scheme.
	Then the derived stack $\bfCoh(X)$ admits a global cotangent complex.
\end{lemma}

\begin{proof}
	Let $S = \Spec(A)$ be an affine derived scheme and let $x \colon S \to \bfCoh(X)$ be a morphism.
	Let $\calF \in \catCoh_{\Spec(A)}( X \times \Spec(A) )$ be the corresponding coherent complex on $X \times S$ relative to $S$.
	Let
	\begin{align}
		F \coloneqq S \times_{\bfCoh(X)} S
	\end{align}
	be the loop stack based at $x$ and let $\delta_x \colon S \to F$ be the induced morphism.
	Since $\bfCoh(X)$ is infinitesimally cohesive thanks to Lemma \ref{lem:coh_inf_cohesive_nilcomplete}, \cite[Proposition~\ref*{shapes-prop:cotangent_complex_loop}]{Porta_Sala_Shapes} implies that $\bfCoh(X)$ admits a cotangent complex at $x$ if and only if $F$ admits a cotangent complex at $\delta_x$ relative to $S\times S$.
	We have to prove that the functor
	\begin{align}
		\mathsf{Der}_F(A;-) \colon A\Mod \longrightarrow \cS
	\end{align}
	defined by
	\begin{align}
		\mathsf{Der}_F(A;M) \coloneqq \fib( F(S[M]) \to F(S) )
	\end{align}
	is representable by an eventually connective module.
	Here $S[M] \coloneqq \Spec(A \oplus M)$, and the fiber is taken at the point $x$.
	We observe that
	\begin{align}
		F(S[M]) \simeq \fib( \Map_{\catQCoh(X \times S)}(d_0^\ast(\calF), d_0^\ast(\calF)) \to \Map_{\catQCoh(X \times S)}(\calF, \calF) ) \ ,
	\end{align}
	the fiber being taken at the identity of $\calF$.
	Unraveling the definitions, we therefore see that
	\begin{align}
		\mathsf{Der}_F(A;M) \simeq \Map_{\catQCoh(X\times S)}( \calF, \calF \otimes p^\ast M ) \ ,
	\end{align}
	where $p \colon X \times S \to S$ is the canonical projection.
	Since $\catQCoh(S)$ and $\catQCoh(X \times S)$ are presentable, the adjoint functor theorem shows that it is enough to show that the functor $\mathsf{Der}_F(A;-)$ commutes with arbitrary limits.
		Since $\Map_{\catQCoh(X \times S)}( \calF, - )$ commutes with limits, it is enough to prove that the functor
	\begin{align}
		\calF \otimes p^\ast(-) \colon \catQCoh(S) \longrightarrow \catQCoh(X \times S) 
	\end{align}
	commutes with limits.
	Since $X$ is quasi-compact and quasi-separated, we know that $\catQCoh(X \times S)$ is generated by a single perfect complex $\calG \in \catPerf(X \times S)$.
	Since $S$ is affine, this implies that the functor
	\begin{align}
			p_\ast(\calG^\vee \otimes - ) \colon \catQCoh(X \times S)  \longrightarrow \catQCoh(S) 
	\end{align}
	is conservative.
	Since $\calG$ is perfect, $\calG^\vee \otimes -$ commutes with arbitrary limits, and since $p^\ast \dashv p_\ast$, the same holds for $p_\ast$.
	Therefore, it is enough to prove that
	\begin{align}
		p_\ast( \calG^\vee \otimes ( \calF \otimes p^\ast(-) ) ) \colon \catQCoh(S) \longrightarrow \catQCoh(S) 
	\end{align}
	commutes with limits.
	Using the projection formula, we can rewrite this functor as
	\begin{align}
		p_\ast(\calG^\vee \otimes \calF) \otimes - \colon \catQCoh(S) \longrightarrow \catQCoh(S) \ . 
	\end{align}
	It is therefore enough to prove that $p_\ast(\calG^\vee \otimes \calF)$ is a perfect $A$-module.
	Since $p$ is proper, $p_\ast(\calG^\vee \otimes \calF)$ is almost perfect.
	In virtue of \cite[Proposition~7.2.4.23-(4)]{Lurie_Higher_algebra}, it is therefore enough to prove that it has finite tor-amplitude.
	Observe that $\calG^\vee \otimes \calF$ has finite tor-amplitude relative to $S$: indeed, if $M$ is a discrete $A$-module, then $\calF \otimes p^\ast(M)$ is again discrete because $\calF$ is flat.
	Since $\calG$ is perfect, we deduce that $\calG^\vee \otimes \calF \otimes p^\ast(M)$ has uniformly bounded cohomological amplitude.
	Therefore, \cite[Proposition~\ref*{shapes-prop:finite_tor_amplitude}]{Porta_Sala_Shapes} implies that $p_\ast(\calG^\vee \otimes \calF)$ has finite tor-amplitude over $S$.
	In conclusion, we deduce that there exists an object $\calE \in A\Mod$ together with a natural equivalence
	\begin{align}
		\Map_{A\Mod}(\calE, M) \simeq \mathsf{Der}_F(A;M) \ . 
	\end{align}
	Now observe that, since $\calF$ is flat relative to $S$, for every eventually coconnective $M \in A\Mod$, the $A$-module
	\begin{align}
		\Map_{\catQCoh(X \times S)}( \calF, \calF \otimes p^\ast(M) )
	\end{align}
	is again eventually coconnective.
	In other words, for every eventually coconnective $M$, the $A$-module $\Map_{A\Mod}(\calE, M)$ is eventually coconnective.
	This implies that $\calE$ must be eventually connective.
	As a consequence, $\calE$ is a cotangent complex for $F$ at $\delta_x$, and therefore $\bfCoh(X)$ admits a cotangent complex at the point $x$, given by $\calE[-1]$.
	
	We are left to prove that the cotangent complex is global.
	It is enough to prove that the cotangent complex of $F$ is global, that is that for every map $f \colon T \coloneqq \Spec(B) \to \Spec(A)$, the object $f^\ast(\calE)$ represents the functor $\mathsf{Der}_F(B;-)$.
	Consider the derived fiber product
	\begin{align}
		\begin{tikzcd}[ampersand replacement = \&]
		X \times T \arrow{r}{g} \arrow{d}{q} \& X \times S \arrow{d}{p} \\
		T \arrow{r}{f} \& S 
	\end{tikzcd} \ .
	\end{align}
	Then for any $M \in B\Mod$, we have
	\begin{align}
		\Map_{B\Mod}(f^\ast\calE, M) & \simeq \Map_{A\Mod}( \calE, f_\ast(M) ) \\[2pt]
		& \simeq \Map_{\catQCoh(X \times S)}( \calF, \calF \otimes p^\ast (f_\ast(M)) ) \\[2pt]
		& \simeq \Map_{\catQCoh(X \times S)}( \calF, \calF \otimes g_\ast( q^\ast(M) ) ) \\[2pt]
		& \simeq \Map_{\catQCoh(X \times S)}(\calF, g_\ast( g^\ast(\calF) \otimes q^\ast(M) )) \\[2pt]
		& \simeq \Map_{\catQCoh(X \times T)}( g^*(\calF) , g^*(\calF) \otimes q^*(M) ) \ .
	\end{align}
	The conclusion therefore follows from the Yoneda lemma.
\end{proof}

\begin{remark}
	In the setting of the above corollary, let $x \colon S \coloneqq \Spec(A) \to \bfCoh(X)$ be a point representing a coherent sheaf $\calF$ on $X \times S$ relative to $S$ which is furthermore perfect in $\catQCoh(X \times S)$ (this is always the case when $X$ is smooth, see Corollary \ref{cor:coherent_over_smooth_are_perfect}).
	In this case, the cotangent complex is given explicitly by the formula $p_+( \calF \otimes \calF^\vee )[1]$, where $p_+$ is the left adjoint to $p^\ast$.
	The existence of $p_+$ is a consequence of the fact that $p$ is proper and flat, see \cite[Proposition~6.4.5.3]{Lurie_SAG} (and also \cite[Proposition~\ref*{shapes-prop:plus_pushforward}]{Porta_Sala_Shapes}).
\end{remark}

As for the truncation of $\bfCoh(X)$, we have:

\begin{lemma} \label{lem:truncation_stack_coherent_sheaves}
	Let $X$ be a proper, underived complex scheme.
	Then the truncation $\trunc{\bfCoh(X)}$ coincides with the usual stack of coherent sheaves on $X$.
	\end{lemma}

\begin{proof}
	Let $S$ be an underived affine scheme. 
	By definition, a morphism $S \to \bfCoh(X)$ corresponds to an almost perfect complex $\calF \in \catAPerf(X \times S)$ which furthermore has tor-amplitude $\le 0$ relative to $S$.
	As $S$ is underived, having tor-amplitude $\le 0$ relative to $S$ is equivalent to asserting that $\calF$ belongs to $\catAPerf^\heartsuit(X \times S)$.
	The conclusion follows.
\end{proof}

In other words, the derived stack $\bfCoh(X)$ provides a derived enhancement of the classical stack\footnote{The construction of such a stack is described, e.g., in \cite[Chapitre~4]{Laumon_Champs}, \cite[Tag~08KA]{Stacks_project}.} of coherent sheaves.
We therefore get:

\begin{proposition}\label{prop:coh_geometric}
	Let $X$ be a proper, underived complex scheme.
	Then the derived stack $\bfCoh(X)$ is geometric and locally of finite presentation.
	If furthermore $X$ is smooth, then the canonical map $\bfCoh(X) \to \bfPerf(X)$ is representable by \'etale geometric $0$-stacks.\footnote{After the first version of the present paper was released, it appeared on the arXiv the second version of \cite{HalpernLeistner_Preygel_Categorical_properness} in which a similar statement was proved, cf.\ \cite[Theorem~5.2.2]{HalpernLeistner_Preygel_Categorical_properness}.}
	\end{proposition}

\begin{proof}
	Lemma~\ref{lem:truncation_stack_coherent_sheaves} implies that $\trunc{\bfCoh(X)}$ coincides with the usual stack of coherent sheaves on $X$, which we know to be a geometric classical stack (cf.\ \cite[Théorème~4.6.2.1]{Laumon_Champs} or \cite[Tag~08WC]{Stacks_project}). On the other hand, combining Corollary~\ref{cor:coh_inf_cohesive_nilcomplete} and Lemma~\ref{lem:coh_cotangent_complex} we see that $\bfCoh(X)$ is infinitesimally cohesive, nilcomplete and admits a global cotangent complex.
	Therefore the assumptions of Lurie's representability theorem \cite[Theorem~18.1.0.2]{Lurie_SAG} are satisfied and so we deduce that $\bfCoh(X)$ is geometric and locally of finite presentation.
	As for the second statement, we already know that $\bfCoh(X) \to \bfPerf(X)$ is formally étale.
	As both stacks are of locally of finite type, it follows that this map is étale as well.
	Finally, since $\bfCoh(X) \to \bfPerf(X)$ is $(-1)$-truncated, we see that for every affine derived scheme $S$, the truncation of $S \times_{\bfPerf(X)} \bfCoh(X)$ takes values in $\mathsf{Set}$.
	The conclusion follows.
\end{proof}

\begin{remark}
	Let $X$ be a smooth and proper complex scheme.
	In this case, the derived stack $\bfCoh(X)$ has been considered to some extent in \cite{Toen_Vaquie_Moduli}.
	Indeed, in their work they provide a geometric derived stack $\calM_{\catPerf(X)}^{1 \textrm{-} \mathsf{rig}}$ classifying families of $1$-rigid perfect complexes (see \S 3.4 in \textit{loc.\ cit.} for the precise definition).
	There is a canonical map $\trunc{\bfCoh(X)} \to \trunc{(\calM_{\catPerf(X)}^{1 \textrm{-} \mathsf{rig}})}$.
	One can check that this map is formally étale.
	Since it is a map between stacks locally almost of finite type, it follows that it is actually étale.
	Therefore, the derived structure of $\calM_{\catPerf(X)}^{1 \textrm{-} \mathsf{rig}}$ induces a canonical derived enhancement of $\trunc{\bfCoh(X)}$.
	Unraveling the definitions, we can describe the functor of points of such derived enhancement as follows: it sends $S \in \dAff$ to the full subcategory of $\bfPerf(X \times S)$ spanned by those $\calF$ whose pullback to $X \times \trunc{S}$ is concentrated in cohomological degree $0$.
	Remark \ref{rem:tor_amplitude_underived} implies that it canonically coincides with our $\bfCoh(X)$.
	However, this method is somehow non-explicit, and heavily relies on the fact that $X$ is a smooth and proper scheme.
	Our method provides instead an explicit description of the functor of points of this derived enhancement, and allows us to deal with a wider class of stacks $X$.
\end{remark}

\begin{corollary} \label{cor:coh_curve_smooth}
	Let $X$ be a smooth and proper complex scheme of dimension $n$.
		Then the cotangent complex $\LL_{\bfCoh(X)}$ is perfect and has tor-amplitude within $[-1, n-1]$.
		In particular, $\bfCoh(X)$ is smooth when $X$ is a curve and derived lci when $X$ is a surface.
\end{corollary}

\begin{proof}
	It is enough to check that for every affine derived $S=\Spec(A) \in \dAff$ and every point $x \colon S \to \bfCoh(X)$, $x^\ast \T_{\bfCoh(X)}$ is perfect and in tor-amplitude $[1-n,1]$.
		Let $\calF \in \catAPerf(X \times \Spec(A))$ be the almost perfect complex classified by $x$ and let $p \colon X \times \Spec(A) \to \Spec(A)$ be the canonical projection.
	Since $X$ is smooth, Corollary \ref{cor:coherent_over_smooth_are_perfect} implies that $\calF$ is perfect.
	Moreover, Lemma \ref{lem:coh_cotangent_complex} shows that
	\begin{align}
		x^\ast \T_{\bfCoh(X)} \simeq p_\ast \End(\calF)[1] \ . 
	\end{align}
	Since $p$ is proper and smooth, the pushforward $p_\ast$ preserves perfect complexes (see \cite[Theorem~6.1.3.2]{Lurie_SAG}).
	As $\End(\calF) \simeq \calF \otimes \calF^\vee$ is perfect, we conclude that $x^\ast \T_{\bfCoh(X)}$ is perfect.
	
	We are then left to check that it is in tor-amplitude $[1-n,1]$. Let $j\colon \trunc{(\Spec(A)}\to \Spec(A)$ be the canonical inclusion. It is enough to prove that $j^\ast x^\ast \T_{\bfCoh(X)}$ has tor-amplitude within $[1-n,1]$. In other words, we can assume $\Spec(A)$ to be underived.
	Using Lemma~\ref{lem:checking_perfectness_on_fibers} we can further assume $S$ is the spectrum of a field.
	
	First note that for every pair of coherent sheaves $\calG$, $\calG' \in \catCoh_S(X)$, $p_\ast \calHom_X(\calG, \calG')$ is coconnective and has coherent cohomology.
	Since $S$ is the spectrum of a field, it is therefore a perfect complex.
	By Grothendieck--Serre duality for smooth proper morphisms of relative pure dimension between Noetherian schemes (see \cite[\S3.4]{Conrad_Duality}, \cite[\S~C.1]{BBHR_Transforms}, and references therein), we have
	\begin{align}
		\left( p_\ast \calEnd(\calF) \right)^\vee \simeq p_\ast \calHom(\calF, \calF \otimes p_X^\ast \omega_X[n])\ ,
	\end{align}
	where $\omega_X$ is the canonical bundle of $X$, and $p_X$ the projection from $X\times S$ to $X$. 
	The right-hand side is $n$-coconnective.
	This implies that
	\begin{align}
			\pi_i( \left( p_\ast \calEnd(\calF) \right)^\vee ) \simeq 0 
	\end{align}
	for $i > n$.
	Since $S$ is the spectrum of a field, $\pi_j( p_\ast(\calEnd(\calF)) )$ is projective and
	\begin{align}
			\left( \pi_j( p_\ast( \calEnd(\calF) ) ) \right)^\vee \simeq \pi_{-j} \left( \left( p_\ast \calEnd(\calF) \right)^\vee \right) 
	\end{align}
	This shows that $p_\ast \calEnd(\calF)$ has tor-amplitude within $[-n,0]$, and therefore that $x^\ast \T_{\bfCoh(X)}$ has tor-amplitude within $[1-n,1]$.
	 \footnote{This argument is borrowed from \cite[Example~2.2.3]{HalpernLeistner_D-equivalence}.}
	
	\end{proof}

\subsubsection{Non-proper case}\label{ss:nonproper}

We can relax the properness assumption on $X$ by working with perfect complexes with proper support:

\begin{definition}
	Let $p \colon X \to S$ be a morphism of derived schemes locally almost of finite presentation and let $\calF \in \catAPerf(X)$ be an almost perfect complex.
	We say that $\calF$ has \textit{proper support relative to $S$} if there exists a closed subscheme $Z \hookrightarrow X$ such that $Z \to S$ is proper and $\calF \vert_{X \smallsetminus Z} \simeq 0$.
\end{definition}

Perfect complexes with proper support have the following property:
\begin{proposition} \label{prop:perfect_complexes_proper_support}
	Let $p \colon X \to S$ be a morphism between quasi-compact, quasi-separated derived schemes locally almost of finite presentation and let $\calF\in \catAPerf(X)$ be an almost perfect complex.
	If $\calF$ has proper support relative to $S$, then for every morphism $T \to S$ and every perfect complex $\calG \in \catPerf(X \times_S T)$ one has
	\begin{align}
		p_{T\ast}(\calG^\vee \otimes \calF) \in \catAPerf(S) \ .
	\end{align}
	In particular, if $\calF$ has proper support then $p_\ast(\calF)$ belongs to $\catAPerf(S)$.
	
	The converse holds true provided that $p$ is separated and $\calF$ has finite cohomological amplitude and finite tor-amplitude relative to $S$.
\end{proposition}

\begin{proof}
	To prove the first statement, it is enough to take $T = S$.
	Assume first that $\calF$ has proper support relative to $S$.
	Let $Z \hookrightarrow X$ be a closed subscheme such that $Z \to S$ is proper and $\calF\vert_{X \smallsetminus Z} \simeq 0$.
	Then for every perfect complex $\calG \in \catPerf(X)$, we have $(\calG^\vee \otimes \calF)\vert_{X \smallsetminus Z} \simeq 0$, and therefore $\calG^\vee \otimes \calF$ has again proper support relative to $S$.
	It is therefore enough to prove the statement when $\calG = \scrO_X$.
	Let $\cC$ be the full subcategory of $\catAPerf(X)$ spanned by those almost perfect complexes $\calF$ such that $p_\ast(\calF)$ belongs to $\catAPerf(S)$.
	We want to show that $\cC$ contains all almost perfect complexes with proper support relative to $S$.
	Let $\calF$ be such an object, and assume furthermore that $\calF \in \catAPerf^\heartsuit(X) \simeq \catAPerf^\heartsuit(\trunc{X})$.
	Then the question only concerns the classical truncations of $X$ and $S$.
	In this case, there exists a nilthickening $Z'$ of $Z$ together with a map $j \colon Z' \to \trunc{X}$, a coherent sheaf $\calF' \in \catAPerf^\heartsuit(Z')$ and an isomorphism $j_\ast(\calF') \simeq \calF$.
	We can therefore compute the pushforward of $\calF$ along $X \to S$ as the pushforward of $\calF'$ along $Z' \to S$.
	As the latter map is again proper, we deduce that $p_\ast(\calF)$ belongs to $\catAPerf(S)$.
	Since $X$ is quasi-compact, it is straightforward to deduce from here that whenever $\calF$ has bounded cohomological amplitude and proper support relative to $S$, then $p_\ast(\calF)$ belongs to $\catAPerf(S)$.
	Finally, since $X$ is quasi-compact and quasi-separated, we see that the functor $p_\ast$ has finite cohomological dimension.
	It is then possible to extend the result to the whole category of almost perfect complexes on $X$ with proper support relative to $S$.
	
	\medskip
	
	Assume now that $p$ is separated and let $\calF$ be a bounded almost perfect complex on $X$ such that for every $\calG\in \catPerf(X)$, one has $p_\ast(\calG^\vee \otimes \calF) \in \catAPerf(S)$.
	We want to prove that it has proper support.
	Since $\calF$ is cohomologically bounded, it is enough to prove that for every integer $i \in \Z$, $\pi_i(\calF)$ has proper support.
	Since $\pi_i(\calF)$ belongs to $\catAPerf^\heartsuit(X) \simeq \catAPerf^\heartsuit(\trunc{X})$, we can assume that both $X$ and $S$ are underived, and that $\calF$ is discrete.
	Let $Z \coloneqq \mathsf{supp}(\calF)$ and observe that since $\calF$ is coherent, this is a closed subset of $X$.
	Since $p$ is separated, it is enough to prove that the map $Z \to S$ is universally closed.
	Since the assumptions on $\calF$ are stable under arbitrary base-change along $T \to S$, we see that it is enough to prove that $Z \to S$ is closed.
	Let $Z' \subseteq Z$ be a closed subset.
	Since $\calF$ is coherent, $Z$ is closed in $X$ and so is $Z'$.
	Using \cite{Thomason}, we can find a perfect complex $\calG$ on $X$ such that the support of $\calG$ coincides exactly with $Z'$.
	It follows that $\calG^\vee \otimes \calF$ is again supported exactly on $Z'$, and furthermore $p_\ast(\calG^\vee \otimes \calF)$ is almost perfect.
	In particular, the support of $p_\ast(\calG^\vee \otimes \calF)$ is closed, and therefore it coincides with the image of the support of $\calG^\vee \otimes \calF$, which was $Z'$.
	This completes the proof.
\end{proof}

Let now $X$ be a quasi-projective (underived) scheme and let $\bfPerf_{\mathsf{prop}}(X)$ be the derived moduli stack parameterizing families of perfect complexes on $X$ with proper support.
There is a natural map $\bfPerf_{\mathsf{prop}}(X) \to \bfPerf(X)$, and we set
\begin{align}
	\bfCoh_{\mathsf{prop}}(X) \coloneqq \bfPerf_{\mathsf{prop}}(X) \times_{\bfPerf(X)} \bfCoh(X)\ . 
\end{align}
This derived stack is infinitesimally cohesive and nilcomplete.
Furthermore, we have:

\begin{proposition}
	The derived stacks $\bfPerf_{\mathsf{prop}}(X)$ and $\bfCoh_{\mathsf{prop}}(X)$ admit a global cotangent complex.
\end{proposition}

\begin{proof}
	It is enough to observe that in the proof of Lemma~\ref{lem:coh_cotangent_complex} one only needs to know that for every affine derived scheme $S$ and every $\calG \in \bfPerf(X)$, the complex $p_\ast(\calG^\vee \otimes \calF)$ is almost perfect in $\catQCoh(S)$.
	This is true in this setting thanks to Proposition \ref{prop:perfect_complexes_proper_support}.
\end{proof}

\begin{remark}
	The derived stack $\bfPerf(X)$ does not admit a global cotangent complex.
	Nevertheless, one can show that it admits a pro-cotangent complex in the sense of \cite[Definition 1.4.1.4]{Gaitsgory_Rozenblyum_Study_II}.
	The natural inclusion $\bfPerf_{\mathsf{prop}}(X) \to \bfPerf(X)$ then becomes formally étale, in the sense that the relative pro-cotangent complex is zero.
\end{remark}

\begin{remark}
	Let $\overline{X}$ be a smooth compactification of $X$.
	Then there is a natural map $\bfCoh_{\mathsf{prop}}(X) \to \bfCoh_{\mathsf{prop}}(\overline{X})$, which is furthermore representable by open Zariski immersions.
\end{remark}

The truncation of $\bfCoh_{\mathsf{prop}}(X)$ coincides with the classical stack of coherent sheaves with proper support.
As shown in \cite[Tag~0DLX]{Stacks_project}, this moduli stack is a geometric classical stack.
We deduce that $\bfCoh_{\mathsf{prop}}(X)$ is a geometric derived stack.

\subsubsection{Other examples of moduli stacks}\label{ss:other-moduli-spaces}

Let $X$ be a smooth projective complex scheme and let $H$ be a fixed ample divisor. Recall that for any polynomial $P(m)\in \Q[m]$ there exists an open substack $\trunc{\bfCoh}^{P}(X)$ of $\trunc{\bfCoh(X)}$ parameterizing flat families of coherent sheaves $\calF$ on $X$ with fixed Hilbert polynomial $P$, i.e., for $n\gg 0$ 
\begin{align}
	\dim H^0(X, \calF\otimes \scrO_X(nH)) = P(n)\ .
\end{align}
We denote by $\bfCoh^{P}(X)$ its canonical derived enhancement\footnote{The construction of such a derived enhancement follows from \cite[Proposition~2.1]{Schurg_Toen_Vezzosi}.}. Similarly, we define $\bfBun^{P}(X)$.

For any nonzero polynomial $P(m)\in \Q[m]$ of degree $d$, we denote by $P(m)^{\mathsf{red}}$ its \textit{reduced polynomial}, which is defined as $P(m)/\alpha_{d}$, where $\alpha_{d}$ is the leading coefficient of $P(m)$. Given a monic polynomial $p$, define
\begin{align}
	\bfCoh^{p}(X)\coloneqq \coprod_{P^{\mathsf{red}}=p}\, \bfCoh^{P}(X)\quad\text{and}\quad \bfBun^{p}(X)\coloneqq \coprod_{P^{\mathsf{red}}=p}\, \bfBun^{P}(X)\ .
\end{align}
Assume that $\deg(p)=\dim(X)$. Recall that Gieseker $H$-semistability is an open property\footnote{Cf.\ \cite[Definition~1.2.4]{Huybrechts_Lehn_Moduli_2010} for the definition of $H$-semistability of coherent sheaves on projective schemes and \cite[Proposition~2.3.1]{Huybrechts_Lehn_Moduli_2010} for the openness property in families of the $H$-semistability.}. Thus there exists an open substack $\trunc{\bfCoh^{\mathsf{ss},\, p}(X)}$ of $\trunc{\bfCoh^{p}(X)}$ parameterizing families of $H$-semistable coherent sheaves on $X$ with fixed reduced polynomial $p$; we denote by $\bfCoh^{\mathsf{ss}, \, p}(X)$ its canonical derived enhancement. Similarly, we define $\bfBun^{\mathsf{ss}, \, p}(X)$.

Finally, let $0\le d\le \dim(X)$ be an integer and define
\begin{align}
	\bfCoh^{\leqslant d}(X)\coloneqq  \coprod_{\deg(P)\leqslant d}\, \bfCoh^{P}(X)\ . 
\end{align}
\begin{remark}
	Let $X$ be a smooth projective complex curve. Then the assignment of a monic polynomial $p(m)\in\Q[m]$ of degree one is equivalent to the assignment of a \textit{slope} $\mu\in \Q$. In addition, in the one-dimensional case we have $\bfBun^{\mathsf{ss}, \, \mu}(X)\simeq \bfCoh^{\mathsf{ss}, \, \mu}(X)$.
\end{remark}

Finally, assume that $X$ is only quasi-projective. As above, we can define the derived moduli stack $\bfCoh_{\mathsf{prop}}^{\leqslant d}(X)$ of coherent sheaves on $X$ with proper support and dimension of the support less or equal $d$.

\subsection{Coherent sheaves on Simpson's shapes}\label{ss:coherent-shapes}

Let $X$ be a smooth and proper complex scheme.
In this section, we introduce derived enhancements of the classical stacks of finite-dimensional representations of $\pi_1(X)$, of vector bundles with flat connections on $X$ and of Higgs sheaves on $X$.
In order to treat these three cases in a uniform way, we shall consider the \textit{Simpson's shapes} $X_\B$, $X_\dR$, and $X_\Dol$ and coherent sheaves on them (cf.\ \cite{Porta_Sala_Shapes} for a small compendium of the theory of Simpson's shapes).

\subsubsection{Moduli of local systems}
Let $K \in \cS^{\mathsf{fin}}$ be a finite space. We let $K_\B \in \dSt$ be its \textit{Betti stack}, that is, the constant stack
\begin{align}
K_\B \colon \dSt^{\mathsf{op}} \longrightarrow \cS
\end{align}
associated to $K$ (cf.\ \cite[\S\ref*{shapes-s:Betti-definition}]{Porta_Sala_Shapes}).

 The first result of this section is the following:
 
\begin{proposition}\label{prop:Betti_coh_geometric}
	The derived stack $\bfCoh(K_\B)$ is a geometric derived stack, locally of finite presentation.
\end{proposition}
To prove this statement, we will apply Lurie's representability theorem \cite[Theorem~18.1.0.2]{Lurie_SAG}. We need some preliminary results.

For every $n \ge 0$, set
\begin{align}
	\bfBun^n(K_\B) \coloneqq \bfMap(K_\B, \BGL_n) \ ,
\end{align}
and introduce
\begin{align}
	\bfBun(K_\B) \coloneqq \coprod_{n \ge 0} \bfBun^n(K_\B) \ . 
\end{align}

\begin{lemma}\label{lem:betti_truncation}
	 The truncation of $\bfBun^n(K_\B)$ corresponds to the classical stack of finite-dimensional representations of $\pi_1(K)$.
	\end{lemma}
\begin{proof}
	This follows from \cite[Proposition~\ref*{shapes-prop:Betti_properties}-(\ref*{shapes-prop:Betti_properties:quasicoherent_I}) and Remark~\ref*{shapes-rem:Betti_Vect_truncation}]{Porta_Sala_Shapes}.
\end{proof}

\begin{lemma}\label{lem:betti_decomposition}
	The canonical map
	\begin{align}
		\bfBun(K_\B) \longrightarrow \bfCoh(K_\B)
	\end{align}
	is an equivalence.
\end{lemma}

\begin{proof}
	We can view both $\bfCoh(K_\B)$ and $\bfBun(K_B)$ as full substacks of $\bfPerf(K_\B)$.
	It is therefore enough to show that they coincide as substacks of $\bfPerf(K_\B)$.
	Suppose first that $K$ is discrete.
	Then it is equivalent to a disjoint union of finitely many points, and therefore
	\begin{align}
		K_\B \simeq \Spec(\C)^I \simeq \Spec(\C) \amalg \Spec(\C) \amalg \cdots \amalg \Spec(\C) \ . 
	\end{align}
	In this case
	\begin{align}
		\bfPerf(K_\B) \simeq \bfPerf \times \bfPerf \times \cdots \times \bfPerf \ . 
	\end{align}
	If $S \in \dAff$, an $S$-point of $\bfPerf(K_\B)$ is therefore identified with an object in $\Fun(I, \catPerf(S))$.
	Having tor-amplitude $\le 0$ with respect to $S$ is equivalent of having tor-amplitude $\le 0$ on $S^I$, and therefore the conclusion follows in this case.
	Using the equivalence $S^{k+1} \simeq \Sigma(S^k)$, we deduce that the same statement is true when $K$ is a sphere.
	We now observe that since $K$ is a finite space, we can find a sequence of maps
	\begin{align}
		K_0 = \emptyset \to K_1 \to \cdots \to K_\ell = K \ ,
	\end{align}
	such that each map $K_i \to K_{i+1}$ fits in a pushout diagram
	\begin{align}
		\begin{tikzcd}[ampersand replacement = \&]
		S^{m_i} \arrow{r} \arrow{d} \& \ast \arrow{d} \\
		K_i \arrow{r} \& K_{i+1} 
		\end{tikzcd}\ .
	\end{align}
	The conclusion therefore follows by induction.
\end{proof}

\begin{proof}[{Proof of Proposition~\ref{prop:Betti_coh_geometric}}]
	We can assume without loss of generality that $K$ is connected.
	Let $x \colon \ast \to K$ be a point and let
	\begin{align}
		u_x \colon \Spec(\C) \simeq \ast_\B \longrightarrow K_\B
	\end{align}
	be the induced morphism.
	Then  \cite[Proposition~\ref*{shapes-prop:Betti_properties}-(\ref*{shapes-prop:Betti_properties:flat_atlas})]{Porta_Sala_Shapes}  implies that $u_x$ is a flat effective epimorphism.
	Therefore, Corollary~\ref{cor:coh_inf_cohesive_nilcomplete} implies that $\bfCoh(K_\B)$ is infinitesimally cohesive and nilcomplete.
	Since $\Spec(\C)$ is smooth, Corollary~\ref{cor:coh_global_cotangent_complex_smooth_case} implies that there is a formally étale map
	\begin{align}
		\bfCoh(K_\B) \longrightarrow \bfPerf(K_\B) \ .
	\end{align}
	Since $K$ is a finite space, $\bfPerf(K_\B)$ is a geometric derived stack (cf.\  \cite[\S\ref*{shapes-s:Betti_representability}]{Porta_Sala_Shapes}), and in particular it admits a global cotangent complex. Therefore, so does $\bfCoh(K_\B)$.
	
	We are left to prove that its truncation is geometric. Recall that the classical stack of finite-dimensional representations of $\pi_1(K)$ is geometric (cf.\ e.g. \cite{Simpson_Moduli_II}). Thus, the geometricity of the truncation follows from Lemmas~\ref{lem:betti_truncation} and \ref{lem:betti_decomposition}. Therefore, Lurie's representability theorem \cite[Theorem~18.1.0.2]{Lurie_SAG} applies.
\end{proof}

Now let $X$ be a smooth and proper complex scheme.
Define the stacks
\begin{align}
	\bfCoh_\B(X)\coloneqq \bfCoh(X_\B)\quad\text{and}\quad \bfBun_\B(X)\coloneqq \bfBun(X_\B)\ .
\end{align}
Lemma~\ref{lem:betti_decomposition} supplies a canonical equivalence $\bfCoh_\B(X) \simeq \bfBun_\B(X)$ and Proposition~\ref{prop:Betti_coh_geometric} shows that they are locally geometric and locally of finite presentation. We refer to this stack as the \textit{derived Betti moduli stack of $X$}.
In addition, we shall call
\begin{align}
	\bfBun_\B^n(X)\coloneqq \bfMap(X_\B, \BGL_n)
\end{align}
the derived stack of of $n$-dimensional representations of the fundamental group $\pi_1(X)$ of $X$. The terminology is justified by Lemma~\ref{lem:betti_truncation}.

\begin{example}
	\hfill
	\begin{enumerate}[leftmargin=0.6cm]\itemsep=0.2cm
		\item \label{item:example1} Consider the case $X = \PP^1_\C$.
		We have
		\begin{align}
				\trunc{\bfBun^n_\B(\PP^1_\C)} \simeq \BGL_n \ . 
		\end{align}
		However, $\bfBun^n_\B(\PP^1_\C)$ has an interesting derived structure.
		To see this, let 
		\begin{align}
			x \colon \Spec(\C) \to \bfBun^n_\B(\PP^1_\C)
		\end{align}
		be the map classifying the constant sheaf $\underline{\C}_{\PP^1_\C}^n$.
		This map factors through $\BGL_n$, and it classifies $\C^n \in \mathsf{Mod}_\C = \catQCoh(\Spec(\C))$.
		The tangent complex of $\BGL_n$ at this point is given by $\End_\C(\C^n)[1]$, and in particular it is concentrated in homological degree $-1$.
		On the other hand, \cite[Corollary~\ref*{shapes-cor:Betti_quasicoherent_II}-(\ref*{shapes-cor:Betti_quasicoherent_II-(2)})]{Porta_Sala_Shapes}
		 shows that the tangent complex of $\bfBun^n_\B(\PP^1_\C)$ at $x$ is computed by
		\begin{align}
				 \sfR\Gamma( S^2; \calE\mathsf{nd}(\underline{\C}_{\PP^1_\C}^n))[1] \simeq \End(\C^n)[1] \oplus \End(\C^n)[-1] \ . 
		\end{align}
		In particular, $\bfBun^n_\B(\PP^1_\C)$ is not smooth (although it is lci), and therefore it does not coincide with $\BGL_n$.
		
		\item Assume more generally that $X$ is a smooth projective complex curve. Then $\bfBun_\B^n(X)$ can be obtained as a \textit{quasi-Hamiltonian derived reduction}\footnote{See \cite{Safronov_Hamiltonian} for the notion of Hamiltonian reduction in the derived setting.}. Indeed, let $X'$ be the topological space $X^{\mathsf{top}}$ minus a disk $D$. Then one can easily see that $X'$ deformation retracts onto a wedge of $2g_X$ circles, where $g_X$ is the genus of $X$. We get
		\begin{align}
			\bfBun_\B^n(X)\simeq \bfBun_\B^n(X')\times_{\bfBun_\B^n(S^1)} \bfBun_\B^n(D)\ .
		\end{align}
		Since $\bfBun_\B^n(S^1)\simeq [\mathsf{GL}_n/\mathsf{GL}_n]$ (see, e.g., \cite[Example~3.8]{Calaque_Lectures}), and $\bfBun_\B^n(D)\simeq \bfBun_\B^n(\mathsf{pt})$, we obtain
		\begin{align}
			\bfBun_\B^n(X)\simeq \bfBun_\B^n(X')\times_{[\mathsf{GL}_n/\mathsf{GL}_n]} [\mathsf{pt}/\mathsf{GL}_n]\ .
		\end{align}
		Thus, $\bfBun_\B^n(X)$ is the quasi-Hamiltonian derived reduction of $\bfBun_\B^n(X')$. By further using $\bfBun_\B^n(X')\simeq \bfBun_\B^n(S^1)^{\times\, 2g_X}$, the derived stack $\bfBun_\B^n(X)$ reduces to
		\begin{align}
			\bfBun_\B^n(X)\simeq [\mathsf{GL}_n^{\times\, 2g_X}\times_{\mathsf{GL}_n} \mathsf{pt}/ \mathsf{GL}_n]\ .
		\end{align}
			\end{enumerate}
\end{example}

Generalizing the above example \eqref{item:example1}, we have: 

\begin{corollary} \label{cor:coh_betti_tor_amplitude}
	Let $X$ be a smooth and proper complex scheme of dimension $n$.
	Then the cotangent complex $\LL_{\bfCoh_\B(X)}$ is perfect and has tor-amplitude within $[-1, 2n-1]$.
	In particular, $\bfCoh_\B(X)$ is derived lci\ when $X$ is a curve.
\end{corollary}

\begin{proof}
	Recall that Corollary~\ref{cor:coh_global_cotangent_complex_smooth_case} provides a canonical formally étale map $\bfCoh_\B(X) \to \bfPerf(X_\B)$. Thus, the cotangent complex $\LL_{\bfCoh_\B(X)}$  at a point $x\colon S \to \bfCoh_\B(X)$, where $S \in \dAff$ is a affine derived scheme, is isomorphic to the cotangent complex $\LL_{\bfPerf(X_\B)}$ at the point $\tilde x\colon S \to \bfCoh_\B(X)\to \bfPerf(X_\B)$. 
	
	Via \cite[Proposition~\ref*{shapes-prop:Betti_properties}-(\ref*{shapes-prop:Betti_properties:quasicoherent_I})]{Porta_Sala_Shapes}, we see that $\tilde x$ corresponds to an object $\calL$ in $\Fun(X^{\mathsf{htop}}, \catPerf(S))$.\footnote{Here, $(-)^{\mathsf{htop}} \colon \mathsf{dSch}^{\mathsf{laft}} \longrightarrow \cS$ is the natural functor sending a (derived) $\C$-scheme locally almost of finite type to the underlying homotopy type of its analytification.} 
	On the other hand, \cite[Proposition~\ref*{shapes-prop:Betti_local_systems}]{Porta_Sala_Shapes} allows one to further identify this $\infty$-category with the $\infty$-category of local systems on $X\an$.
	Since $X_\B$ is categorically proper (cf.\ \cite[Proposition~\ref*{shapes-prop:Betti_properties}-(\ref*{shapes-prop:Betti_properties:categorically_proper})]{Porta_Sala_Shapes}), to check tor-amplitude of $\LL_{\bfPerf(X_\B)}$ at the point $\tilde x$ it is enough to assume that $S$ is underived. In addition, since $\calL$ arises from the point $x$, we see that it is discrete.
	Applying the characterization of the derived global sections of any $\calF\in \catQCoh(X_\B)$ in \cite[Corollary~\ref*{shapes-cor:Betti_quasicoherent_II}]{Porta_Sala_Shapes}, we finally deduce that
	\begin{align}
		 \mathbb T_{\bfCoh_\B(X),x} \simeq \sfR \Gamma(X\an; \calE \mathsf{nd}( \calL ))[1]\ . 
	\end{align}
	As this computes the (shifted) singular cohomology of $X\an$ with coefficients in $\calE\mathsf{nd}(\calL)$, the conclusion follows.
\end{proof}

\subsubsection{Moduli of flat bundles} 
Let $X$ be a smooth, proper and connected scheme over $\C$. The \textit{de Rham shape} of $X$ is the derived stack  $X_\dR\in \dSt$ defined by
\begin{align}
X_\dR(S) \coloneqq X(\trunc{S}_{\mathsf{red}})\ ,
\end{align} 
for any $S\in \dAff$ (cf.\ \cite[\S\ref*{shapes-s:dR-definition}]{Porta_Sala_Shapes}). Here, we denote by $T_{\mathsf{red}}$ the underlying reduced scheme of an affine scheme $T\in \mathsf{Aff}$.

Define the stacks
	\begin{align}
		\bfCoh_\dR(X)\coloneqq \bfCoh(X_\dR)\quad\text{and}\quad \bfBun_\dR(X)\coloneqq \bfBun(X_\dR)\ .
	\end{align}
	
	\begin{lemma}\label{lem:dR_decomposition}
		There is a natural equivalence
		\begin{align}
			\bfCoh_\dR(X) \simeq \bfBun_\dR(X) \ .
		\end{align}
		\end{lemma}
	
	\begin{proof}
		First, recall that there exists a canonical map $\lambda_X\colon X\to X_\dR$ (see \cite[\S\ref*{shapes-s:dR-definition}]{Porta_Sala_Shapes}).
		
		We can see both derived stacks as full substacks of $\bfMap(X_\dR, \bfPerf)$.
		Let $S \in \dAff$ and let $x \colon S \to \bfMap(X_\dR, \BGL_n)$.
		Then $x$ classifies a perfect complex $\calF \in \catPerf( X_\dR \times S )$ such that $\calG\coloneqq (\lambda_X \times \id_S)^\ast(\calF) \in \catPerf(X \times S)$ has tor-amplitude $\le 0$ and rank $n$.
		Since the map $X \times S \to S$ is flat, it follows that $\calG$ has tor-amplitude $\le 0$ relative to $S$, and therefore that $x$ determines a point in $\bfCoh_\dR(X)$.
		
		Conversely, let $x \colon S \to \bfCoh_\dR(X)$.
		Let $\calF \in \catPerf(X_\dR \times S)$ be the corresponding perfect complex and let $\calG \coloneqq (\lambda_X \times \id_S)^\ast(\calF)$.
		Then by assumption $\calG$ has tor-amplitude $\le 0$ relative to $S$.
		We wish to show that it has tor-amplitude $\le 0$ on $X \times S$.
		Using Lemma \ref{lem:checking_perfectness_on_fibers}, we see that it is enough to prove that for every geometric point $s \colon \Spec(K) \to S$, the perfect complex $j^\ast(\calG) \in \catPerf(X_K)$ has tor-amplitude $\le 0$.
		Here $X_K \coloneqq \Spec(K) \times X$ and $j \colon X_K \to X$ is the natural morphism.
		Consider the commutative diagram
		\begin{align}
			\begin{tikzcd}[ampersand replacement = \&]
			X_K \arrow{d}{\lambda_{X_K}} \arrow{r}{j} \& X \times S \arrow{d}{\lambda_X \times \id_S} \\
			(X_K)_\dR \arrow{r}{j_\dR} \arrow{r}{j_\dR} \& X_\dR \times S 
			\end{tikzcd}\ .
		\end{align}
		Then
		\begin{align}
			 j_s^\ast \calG \simeq \lambda_{X_K}^\ast j_\dR^\ast \calF \ .
		\end{align}
		We therefore see that $j^\ast \calG$ comes from a $K$-point of $\bfCoh_\dR(X)$.
		By \cite[Theorem~1.4.10]{Hotta_Takeuchi_Tanisaki_D-modules_1995}, $j^\ast \calG$ is a vector bundle on $X$, i.e.\ that it has tor-amplitude $\le 0$.
		The conclusion follows.
	\end{proof}

\begin{proposition}\label{prop:dR_coh_geometric}
		The derived stack $\bfCoh(X_\dR)$ is a geometric derived stack, locally of finite presentation.
\end{proposition}

\begin{proof}
	Consider the canonical map $\lambda_X \colon X\to X_\dR$. Then, \cite[Proposition~\ref*{shapes-prop:de_Rham_properties}-(\ref*{shapes-prop:de_Rham_properties:effective_epi}) and -(\ref*{shapes-prop:de_Rham_properties:flat_atlas})]{Porta_Sala_Shapes} shows that $\lambda_X$ is a flat effective epimorphism. Thus, $\bfCoh(X_\dR)$ fits into the pullback square (cf.\ Lemma~\ref{lem:universally_flat_atlas} and Corollary~\ref{cor:coherent_over_smooth_are_perfect})
	\begin{align}
		\begin{tikzcd}[ampersand replacement = \&]
		\bfCoh(X_\dR) \arrow{r} \arrow{d} \& \bfCoh(X) \arrow{d} \\
		\bfPerf(X_\dR) \arrow{r} \& \bfPerf(X)
		\end{tikzcd}
	\end{align}
	and since $\bfPerf(X)$ and $\bfPerf(X_\dR)$ are geometric (cf.\ \cite[Corollary~3.29]{Toen_Vaquie_Moduli} and  \cite[\S\ref*{shapes-s:dR_representability}]{Porta_Sala_Shapes}, respectively) and $\bfCoh(X)$ is geometric because of Proposition~\ref{prop:coh_geometric}, we conclude that $\bfCoh(X_\dR)$ is geometric as well.
\end{proof}
We shall call $\bfCoh_\dR(X)$ the \textit{derived de Rham moduli stack of $X$}.
	
\subsubsection{Moduli of Higgs sheaves}\label{sec:Dolbeault-moduli-stacks}
Let $X$ be a smooth, proper and connected complex scheme. Let
\begin{align}
	\sfT X \coloneqq \Spec_X( \Sym_{\calO_X}( \mathbb L_X ) )
\end{align}
be the derived tangent bundle to $X$ and let $\widehat{\sfT X} \coloneqq X_\dR \times_{(\sfT X)_\dR} \sfT X$ be the formal completion of $\sfT X$ along the zero section.  
The natural commutative group structure of $\sfT X$ relative to $X$ (seen as an associative one) lifts to $\widehat{\sfT X}$. Thus, we define the \textit{Dolbeault shape $X_\Dol$} of $X$ as the relative classifying stack:
\begin{align}
	X_\Dol\coloneqq \calB_X \widehat{\sfT X}\ ,
\end{align}
while we define the \textit{nilpotent Dolbeault shape $X^\nil_\Dol$} of $X$ as:
\begin{align}
	X_\Dol^\nil \simeq \calB_X  \sfT X\ .
\end{align}

We define
\begin{align}
\bfCoh_\Dol(X)\coloneqq \bfCoh(X_\Dol)\quad \text{and} \quad \bfCoh_\Dol^\nil(X)\coloneqq \bfCoh(X_\Dol^\nil)\ ,
\end{align}
and
\begin{align}
\bfBun_\Dol(X)\coloneqq \bfBun(X_\Dol)\quad \text{and} \quad \bfBun_\Dol^\nil(X)\coloneqq \bfBun(X_\Dol^\nil)\ .
\end{align}
\begin{proposition}\label{prop:Dol_coh_geometric}
	The derived stacks $\bfCoh_\Dol(X), \bfCoh_\Dol^\nil(X), \bfBun_\Dol(X), \bfBun_\Dol^\nil(X)$ are geometric and locally of finite presentation.
\end{proposition}
\begin{proof}
	First, recall that there exist canonical maps $\kappa_X\colon X\to X_\Dol$ and $\kappa_X^\nil\colon X\to X_\Dol^\nil$ (cf.\ \cite[\S\ref*{shapes-s:Dol-definition}]{Porta_Sala_Shapes}).
	
	By  \cite[Lemma~\ref*{shapes-lem:effective-Dol}]{Porta_Sala_Shapes}, $\kappa_X$ and $\kappa_X^\nil$ are flat effective epimorphisms. Thanks to Lemma~\ref{lem:universally_flat_atlas} and Corollary~\ref{cor:coherent_over_smooth_are_perfect}, we are left to check that $\bfPerf(X_\Dol)$ and $\bfPerf(X_\Dol^\nil)$ are geometric and locally of finite presentation (cf.\  \cite[\S\ref*{shapes-s:Dolbeault_representability}]{Porta_Sala_Shapes}).
\end{proof}
We call $\bfCoh_\Dol(X)$ the \textit{derived Dolbeault moduli stack of $X$}, while $\bfCoh_\Dol^\nil(X)$ is the \textit{derived nilpotent Dolbeault moduli stack of $X$}. The truncation $\trunc{\bfCoh_\Dol(X)}$ (resp.\ $\trunc{\bfCoh_\Dol^\nil(X)}$) coincides with the moduli stack of Higgs sheaves (resp.\ nilpotent Higgs sheaves) on $X$.
	
We denote by $\jmath_X\colon \bfCoh_\Dol^\nil(X)\to \bfCoh_\Dol(X)$ and $\jmath_X^{\mathsf{bun}}\colon \bfBun_\Dol^\nil(X)\to \bfBun_\Dol(X)$ the canonical maps induced by $\imath_X\colon X_\Dol \to X_\Dol^\nil$.
	
\begin{remark}\label{rem:naive_Higgs}
	Let $X$ be a smooth and proper complex scheme. 	Define the geometric derived stack
	\begin{align}
		\bfHiggs^{\textrm{na\"if}}(X)\coloneqq \mathbf{T}^\ast[0]\bfCoh(X)= \Spec_{\bfCoh(X)} \big(\Sym( \T_{\bfCoh(X)})\big)\ .
	\end{align}
	There is a natural morphism
	\begin{align}
		\bfCoh_\Dol(X) \longrightarrow \bfHiggs^{\textrm{na\"if}}(X) \ ,
	\end{align}
	which is an equivalence when $X$ is a smooth and projective curve (see, e.g., \cite{Ginzburg_Gaiotto}).
	In higher dimension, this morphism is no longer an equivalence.
	This is due to the fact that in higher dimensions the symmetric algebra and the tensor algebra on $\T_{\bfCoh(X)}$ differ.		
\end{remark}

Let $X$ be a smooth projective complex scheme. For any monic polynomial $p(m)\in \Q[m]$, we set 
\begin{align}
	\bfCoh_\Dol^{p}(X) &\coloneqq \bfPerf(X_\Dol) \times_{\bfPerf(X)} \bfCoh^{p}(X)\ , \\ 
	\bfCoh_\Dol^{\mathsf{nil}, \, p}(X) &\coloneqq \bfPerf(X_\Dol^\mathsf{nil}) \times_{\bfPerf(X)} \bfCoh^{p}(X)\ ,
\end{align}
and
\begin{align}
	\bfBun_\Dol^{p}(X) & \coloneqq \bfPerf(X_\Dol) \times_{\bfPerf(X)} \bfBun^{p}(X)\ , \\ 
	\bfBun_\Dol^{\mathsf{nil}, \, p}(X) & \coloneqq \bfPerf(X_\Dol^\mathsf{nil}) \times_{\bfPerf(X)} \bfBun^{p}(X)\ ,
\end{align}
These are geometric derived stacks locally of finite presentation.
	
As shown by Simpson \cite{Simpson_Moduli_I, Simpson_Moduli_II}, the higher-dimensional analogue of the semistability condition for Higgs bundles on a curve (introduced, e.g., in \cite{Nitsure_Semistable}) is an instance of the Gieseker stability condition for modules over a sheaf of rings of differential operators, when such a sheaf is induced by $\Omega_X^1$ with zero symbol (see \cite[\S2]{Simpson_Moduli_I} for details). This semistability condition is an open property for flat families (cf.\ \cite[Lemma~3.7]{Simpson_Moduli_I}). Thus, there exists an open substack $\trunc{\bfCoh_\Dol^{\mathsf{ss},\, p}(X)}$ of $\trunc{\bfCoh_\Dol^{p}(X)}$ parameterizing families of semi\-sta\-ble Higgs sheaves on $X$ with fixed reduced polynomial $p(m)$; we denote by 
\begin{align}
	\bfCoh_\Dol^{\mathsf{ss},\, p}(X)
\end{align}
its canonical derived enhancement. Similarly, we define $\bfCoh_\Dol^{\mathsf{nil},\, \mathsf{ss},\, p}(X)$, $\bfBun_\Dol^{\mathsf{ss},\, p}(X)$ and $\bfBun_\Dol^{\mathsf{nil},\, \mathsf{ss},\, p}(X)$. These are geometric derived stacks locally of finite presentation.

Finally, for any integer $0\le d\le \dim(X)$, set 
\begin{align}
	\bfCoh_\Dol^{\leqslant d}(X) &\coloneqq \bfPerf(X_\Dol) \times_{\bfPerf(X)} \bfCoh^{\leqslant d}(X)\ ,\\
	\bfCoh_\Dol^{\mathsf{nil},\, \leqslant d}(X) &\coloneqq \bfPerf(X_\Dol^\mathsf{nil}) \times_{\bfPerf(X)} \bfCoh^{\leqslant d}(X)\ .
\end{align}
These are geometric derived stacks locally of finite presentation.
\begin{remark}
	Let $X$ be a smooth projective complex curve and let $\mu\in \Q$ (which corresponds to a choice of a reduced Hilbert polynomial). Then one has $\bfCoh_\Dol^{\mathsf{ss},\, \mu}(X) \simeq \bfBun_\Dol^{\mathsf{ss},\, \mu}(X)$ and $\bfCoh_\Dol^{\mathsf{nil},\, \mathsf{ss},\, \mu}(X) \simeq \bfBun_\Dol^{\mathsf{nil},\,\mathsf{ss},\, \mu}(X)$.
\end{remark}
	
\bigskip\section{Derived moduli stack of extensions of coherent sheaves} \label{s:coh_ext}

Our goal in this section is to introduce and study a derived enhancement of the moduli stack of extensions of coherent sheaves on a proper complex algebraic variety $X$.
As usual, we also deal with the case of Higgs sheaves, vector bundles with flat connection and finite-dimensional representation of the fundamental group of $X$.
We will see in Section~\ref{s:categorifiedHall} that the derived moduli stack of extensions of coherent sheaves is a particular case of a more fundamental construction, known as the Waldhausen construction.
If on the one hand it is a certain property of the Waldhausen construction (namely, its being a $2$-Segal object) the main responsible for the higher associativity of the Hall convolution product at the categorified level, at the same time the analysis carried out in this section of the stack of extensions of coherent sheaves yields a fundamental input for the overall construction.
More specifically, we will show that when $X$ is a surface, certain maps are derived lci, which is the key step in establishing the categorification we seek.

\subsection{Extensions of almost perfect complexes}

Let $\Delta^1$ be the 1-simplex, and define the functor
\begin{align}
	\bfAPerf^{\Delta^1 \times \Delta^1} \colon \dAff\op \longrightarrow \cS
\end{align}
by
\begin{align}
	\bfAPerf^{\Delta^1 \times \Delta^1}(\Spec(A)) \coloneqq \Fun(\Delta^1\times \Delta^1, \catAPerf(A))^\simeq\ .
\end{align}
We let $\bfAPerf^{\mathsf{ext}}$ denote the full substack of $\bfAPerf^{\Delta^1 \times \Delta^1}$ whose $\Spec(A)$-points corresponds to diagrams
\begin{align}
	\begin{tikzcd}[ampersand replacement = \&]
	\calF_1 \arrow{r} \arrow{d} \& \calF_2 \arrow{d} \\ \calF_4 \arrow{r} \& \calF_3
	\end{tikzcd}
\end{align}
in $\catAPerf(A)$ which are pullbacks and where $\calF_4 \simeq 0$.

Observe that the natural map $\bfAPerf^{\mathsf{ext}} \to \bfAPerf^{\Delta^1 \times \Delta^1}$ is representable by Zariski open immersions.
There are three natural morphisms
\begin{align}
	\ev_i \colon \bfAPerf^{\mathsf{ext}} \longrightarrow \bfAPerf \ , \quad i = 1, 2, 3 \ ,
\end{align}
which at the level of functor of points send a fiber sequence
\begin{align}
	\calF_1 \longrightarrow \calF_2 \longrightarrow \calF_3
\end{align}
to $\calF_1$, $\calF_2$ and $\calF_3$, respectively.

Let $Y\in \dSt$ be a derived stack. 
We define
\begin{align}
	\bfAPerf^{\Delta^1 \times \Delta^1}(Y) \coloneqq \bfMap( Y, \bfAPerf^{\Delta^1 \times \Delta^1} ) \ ,
\end{align}
and
\begin{align}
	\bfAPerf^{\mathsf{ext}}(Y) \coloneqq \bfMap( Y, \bfAPerf^{\mathsf{ext}} ) \ .
\end{align}

Once again, the morphism
\begin{align} \label{eq:extension_inside_composable_morphisms}
	\bfAPerf^{\mathsf{ext}}(Y) \longrightarrow \bfAPerf^{\Delta^1 \times \Delta^1}(Y)
\end{align}
is representable by Zariski open immersions.
Moreover, the morphism $\ev_i$ induce a morphism $\bfAPerf^{\mathsf{ext}}(Y) \to \bfAPerf(Y)$, which we still denote $\ev_i$.

Let now $Y\in \dSt$ be a derived stack.
In \S~\ref{s:coh} we introduced the derived moduli stack $\bfCoh(Y)$, parameterizing coherent sheaves on $Y$.
It is equipped with a natural map $\bfCoh(Y) \to \bfAPerf(Y)$.
We define $\bfCoh^{\mathsf{ext}}(Y)$ as the pullback
\begin{align}\label{eq:coh_ext}
	\begin{tikzcd}[ampersand replacement = \&]
	\bfCoh^{\mathsf{ext}}(Y) \arrow{r} \arrow{d} \& \bfAPerf^{\mathsf{ext}}(Y) \arrow{d}{\ev_1 \times \ev_2 \times \ev_3} \\
	\bfCoh(Y)^{\times 3} \arrow{r} \& \bfAPerf(Y)^{\times 3} 
	\end{tikzcd}\ .
\end{align}
and we refer to it as the derived moduli stack of extensions of coherent sheaves.

\begin{remark}
	Let $Y\in \dSt$ be a derived stack.
	Assume there exists a flat effective epimorphism $u \colon U \to Y$ from a geometric derived stack $U$.
	Then Corollary~\ref{cor:formally_etale} implies that the natural map $\bfCoh(Y) \to \bfAPerf(Y)$ is formally \'etale.
	Since formally étale maps are stable under pullback, the very definition of $\bfCoh^{\mathsf{ext}}(Y)$ shows that the natural map $\bfCoh^{\mathrm{ext}}(Y) \to \bfAPerf^{\mathsf{ext}}(Y)$ is formally étale as well.
\end{remark}

Analogous considerations can be made for $\bfPerf(Y)$ instead of $\bfAPerf(Y)$.
In particular, there are well defined stacks $\bfPerf^{\Delta^1 \times \Delta^1}(Y)$ and $\bfPerfext(Y)$.
The commutative diagram
\begin{align} \label{eq:perfext_inside_aperfext}
	\begin{tikzcd}[ampersand replacement = \&]
		\bfPerfext(Y) \arrow{r} \arrow{d} \& \bfAPerf^{\mathsf{ext}}(Y) \arrow{d} \\
		\bfPerf(Y)^{\times 3} \arrow{r} \& \bfAPerf(Y)^{\times 3}
	\end{tikzcd}
\end{align}
is a pullback, and the horizontal arrows are formally étale.
When there is a flat effective epimorphism $u \colon U \to Y$ from a smooth geometric derived stack $U$, Corollary~\ref{cor:coherent_over_smooth_are_perfect} shows that the map $\bfCohext(Y) \to \bfAPerf^{\mathsf{ext}}(Y)$ factors through $\bfPerfext(Y)$, and that the map $\bfCohext(Y) \to \bfPerfext(Y)$ is formally étale as well.

Similarly, we define $\bfBun^{\mathsf{ext}}(Y)$ as the pullback with respect to a diagram of the form \eqref{eq:coh_ext}, with $\bfCoh(Y)^{\times 3}$ replaced with $\bfBun(Y)^{\times 3}$.

\subsection{Explicit computations of cotangent complexes}\label{ss:cotangent}

In this section we carry out the first key computation: we give explicit formulas for the cotangent complexes of the the stack $\bfPerfext(Y)$ and of the map $\ev_3 \times \ev_1 \colon \bfPerfext(Y) \to \bfPerf(Y) \times \bfPerf(Y)$.
We assume throughout this section that $Y$ is a derived stack satisfying the following assumptions:
\begin{enumerate}\itemsep=0.2cm
	\item \label{tor_amplitude} $Y$ has finite local tor-amplitude, see \cite[Definition~\ref*{shapes-def:local_tor_amplitude}]{Porta_Sala_Shapes}.
	
	\item \label{categorical_properness} $Y$ is categorically proper, see \cite[Definition~\ref*{shapes-def:categorical_properness}]{Porta_Sala_Shapes}.
	
	\item \label{effective_epimorphism} There exists an effective epimorphism $u \colon U \to Y$, where $U$ is a quasi-compact derived scheme.
\end{enumerate}
These hypotheses guarantee in particular the following: for every $S \in \dAff$ let
\begin{align}
	p_S \colon Y \times S \longrightarrow S 
\end{align}
be the natural projection.
Then the pullback functor
\begin{align}
	p_S^\ast \colon \catPerf(S) \longrightarrow \catPerf(Y \times S)
\end{align}
admits a left adjoint, which will be denoted $p_{S+}$.
See \cite[Proposition~\ref*{shapes-prop:plus_pushforward}]{Porta_Sala_Shapes} for the construction and the main properties of this functor.

\begin{proposition} \label{prop:cotangent_complex_extensions}
	Let $Y \in \dSt$ be a derived stack satisfying assumptions \eqref{tor_amplitude}, \eqref{categorical_properness}, and \eqref{effective_epimorphism}.
	Then $\bfPerfext(Y)$ admits a global cotangent complex.
	Furthermore, let $S = \Spec(A) \in \dAff$ be an affine derived scheme and let $x \colon S \to \bfPerf^{\mathsf{ext}}(Y)$ be a morphism.
	Write
	\begin{align}
		\calF_1 \longrightarrow \calF_2 \longrightarrow \calF_3
	\end{align}
	for the fiber sequence in $\catPerf(Y \times S)$ classified by $x$.
	Then $x^\ast \LL_{\bfPerf^{\mathsf{ext}}(Y)}[1]$ coincides with the colimit in $\catPerf(S)$ of the diagram
	\begin{align} \label{eq:cotangent_complex_extensions}
		\begin{tikzcd}[ampersand replacement = \&]
			{} \& p_{S+} ( \calF_2 \otimes \calF_3^\vee ) \arrow{r} \arrow{d} \& p_{S+}( \calF_3 \otimes \calF_3^\vee ) \\
			p_{S+}( \calF_1 \otimes \calF_2^\vee ) \arrow{r} \arrow{d} \& p_{S+}( \calF_2 \otimes \calF_2^\vee ) \\
			p_{S+}( \calF_1 \otimes \calF_1^\vee )
		\end{tikzcd}\ .
	\end{align}
\end{proposition}

\begin{proof}
	First of all, we consider the diagram
	\begin{align}
		\begin{tikzcd}[ampersand replacement = \&]
			\bfPerfext(Y) \arrow{r} \arrow{d} \& \bfPerf^{\Delta^1 \times \Delta^1}(Y) \arrow{d} \\
			\bfAPerf^{\mathsf{ext}}(Y) \arrow{r} \& \bfAPerf^{\Delta^1 \times \Delta^1}(Y) 
		\end{tikzcd} \ .
	\end{align}
	Since \eqref{eq:perfext_inside_aperfext} is a pullback, we see that the above square is a pullback.
	In particular, the top horizontal morphism is a Zariski open immersion.
	It is therefore enough to compute the cotangent complex of $\bfPerf^{\Delta^1 \times \Delta^1}(Y)$ at the induced point, which we still denote by $x \colon S \to \bfPerf^{\Delta^1 \times \Delta^1}(Y)$.
	
	Write
	\begin{align}
		F \coloneqq S \times_{\bfPerf^{\Delta^1 \times \Delta^1}(Y)} S \ ,
	\end{align}
	and let $\delta_x \colon S \to F$ be the diagonal morphism induced by $x$. Using \cite[Proposition~\ref*{shapes-prop:inf_cohesive}-(1) and (3)]{Porta_Sala_Shapes} we see that $\bfPerf^{\Delta^1 \times \Delta^1}$ and hence $\bfPerf^{\Delta^1 \times \Delta^1}(Y)\coloneqq \bfMap(Y, \bfPerf^{\Delta^1 \times \Delta^1})$ are infinitesimally cohesive.
	Thus, \cite[Proposition~\ref*{shapes-prop:cotangent_complex_loop}]{Porta_Sala_Shapes} guarantees that
	\begin{align}
		x^\ast \LL_{\bfPerf^{\Delta^1 \times \Delta^1}(Y)} \simeq \delta_x^\ast \LL_{F / S \times S}[-1] \ .
	\end{align}
	We therefore focus on the computation of $\delta_x^\ast \LL_F$.
	Given $(f,g) \colon T = \Spec(B) \to S \times S$, write $f_Y$ and $g_Y$ for the induced morphisms
	\begin{align}
		f_Y \ , \ g_Y \colon Y \times T \longrightarrow Y \times S \ .
	\end{align}
	We can identify $F(T)$ with the $\infty$-groupoid of commutative diagrams
	\begin{align}
		\begin{tikzcd}[ampersand replacement = \&]
			f_Y^\ast \calF_1 \arrow{r} \arrow{d}{\alpha_1} \& f_Y^\ast \calF_2 \arrow{r} \arrow{d}{\alpha_2} \& f_Y^\ast \calF_3 \arrow{d}{\alpha_3} \\
			g_Y^\ast \calF_1 \arrow{r} \& g_Y^\ast \calF_2 \arrow{r} \& g_Y^\ast \calF_3 
		\end{tikzcd}
	\end{align}
	in $\catPerf( Y \times T )$, where $\alpha_1$, $\alpha_2$ and $\alpha_3$ are equivalences.
	In other words, $F(T)$ fits in the following limit diagram:
	\begin{align}
		\begin{tikzcd}[ampersand replacement = \&]
			F(T) \arrow{r} \arrow{d} \& \square \arrow{d} \arrow{r} \& \Map^{\simeq}( f_Y^\ast \calF_3, g_Y^\ast \calF_3 ) \arrow{d} \\
			\square \arrow{r} \arrow{d} \& \Map^{\simeq}(f_Y^\ast \calF_2, g_Y^\ast \calF_2) \arrow{r} \arrow{d} \& \Map(f_Y^\ast \calF_2, g_Y^\ast \calF_3) \\
			\Map^{\simeq}(f_Y^\ast \calF_1,g_Y^\ast \calF_1 ) \arrow{r} \& \Map(f_Y^\ast \calF_1, g_Y^\ast \calF_2) \& \phantom{\Map(f_Y^\ast \calF_2, g_Y^\ast \calF_3)} 
		\end{tikzcd}\ .
	\end{align}
	Here the mapping and isomorphism spaces are taken in $\catPerf( Y \times T )$.
	We have to represent the functor
	\begin{align}
		\mathsf{Der}_F(S; -) \colon \catQCoh(S) \longrightarrow \cS
	\end{align}
	which sends $\calG \in \catQCoh(S)$ to the space
	\begin{align}
		\fib_{\delta_x}\left( F( S[\calG] ) \longrightarrow F(S) \right) \ . 
	\end{align}
	Write $Y_S \coloneqq Y \times S$ and let $p_S \colon Y_S \to S$ be the natural projection, so that
	\begin{align}
		(Y_S)[p_S^\ast \calG] \simeq Y \times S[\calG] \ .
	\end{align}
	Let $d_0 \colon Y_S[p_S^\ast \calG] \to Y_S$ be the zero derivation.
	Observe now that
	\begin{align}
		\{ \id_{\calF_i}\} \times_{\Map(\calF_i, \calF_i)} \Map(d_0^\ast \calF_i, d_0^\ast \calF_i) \simeq \{ \id_{\calF_i} \} \times_{\Map(\calF_i, \calF_i)} \Aut(d_0^\ast \calF_i) \ .
	\end{align}
	We are therefore free to replace $\Aut(d_0^\ast \calF_i)$ by $\Map(d_0^\ast \calF_i, d_0^\ast \calF_i)$ in the diagram computing $F(S[\calG])$.
	Unraveling the definitions, we can thus identify $\mathsf{Der}_F(S;\calG)$ with the pullback diagram
	\begin{align}
		\begin{tikzcd}[ampersand replacement = \&]
			\mathsf{Der}_F(S;\calG) \arrow{r} \arrow{d} \& \square \arrow{r} \arrow{d} \& \Map(\calF_3, \calF_3 \otimes p_S^\ast \calG) \arrow{d} \\
			\square \arrow{r} \arrow{d} \& \Map(\calF_2, \calF_2 \otimes p_S^\ast \calG) \arrow{d} \arrow{r} \& \Map(\calF_2, \calF_3 \otimes p_S^\ast \calG) \\
			\Map(\calF_1, \calF_1 \otimes p_S^\ast \calG) \arrow{r} \& \Map(\calF_1, \calF_2 \otimes p_S^\ast \calG) \& \phantom{\Map(\calF_2, \calF_3 \otimes p_S^\ast \calG)} 
		\end{tikzcd} \ .
	\end{align}
	Since $\calF_1$, $\calF_2$ and $\calF_3$ are perfect, they are dualizable.
	Moreover, \cite[Proposition~\ref*{shapes-prop:plus_pushforward}-(\ref*{shapes-prop:plus_pushforward:existence})]{Porta_Sala_Shapes} guarantees the existence of a left adjoint $p_{S+}$ for $p_S^\ast$.
	We can therefore rewrite the above diagram as
	\begin{align}
		\begin{tikzcd}[ampersand replacement = \&]
			\mathsf{Der}_F(S;\calG) \arrow{r} \arrow{d} \& \square \arrow{r} \arrow{d} \& \Map( p_{S+} (\calF_3 \otimes \calF_3^\vee) , \calG ) \arrow{d} \\
			\square \arrow{r} \arrow{d} \& \Map( p_{S+}( \calF_2 \otimes \calF_2^\vee ), \calG) \arrow{r} \arrow{d} \& \Map( p_{S+}( \calF_2 \otimes \calF_3^\vee ), \calG ) \\
			\Map( p_{S+}( \calF_1 \otimes \calF_1^\vee ), \calG ) \arrow{r} \& \Map( p_{S+}( \calF_1 \otimes \calF_2^\vee ), \calG ) \& \phantom{\Map( p_{S+}( \calF_2 \otimes \calF_3^\vee ), \calG )} 
		\end{tikzcd}
	\end{align}
	where now the mapping spaces are computed in $\catPerf(Y \times S)$.
	Therefore, the Yoneda lemma implies that $\mathsf{Der}_F(S;\calG)$ is representable by the colimit of the diagram \eqref{eq:cotangent_complex_extensions} in $\catPerf(Y \times S)$.
	At this point, \cite[Proposition~\ref*{shapes-prop:plus_pushforward}-(\ref*{shapes-prop:plus_pushforward:base_change})]{Porta_Sala_Shapes} guarantees that $\bfPerfext(Y)$ also admits a global cotangent complex.
\end{proof}

\begin{remark}\label{rem:deformation_morphisms_deformation_extensions}
	There are two natural morphisms
	\begin{align}
	\fib , \mathsf{cofib} \colon \bfPerf^{\Delta^1}(Y) \longrightarrow \bfPerf^{\mathsf{ext}}(Y) \ , 
	\end{align}
	which send a morphism $\beta \colon \calF \to \calG$ to the fiber sequence
	\begin{align}
	\fib(\beta) \longrightarrow \calF \longrightarrow \calG \quad ( \text{resp.\ } \calF \longrightarrow \calG \longrightarrow \mathsf{cofib}(\beta) ) \ . 
	\end{align}
	Applying \cite[Proposition~4.3.2.15]{HTT} twice, we see that these morphisms are equivalences.
	
	Let $y \colon \Spec(A) \to \bfPerf^{\Delta^1\times \Delta^1}(Y)$ be a morphism classifying a diagram
	\begin{align}
	\begin{tikzcd}[ampersand replacement = \&]
	\calF_1 \arrow{r} \arrow{d} \& \calF_2 \arrow{d} \\
	0 \arrow{r} \& \calF_3 
	\end{tikzcd} \ .
	\end{align}
	Let $x \colon \Spec(A) \to \bfPerf^{\Delta^1}(Y)$ be the point corresponding to $\calF_1 \to \calF_2$.
	Then we have a canonical morphism
	\begin{align}
		x^\ast \LL_{\bfPerf^{\Delta^1}(Y)}[1] \longrightarrow y^\ast \LL_{\bfPerf^{\Delta^1 \times \Delta^1}(Y)}[1] \ ,
	\end{align}
	which in general is not an equivalence.
	When the point $y$ factors through the open substack $\bfPerf^{\mathsf{ext}}(Y)$, then the above morphism becomes an equivalence.
\end{remark}

Next, we compute the cotangent complex of $\ev_3\times \ev_1$.
We start with a couple of preliminary considerations:

\begin{definition}[{\cite[Definition~\ref*{shapes-def:linearstack}]{Porta_Sala_Shapes}}]
	Let $Y$ be a derived stack and let $\calF \in \catPerf(Y)$ be a perfect complex on $\calF$.
	The \textit{linear stack\footnote{Note that sometimes in the literature this stack (rather its truncation) is also called \textit{cone stack}. See e.g. \cite[\S2.1]{KV_Hall}.} associated to $\calF$ over $Y$} is the derived stack $\mathbb V_Y(\calF) \in \dSt_{/Y}$ defined as
\begin{align}
	 \mathbb V_Y(\calF) \coloneqq \Spec_Y( \Sym_{\scrO_Y}(\calF) )\  .
\end{align}
\end{definition}

In other words, for every $f \colon S = \Spec(A) \to Y$, one has
\begin{align}
	 \Map_{/Y}( S, \mathbb V_Y(\calF) ) \simeq \Map_{A \Mod}(f^\ast(\calF), A) \ . 
\end{align}

\begin{construction}
	Let $Y \in \dSt$ be a derived stack satisfying assumptions \eqref{tor_amplitude}, \eqref{categorical_properness}, and \eqref{effective_epimorphism}.
	Let
	\begin{align}
		\begin{tikzcd}[ampersand replacement = \&]
		{} \& Y \times \bfPerf(Y) \times \bfPerf(Y) \arrow{dl}[swap]{\mathsf{pr}_1} \arrow{d}{q} \arrow{dr}{\mathsf{pr}_2} \\
		Y \times \bfPerf(Y) \& \bfPerf(Y) \times \bfPerf(Y) \& Y \times \bfPerf(Y)
		\end{tikzcd}
	\end{align}
	be the natural projections.
	Let $\calF \in \catPerf(Y \times \bfPerf(Y))$ be the universal family of perfect complexes on $Y$ and for $i = 1,2$ set
	\begin{align}
		 \calF_i \coloneqq \mathrm{pr}_i^\ast(\calF) \in \catPerf(Y \times \bfPerf(Y) \times \bfPerf(Y)) \ .
	\end{align}
	We set
	\begin{align}\label{eq:complex-extensions}
		 \calG \coloneqq \calHom_{Y \times \bfPerf(Y) \times \bfPerf(Y)}( \calF_1, \calF_2 )[-1]\ . 
	\end{align}
	Using \cite[Corollary~\ref*{shapes-cor:plus_pushforward_general_base}]{Porta_Sala_Shapes} we see that the functor
	\begin{align}
		 q^\ast \colon \catPerf( \bfPerf(Y) \times \bfPerf(Y) ) \longrightarrow \catPerf( Y \times \bfPerf(Y) \times \bfPerf(Y) ) 
	\end{align}
	admits a left adjoint $q_+$.
	We can therefore consider the linear stack
	\begin{align}
		 \mathbb V_{\bfPerf(Y) \times \bfPerf(Y)}( q_+ \calG ) \ , 
	\end{align}
	equipped with its natural projection $\pi \colon \mathbb V_{\bfPerf(Y) \times \bfPerf(Y)}( q_+ \calG ) \to \bfPerf(Y) \times \bfPerf(Y)$.
\end{construction}

We have:

\begin{proposition} \label{prop:stack_extensions_is_relatively_affine}
	Let $Y \in \dSt$ be a derived stack satisfying assumptions  \eqref{tor_amplitude}, \eqref{categorical_properness}, and \eqref{effective_epimorphism}.
	Keeping the notation of the above construction, there is a natural commutative diagram
	\begin{align}
		\begin{tikzcd}[column sep = small, ampersand replacement = \&]
			\bfPerfext(Y) \arrow{dr}[swap]{\ev_3 \times \ev_1} \arrow{rr}{\phi} \& \& \mathbb V_{\bfPerf(Y) \times \bfPerf(Y)}( q_+ \calG ) \arrow{dl}{\pi} \\
			{} \& \bfPerf(Y) \times \bfPerf(Y)
		\end{tikzcd} \ ,
	\end{align}
	where $\phi$ is furthermore an equivalence.
\end{proposition}

\begin{proof}
	For any $S \in \dAff$ and any point $x \colon S \to \bfPerf(Y) \times \bfPerf(Y)$, we can identify the fiber at $x$ of the morphism
	\begin{align}
		\Map_{\dSt}(S, \V_{\bfPerf(Y)\times\bfPerf(Y)}(q_+ \calG) ) \longrightarrow \Map_{\dSt}(S, \bfPerf(Y) \times \bfPerf(Y))
	\end{align}
	with the mapping space
	\begin{align}
		\Map_{\catPerf(S)}(x^\ast q_+(\calG), \scrO_S ) \ .
	\end{align}
	Consider the pullback square
	\begin{align}
		\begin{tikzcd}[ampersand replacement = \&]
		Y \times S \arrow{r}{y} \arrow{d}{q_S} \& Y \times \bfPerf(Y) \times \bfPerf(Y) \arrow{d}{q} \\
		S \arrow{r}{x} \& \bfPerf(Y) \times \bfPerf(Y)
		\end{tikzcd}\ .
	\end{align}
	The base change for the plus pushforward (cf.\ \cite[Corollary~\ref*{shapes-cor:plus_pushforward_general_base}-(\ref*{shapes-cor:plus_pushforward_general_base-(2)})]{Porta_Sala_Shapes}) allows us to rewrite
	\begin{align}
		x^\ast q_+(\calG) \simeq q_{S+} y^\ast(\calG) \ .
	\end{align}
	Therefore, we have
	\begin{align}
		\Map_{\catPerf(S)}( x^\ast q_+(\calG), \scrO_S ) & \simeq \Map_{\catPerf(S)}(q_{S+} y^\ast(\calG), \scrO_S) \\
		& \simeq \Map_{\catPerf(Y \times S)}(y^\ast(\calG), \scrO_{Y\times S} ) \\
		& \simeq \Map_{\catPerf(Y \times S)}( \scrO_{Y \times S}, y^\ast(\calG^\vee) ) \\
		& \simeq \tau_{\ge 0}\Gamma( Y \times S, \calHom_{Y \times S}( y^\ast \calF_2, y^\ast \calF_1 )[1] ) \ .
	\end{align}
	We therefore see that any choice of a fiber sequence
	\begin{align}
		y^\ast \calF_1 \longrightarrow \calF \longrightarrow y^\ast \calF_2
	\end{align}
	in $\catPerf(Y \times S)$ gives rise to a point $S \to \V_{\bfPerf(Y) \times \bfPerf(Y)}( q_+ \calG )$.
	This provides us with a canonical map
	\begin{align}
		\bfPerfext(Y) \longrightarrow \V_{\bfPerf(Y) \times \bfPerf(Y)}( q_+ \calG ) \ , 
	\end{align}
	which induces, for every point $x \colon S \to \bfPerf(Y) \times \bfPerf(Y)$ an equivalence
	\begin{align}
		\Map_{\dSt_{/\bfPerf(Y) \times \bfPerf(Y)}}\left( S, \bfPerfext(Y) \right) \simeq \Map_{\dSt_{\bfPerf(Y) \times \bfPerf(Y)}}\left( S, \V_{\bfPerf(Y) \times \bfPerf(Y)}( q_+ \calG )\right) \ .
	\end{align}
	The conclusion follows.
\end{proof}

\begin{corollary} \label{cor:cotangent_complex_extremal_projections}
	Let $Y$ be a derived stack satisfying the same assumptions of Proposition~\ref{prop:stack_extensions_is_relatively_affine}.
	Then the cotangent complex of the map
	\begin{align}
		\ev_3 \times \ev_1 \colon \bfPerfext(Y) \longrightarrow \bfPerf(Y) \times \bfPerf(Y)
	\end{align}
	is computed as
	\begin{align}
		(\ev_3\times \ev_1)^\ast \left( q_+ \left(\calHom_{Y \times \bfPerf(Y) \times \bfPerf(Y)}( \calF_1, \calF_2 )[-1] \right)  \right)\ .
	\end{align}
\end{corollary}

\begin{proof}
	This is an immediate from Proposition~\ref{prop:stack_extensions_is_relatively_affine} and \cite[Proposition~7.4.3.14]{Lurie_Higher_algebra}.
\end{proof}

\subsection{Extensions of coherent sheaves on schemes}

We now specify the constructions of the previous section to the case where $Y$ is a smooth and proper complex scheme.
Assumptions \eqref{tor_amplitude}, \eqref{categorical_properness}, and \eqref{effective_epimorphism} are satisfied in this case, see \cite[Example~\ref*{shapes-eg:various_notions_of_properness}]{Porta_Sala_Shapes}.
Our goal is to provide estimates on the tor-amplitude of the cotangent complexes of $\bfCohext(Y)$ and of the map $\ev_3\times \ev_1 \colon \bfCohext(Y) \to \bfCoh(Y) \times \bfCoh(Y)$:

\begin{proposition} \label{prop:extension_coherent_curve_smooth}
	Let $X$ be a smooth and proper complex scheme of dimension $n$.
	Then the cotangent complex $\LL_{\bfCoh^{\mathsf{ext}}(X)}$ is perfect and has tor-amplitude within $[-1, n-1]$.
	In particular, $\bfCoh^{\mathsf{ext}}(X)$ is smooth when $X$ is a curve and derived lci when $X$ is a surface.
\end{proposition}

\begin{remark}
	Notice that $\bfPerf^{\mathsf{ext}}(X)$ is not smooth, even if $X$ is a smooth projective complex curve.
\end{remark}

\begin{proof}[Proof of Proposition~\ref{prop:extension_coherent_curve_smooth}]
	Let $\Spec(A) \in \dAff$ and let $x \colon \Spec(A) \to \bfCoh^{\mathsf{ext}}(X)$ be a point.
	We have to check that $x^\ast \T_{\bfCoh^{\mathsf{ext}}(X)}$ is perfect and in tor-amplitude $[1-n, 1]$.
	Since the map $\bfCoh^{\mathsf{ext}}(X) \to \bfPerf^{\mathsf{ext}}(X)$ is formally \'etale, we can use Proposition \ref{prop:cotangent_complex_extensions} to compute the cotangent complex, and hence the tangent one.
	Let
	\begin{align} \label{eq:point_stack_extensions}
		\calF_1 \longrightarrow \calF_2 \longrightarrow \calF_3
	\end{align}
	be the fiber sequence in $\catPerf(X \times \Spec(A))$ corresponding to the point $x$.
	Let $p \colon X \times \Spec(A) \to \Spec(A)$ be the canonical projection.
	Using Remark \ref{rem:deformation_morphisms_deformation_extensions} we see that $x^\ast \T_{\bfCoh^{\mathsf{ext}}(X)}$ fits in the pullback diagram
	\begin{align}
		\begin{tikzcd}[ampersand replacement = \&]
		x^\ast \T_{\bfCoh^{\mathsf{ext}}(X)} \arrow{r} \arrow{d} \& p_\ast( \calF_2^\vee \otimes \calF_2 )[1] \arrow{d} \\
		p_\ast( \calF_1^\vee \otimes \calF_1 )[1] \arrow{r}\& p_\ast( \calF_1^\vee \otimes \calF_2 )[1] 
		\end{tikzcd}\ .
	\end{align}
	Since $X$ is smooth and proper, $p_\ast$ preserves perfect complexes.
	Therefore, $x^\ast \T_{\bfCoh^{\mathsf{ext}}(X)}$ is perfect.
	
	In order to check that it has tor-amplitude within $[1-n,1]$, it is sufficient to check that its pullback to $\Spec(\pi_0(A))$ has tor-amplitude within $[1-n,1]$.
	In other words, we can suppose from the very beginning that $A$ is discrete.
	In this case, $\calF_1$, $\calF_2$ and $\calF_3$ are discrete as well and the map $\calF_1 \to \calF_2$ is a monomorphism.
	Since $X$ is an $n$-dimensional scheme, the functor $p_\ast$ has cohomological dimension $n$.
	It is therefore sufficient to check that $\pi_{-n}( x^\ast \T_{\bfCoh^{\mathsf{ext}}(X)} ) = 0$.
	We have a long exact sequence
	\begin{align}
		\calExt^n_p(\calF_1, \calF_1) \oplus \calExt^n_p(\calF_2, \calF_2) & \to \calExt^n_p( \calF_1, \calF_2 )\\
 		\to & \pi_{-n}( x^\ast \T_{\bfCoh^{\mathsf{ext}}(X)} ) \to \calExt^{n+1}_p( \calF_1, \calF_1 ) \oplus \calExt^{n+1}_p( \calF_2, \calF_2 ) \ .
	\end{align}
	
	By using Grothendieck--Serre duality (as in the second part of the proof of Corollary~\ref{cor:coh_curve_smooth}), one can show that
	\begin{align}
	\calExt^{n+1}_p( \calF_1, \calF_1 ) =0\quad\text{and}\quad \calExt^{n+1}_p(\calF_2, \calF_2) = 0\ .
	\end{align}

	We are thus left to check that the map
	\begin{align}
		\calExt^n_p(\calF_1, \calF_1) \oplus \calExt^n_p(\calF_2, \calF_2) \longrightarrow \calExt^n_p(\calF_1, \calF_2)
	\end{align}
	is surjective.
	It is enough to prove that
	\begin{align}
		\calExt^n_p(\calF_2, \calF_2) \longrightarrow \calExt^n_p(\calF_1, \calF_2)
	\end{align}
	is surjective.
	We have a long exact sequence
	\begin{align}
		 \calExt^n_p(\calF_2, \calF_2) \longrightarrow \calExt^n_p(\calF_1, \calF_2) \longrightarrow \calExt^{n+1}_p(\calF_3, \calF_2) \ . 
	\end{align}
	The same argument as above shows that $\calExt^{n+1}_p(\calF_3, \calF_2) = 0$.
	The proof is therefore complete.
\end{proof}

\begin{proposition} \label{prop:extremal_projections_scheme}
	Let $X$ be a smooth and proper complex scheme of dimension $n$.
	Then the relative cotangent complex of the map
	\begin{align} \label{eq:extreme_projections_extension_coh}
		\ev_3 \times \ev_1 \colon \bfCohext(X) \longrightarrow \bfCoh(X) \times \bfCoh(X)
	\end{align}
	is perfect and has tor-amplitude within $[-1,n-1]$.
	In particular, it is smooth when $X$ is a curve and derived lci when $X$ is a surface.
\end{proposition}

\begin{rem}
	When $X$ is a curve, Corollary~\ref{cor:coh_curve_smooth} and Proposition~\ref{prop:extension_coherent_curve_smooth} imply that $\bfCohext(X)$ and $\bfCoh(X)$ are smooth.
	This immediately implies that $\ev_3 \times \ev_1$ is derived lci, hence
	the above corollary improves this result.\hfill $\triangle$
\end{rem}

\begin{proof}[Proof of Proposition~\ref{prop:extremal_projections_scheme}]
	Let $S \in \dAff$ and let $x \colon S \to \bfPerfext(X)$ be a point classifying a fiber sequence
	\begin{align} \label{eq:generic_extension}
		\calF_1 \longrightarrow \calF_2 \longrightarrow \calF_3
	\end{align}
	in $\catPerf(X \times S)$.
	If $\calF_1$ and $\calF_3$ have tor-amplitude $\le 0$ relative to $S$, then the same goes for $\calF_2$.
	This implies that the diagram
	\begin{align}
		\begin{tikzcd}[ampersand replacement = \&]
		\bfCohext(X) \arrow{r} \arrow{d}{\ev_3 \times \ev_1} \& \bfPerfext(X) \arrow{d}{\ev_3 \times \ev_1} \\
		\bfCoh(X) \times \bfCoh(X) \arrow{r} \& \bfPerf(X) \times \bfPerf(X)
		\end{tikzcd}
	\end{align}
	is a pullback square.
	Smooth and proper schemes are categorically proper and have finite local tor-amplitude, see \cite[Example~\ref*{shapes-eg:various_notions_of_properness}]{Porta_Sala_Shapes}.
	Therefore the assumptions of Proposition~\ref{prop:stack_extensions_is_relatively_affine} are satisfied.
	Since the horizontal maps in the above diagram are formally \'etale, we can therefore use Corollary~\ref{cor:cotangent_complex_extremal_projections} to compute the relative cotangent complex of the morphism \eqref{eq:extreme_projections_extension_coh}.
	This immediately implies that this relative cotangent complex is perfect, and we are left to prove that it has tor-amplitude within $[-1,n-1]$.
	For this reason, it is enough to prove that for any (underived) affine scheme $S \in \mathsf{Aff}$ and any point $x \colon S \to \bfCohext(X)$, the perfect complex $x^\ast \LL_{\ev_3\times \ev_1}$ has tor-amplitude within $[-1,n-1]$.
	Let $\calF_1 \to \calF_2 \to \calF_3$ be the extension classified by $x$ and let $q_S \colon Y \times S \to S$ be the canonical projection.
	Base change for the plus pushforward  (see \cite[Proposition~\ref*{shapes-prop:plus_pushforward}-(\ref*{shapes-prop:plus_pushforward:base_change})]{Porta_Sala_Shapes}) reduces our task to computing the tor-amplitude of
	\begin{align}
		q_{S+}( \calHom_{X \times S}( \calF_1, \calF_3 )[-1] ) \simeq ( q_{S\ast}( \calHom_{X \times S}(\calF_3, \calF_1)[1] ) )^\vee \ . 
	\end{align}
	Moreover, since $S$ is arbitrary, it is enough to prove that
	\begin{align}
		 \pi_i( q_{S\ast}( \calHom_{X \times S}( \calF_3, \calF_1 )[1] ) ) \simeq 0
	\end{align}
	for $i \le 1 - n$.
	However
	\begin{align}
		 \pi_i( q_{S\ast}( \calHom_{X \times S}( \calF_3, \calF_1 )[1] ) ) \simeq \calExt_q^{-i+1}( \calF_3, \calF_1 ) \ . 
	\end{align}
	Since $S$ is underived, $\calF_1$ and $\calF_3$ belong to $\catQCoh^\heartsuit(X \times S)$.
	Since $X$ has dimension $n$, it follows that $\calExt_q^j( \calF_3, \calF_1 ) \simeq 0$ for $j > n$.
	The conclusion follows.
\end{proof}

\begin{corollary} 
	Let $X$ be a smooth and proper complex scheme of dimension $n$. Then the relative cotangent complex of the map
	\begin{align}
	\ev_3 \times \ev_1 \colon \bfBun^{\mathsf{ext}}(X) \longrightarrow \bfBun(X) \times \bfBun(X)
	\end{align}
	is perfect and has tor-amplitude within $[-1,n-1]$.
\end{corollary}

\begin{proof} 
	The assertion follows by noticing that the diagram
	\begin{align}
	\begin{tikzcd}[ampersand replacement = \&]
	\bfBun^{\mathsf{ext}}(X) \arrow{r} \arrow{d}{\ev_3 \times \ev_1} \& \bfCoh^{\mathsf{ext}}(X) \arrow{d}{\ev_3 \times \ev_1} \\
	\bfBun(X) \times \bfBun(X) \arrow{r} \& \bfCoh(X) \times \bfCoh(X)
	\end{tikzcd}
	\end{align}
	is a pullback square.
\end{proof}

\subsection{Extensions of coherent sheaves on Simpson's shapes}\label{ss:ext-shapes}

In this section, we carry out an analysis similar to the one of the previous section in the case where $Y$ is one of the Simpson's shapes $X_\B, X_\dR$, and $X_\Dol$, where $X$ is a smooth and proper scheme.

\subsubsection{Betti shape} 
	Let $K$ be a finite connected space. By  \cite[Proposition~\ref*{shapes-prop:Betti_properties}-(\ref*{shapes-prop:Betti_properties:categorically_proper})]{Porta_Sala_Shapes}, $K_\B$ is categorically proper and it has finite local tor-amplitude. In addition, by \cite[Proposition~\ref*{shapes-prop:Betti_properties}-(\ref*{shapes-prop:Betti_properties:flat_atlas})]{Porta_Sala_Shapes}, the map $\Spec(\C) \simeq \ast_\B \longrightarrow X_\B $ is an effective epimorphism. Thus, the assumptions of  Corollary~\ref{cor:cotangent_complex_extremal_projections} are satisfied. Therefore, the relative cotangent complex of the map
	\begin{align} \label{eq:extremal_projections_Betti}
		\ev_3 \times \ev_1 \colon \bfCohext(K_\B) \longrightarrow \bfCoh(K_\B) \times \bfCoh(K_\B)
	\end{align}
	at a point $S \to \bfCohext(K_\B)$ classifying an extension $\calF_1 \to \calF \to \calF_2$ in $\catPerf(K_\B \times S)$ is computed by the pullback along the projection $S \times_{\bfCoh(K_\B) \times \bfCoh(K_\B)} \bfCohext(K_\B) \to S$ of
	\begin{align}
		q_{S+} \left( \calHom_{K_\B \times S}( \calF_2, \calF_1 )[-1] \right) \ .
	\end{align}
	Here $q_S \colon K_\B \times S \to S$ is the natural projection.
	In particular, we obtain:
	
	\begin{proposition}
		Suppose that $K_\B$ has cohomological dimension $\le m$.
		The relative cotangent complex of the map \eqref{eq:extremal_projections_Betti} has tor-amplitude within $[-1,m-1]$.
		Furthermore, if $K$ is the space underlying a complex scheme $X$ of complex dimension $n$, then we can take $m = 2n$.
	\end{proposition}
	
	\begin{proof}
		It is enough to prove that for every unaffine derived scheme $S \in \mathsf{Aff}$ and every point $x \colon S \to \bfCohext(K_\B, u)$ classifying an extension $\calF_1 \to \calF \to \calF_2$ in $\catPerf(K_\B \times S)$ of perfect complexes of tor-amplitude $\le 0$ relative to $S$, the complex $q_{S+} (\calHom_{K_\B \times S}( \calF_2, \calF_1 )[-1])$ is contained in cohomological amplitude $[-1,m-1]$.
		Unraveling the definitions, this is equivalent to check that the complex $q_{S\ast}( \calHom_{K_\B \times S}(\calF_1, \calF_2) )$ has cohomological amplitude within $[-m,0]$.
		The latter follows from the assumption on the cohomological dimension of $K_\B$ and from Lemma~\ref{lem:betti_decomposition}.
	\end{proof}
	
	Now let $X$ be a smooth and proper complex scheme. Define the stacks
	\begin{align}
	\bfCohext_\B(X)\coloneqq \bfCohext(X^{\mathsf{top}}_\B)\quad\text{and}\quad \bfBun^{\mathsf{ext}}_\B(X)\coloneqq \bfBun^{\mathsf{ext}}(X^{\mathsf{top}}_\B)\ .
	\end{align}
	These stacks are geometric and locally of finite presentation since the stacks $\bfPerfext(X^{\mathsf{top}}_\B)$ and $\bfPerf^{\Delta^1\times \Delta^1}(X^{\mathsf{top}}_\B)$ are so. By using similar arguments as in the proof of Proposition~\ref{prop:extension_coherent_curve_smooth}, we find that the cotangent complex $\LL_{\bfCohext_\B(X)}$ is perfect and has tor-amplitude within $[-1, 2n-1]$. Finally, by Lemma~\ref{lem:betti_decomposition} we get $\bfCohext_\B(X)\simeq  \bfBun^{\mathsf{ext}}_\B(X)$. 
	
	\begin{corollary}\label{cor:extremal_projections_betti}
		If $X$ is a smooth projective complex curve and $K \coloneqq X^{\mathsf{top}}$, then the map \eqref{eq:extremal_projections_Betti} is derived locally complete intersection.
	\end{corollary}
	
\subsubsection{De Rham shape}		
Let $X$ be a smooth and proper complex scheme of dimension $n$. First note that, by \cite[Proposition~\ref*{shapes-prop:de_Rham_properties}-(\ref*{shapes-prop:de_Rham_properties:categorically_proper})]{Porta_Sala_Shapes}, $X_\dR$ is categorically proper and it has finite local tor-amplitude. Moreover, by \cite[Proposition~\ref*{shapes-prop:de_Rham_properties}-(\ref*{shapes-prop:de_Rham_properties:effective_epi})]{Porta_Sala_Shapes}, the canonical map $\lambda_X\colon X\to X_\dR$ is an effective epimorphism.
	
Define the stacks
\begin{align}
	\bfCohext_\dR(X)\coloneqq \bfCohext(X_\dR)\quad\text{and}\quad \bfBun^{\mathsf{ext}}_\dR(X)\coloneqq \bfBun^{\mathsf{ext}}(X_\dR)\ .
\end{align}
These stacks are geometric and locally of finite presentation since the stacks $\bfPerfext(X_\dR)$ and $\bfPerf^{\Delta^1\times \Delta^1}(X_\dR)$ are so. 
	
Since $X_\dR$ satisfies the assumptions of Proposition~\ref{prop:cotangent_complex_extensions}, we may use similar arguments as in the proof of Proposition~\ref{prop:extension_coherent_curve_smooth}, and we get that the cotangent complex $\LL_{\bfCohext_\dR(X)}$ is perfect and has tor-amplitude within $[-1, 2n-1]$.	Finally, by Lemma~\ref{lem:dR_decomposition} we get $\bfCohext_\dR(X)\simeq \bfBun^{\mathsf{ext}}_\dR(X)$.
	
As in the case of the Betti shape, we deduce that the relative cotangent complex of the map
\begin{align} \label{eq:extremal_projections_de_Rham}
	\ev_3 \times \ev_1 \colon \bfCohext_\dR(X) \longrightarrow \bfCoh_\dR(X) \times \bfCoh_\dR(X)
\end{align}
at a point $x \colon S\to \bfCohext(X_\dR)$ classifying an extension $\calF_1 \to \calF \to \calF_2$ in $\catPerf(X_\dR \times S)$ is computed by the pullback along the projection $S \times_{\bfCoh_\dR(X) \times \bfCoh_\dR(X)} \bfCohext_\dR(X) \to S$ of
\begin{align}
	q_{S+} \left( \calHom_{X_\dR \times S}( \calF_2, \calF_1)[-1] \right) \ .
\end{align}
Here $q_S \colon X_\dR \times S \to S$ is the natural projection.
In particular, we obtain:
\begin{proposition}
	Suppose that $X$ is connected and of dimension $n$.
	Then the relative cotangent complex of the map \eqref{eq:extremal_projections_de_Rham} has tor-amplitude within $[-1, 2n-1]$.
\end{proposition}
	
\begin{proof}
	It is enough to prove that for every unaffine derived scheme $S \in \mathsf{Aff}$ and every point $x \colon S \to \bfCohext_\dR(X)$ classifying an extension $\calF_1 \to \calF \to \calF_2$ in $\catPerf_\dR(X \times S)$ of perfect complexes of tor-amplitude $\le 0$ relative to $S$, the complex $q_{S+} (\calHom_{X_\dR \times S}( \calF_1, \calF_2 )[-1])$ has cohomological amplitude within $[-1,2n-1]$.
	Unraveling the definitions, this is equivalent to check that the complex $q_{S\ast}( \calHom_{X_\dR \times S}(\calF_2, \calF_1) )$ has cohomological amplitude within $[-2n,0]$.
	In other words, we have to check that
	\begin{align}
		\Ext^i_{X_\dR \times S}( \calF_2, \calF_1 ) = 0
	\end{align}
	for $i > 2n$.
	This follows from \cite[Theorem~2.6.11]{Hotta_Takeuchi_Tanisaki_D-modules_1995} and \cite[\S11]{Bernstein_D-modules}.
\end{proof}
	
\begin{corollary}\label{cor:dR-lci}
	If $X$ is a smooth projective complex curve, then the map \eqref{eq:extremal_projections_de_Rham} is derived locally complete intersection.
\end{corollary}
	
\subsubsection{Dolbeault shape}
	
Let $X$ be a smooth and proper complex scheme. By  \cite[Lemmas~\ref*{shapes-lem:Dol_proper} and \ref*{shapes-lem:nilDol_proper}]{Porta_Sala_Shapes}, $X_\Dol$ and $X_\Dol^\nil$ are categorically proper and they have finite local tor-amplitude. Moreover, by  \cite[Lemma~\ref*{shapes-lem:effective-Dol}]{Porta_Sala_Shapes}, the canonical maps $\kappa_X\colon X\to X_\Dol$ and $\kappa_X^\nil\colon X\to X_\Dol^\nil$ are effective epimorphisms.
	
Define the stacks
\begin{align}
	\bfCohext_\Dol(X)\coloneqq \bfCohext(X_\Dol)\quad &\text{and}\quad \bfBun^{\mathsf{ext}}_\Dol(X)\coloneqq \bfBun^{\mathsf{ext}}(X_\Dol)\ , \\
	\bfCoh^{\nil, \, \mathsf{ext}}_\Dol(X)\coloneqq \bfCohext(X_\Dol^\nil)\quad &\text{and}\quad \bfBun^{\nil, \, \mathsf{ext}}_\Dol(X)\coloneqq \bfBun^{\mathsf{ext}}(X_\Dol^\nil)\ .
\end{align}
These stacks are geometric and locally of finite presentation since $\bfPerfext(X_\Dol)$, $\bfPerf^{\Delta^1\times \Delta^1}(X_\Dol)$ and $\bfPerfext(X_\Dol^\nil)$, $\bfPerf^{\Delta^1\times \Delta^1}(X_\Dol^\nil)$ are so. 

Since $X_\Dol$ and $X_\Dol^\nil$ satisfy the assumptions of Proposition~\ref{prop:cotangent_complex_extensions}, we may use similar arguments as in the proof of Proposition~\ref{prop:extension_coherent_curve_smooth}, and we get that the cotangent complexes $\LL_{\bfCohext_\Dol(X)}$ and $\LL_{\bfCoh^{\nil, \, \mathsf{ext}}_\Dol(X)}$ are perfect and have tor-amplitude within $[-1, 2n-1]$.	
	
As in the case of the Betti and de Rham shapes, we thus deduce that the relative cotangent complex of the map
\begin{align} \label{eq:extremal_projections_Dolbeault}
	\ev_3 \times \ev_1 \colon \bfCohext_\Dol(X) \longrightarrow \bfCoh_\Dol(X) \times \bfCoh_\Dol(X)
\end{align}
at a point $x \colon S\to \bfCohext_\Dol(X)$ classifying an extension $\calF_1 \to \calF \to \calF_2$ in $\catPerf(X_\Dol \times S)$ is computed by the pullback along the projection $S \times_{\bfCoh_\Dol(X) \times \bfCoh_\Dol(X)} \bfCohext_\Dol(X) \to S$ of
\begin{align}
	q_{S+} \left( \calHom_{X_\Dol \times S}( \calF_1, \calF_2)[-1] \right) \ .
\end{align}
Here $q_S \colon X_\Dol \times S \to S$ is the natural projection.
In particular, we obtain:	
\begin{proposition}
	Suppose that $X$ is connected and of dimension $n$.
	Then the relative cotangent complex of the map \eqref{eq:extremal_projections_Dolbeault} has tor-amplitude within $[-1,2n-1]$.
\end{proposition}

\begin{proof}
	It is enough to check that for every unaffine derived scheme $S \in \mathsf{Aff}$ and every point $x \colon S \to \bfCohext_\Dol(X)$ classifying an extension $\calF_1 \to \calF \to \calF_2$ in $\catPerf(X_\Dol \times S)$ of perfect complexes of tor-amplitude $\le 0$ relative to $S$, the complex $q_{S+}(\calHom_{X_\Dol \times S}( \calF_1, \calF_2 )[-1])$ has cohomological amplitude within $[-1,2n-1]$.
	Unraveling the definitions, this is equivalent to check that the complex $q_{S\ast} \calHom_{X_\Dol \times S}( \calF_2, \calF_1 )$ has cohomological amplitude within $[-2n,0]$.
	In other words, we have to check that
	\begin{align}
		\Ext^i_{X_\Dol \times S}( \calF_2, \calF_1 ) = 0 
	\end{align}
	for $i > 2n$.
	This follows from the BNR correspondence \cite[Lemma~6.8]{Simpson_Moduli_II} (cf.\ also \cite[\S4]{Gothen_King_Quiver} and \cite[\S2.3]{Sala_Schiffmann}).
	\end{proof}
	
\begin{corollary}\label{cor:Dol-lci}
	If $X$ is a smooth projective complex curve, then the map \eqref{eq:extremal_projections_Dolbeault} is derived locally complete intersection.
\end{corollary}

\bigskip\section{Two-dimensional categorified Hall algebras}\label{s:categorifiedHall}

\subsection{Convolution algebra structure for the stack of perfect complexes} \label{ss:convolution_structure}

Most of the results in this section are due to T.\ Dyckerhoff and M.\ Kapranov \cite{Dyckerhoff_Kapranov_Higher_Segal}.
For the convenience of the reader we briefly recall their constructions.

Let
\begin{align}
	\sfT \coloneqq \Hom_{\mathbf \Delta}( [1], - ) \colon \mathbf \Delta \longrightarrow \Cat_\infty\ ,
\end{align}
where $\mathbf \Delta$ is the simplicial category.
We write $\sfT_n$ instead of $\sfT([n])$.
Given any $\C$-linear stable $\infty$-category $\cC$, we let
\begin{align}
	\calS_\bullet \cC \colon \mathbf \Delta\op \longrightarrow \Cat_\infty
\end{align}
be the subfunctor of $\Fun( T(-), \cC )$ that sends $[n]$ to the full subcategory $\calS_n \cC$ of $\cC^{\sfT_n} \coloneqq \Fun( \sfT_n, \cC )$ spanned by those functors $F \colon \sfT_n \to \cC$ satisfying the following two conditions:
\begin{enumerate}\itemsep0.2cm
	\item $F(i,i) \simeq 0$ for every $0 \le i \le n$;
	\item for every $0 \le i , j \le n-1$, $i \le j - 1$, the square
	\begin{align}
		\begin{tikzcd}[ampersand replacement = \&]
			F(i,j) \arrow{r} \arrow{d} \& F(i+1, j) \arrow{d} \\
			F(i,j+1) \arrow{r} \& F(i+1, j+1)
		\end{tikzcd}
	\end{align}
	is a pullback in $\cC$.
\end{enumerate}
We refer to $\calS_\bullet \cC$ as the $\infty$-categorical Waldhausen construction on $\cC$.
It follows from \cite[Theorem~7.3.3]{Dyckerhoff_Kapranov_Higher_Segal} that $\calS_\bullet \cC$ is a \textit{$2$-Segal object} in $\Cat_\infty$.
Consider the functor
\begin{align}
	\dAff\op \times \mathbf \Delta\op \longrightarrow \Cat_\infty
\end{align}
defined by sending $(\Spec(A),[n])$ to $\calS_n \catAPerf(A)$.
We denote by
\begin{align}
	\calS_\bullet \catAPerf \colon \mathbf \Delta\op \longrightarrow \Fun(\dAff\op, \Cat_\infty)
\end{align}
the corresponding functor.
Since limits are computed objectwise in $\Fun(\dAff\op, \Cat_\infty)$, we see that $\calS_\bullet \catAPerf$ is a $2$-Segal object in $\Fun(\dAff\op, \Cat_\infty)$.
The maximal $\infty$-groupoid functor $(-)^\simeq \colon \Cat_\infty \to \scrS$ is a right adjoint, and in particular it commutes with limits.
We let
\begin{align}
	\calS_\bullet \bfAPerf \colon \mathbf \Delta\op \longrightarrow \dSt
\end{align}
be the functor obtained by $\calS_\bullet \catAPerf$ by applying the maximal $\infty$-groupoid functor.
The above considerations show that $\calS_\bullet \bfAPerf$ is a $2$-Segal object in $\dSt$.

Let now $X$ be a derived stack.
The functor $\bfMap(X,-) \colon \dSt \to \dSt$ commutes with limits, and therefore the simplicial derived stack
\begin{align}
	 \calS_\bullet \bfAPerf(X) \coloneqq \bfMap(X, \calS_\bullet \bfAPerf) 
\end{align}
is again a $2$-Segal object in $\dSt$.
The same construction can be performed using $\bfPerf$ instead of $\bfAPerf$: thus we obtain $2$-Segal objects $\calS_\bullet \bfPerf$ and $\calS_\bullet \bfPerf(X)$ in $\dSt$.

As in Section~\ref{s:coh_ext}, we extract a full substack of coherent sheaves as follows.
For every $n \ge 0$, let $N \coloneqq \frac{n(n+1)}{2}$.
Evaluation at $(i,j) \in T_n$ induces a well defined map $\calS_n \bfAPerf(X) \to \bfAPerf(X)^N$.
We define $\calS_n \bfCoh(X)$ by the fiber product
\begin{align}
\begin{tikzcd}[ampersand replacement = \&]
	\calS_n \bfCoh(X) \arrow{r} \arrow{d} \& \calS_n \bfAPerf(X) \arrow{d} \\
	\bfCoh(X)^N \arrow{r} \& \bfAPerf(X)^N 
\end{tikzcd} \ .
\end{align}
Notice that for $n = 2$ this construction yields a canonical identification $\calS_2 \bfCoh(X) \simeq \bfCohext(X)$.
We will prove

\begin{lemma} \label{lem:coh_2_Segal_geometric}
	The simplicial object $\calS_\bullet \bfCoh(X)$ is a $2$-Segal object. 
\end{lemma}

\begin{proof}
	Using \cite[Proposition~2.3.2(3)]{Dyckerhoff_Kapranov_Higher_Segal}, it remains to check that for every $n \ge 3$ and every $0 \le i < j \le n$, the natural morphism 
	\begin{align}  
	\calS_n \bfCoh(X) \longrightarrow \calS_{n-j+i+1} \bfCoh(X) \times_{\calS_1 \bfCoh(X)} \calS_{j-i} \bfCoh(X) 
	\end{align} 
	is an equivalence.
	Here the morphism is induced by the maps $[n-j+i+1] \to [n]$ and $[j-i] \to [n]$ corresponding to the inclusions
	\begin{align}  
	\{ 0, 1, \ldots, i , j , j + 1 , \ldots , n \} \subset \{0, \ldots , n\} \quad \text{and} \quad \{i, i+1, \ldots, j\} \subset \{0, \ldots, n\} \ . 
	\end{align} 
	We have the following commutative diagram:
	\begin{align}  
	\begin{tikzcd}[ampersand replacement = \&]
	\calS_n \bfCoh(X) \arrow{r} \arrow{d} \& \calS_{n-j+i+1} \bfCoh(X) \times_{\calS_1 \bfCoh(X)} \calS_{j-i} \bfCoh(X) \arrow{d} \\
	\calS_n \bfAPerf(X) \arrow{r} \& \calS_{n-j+i+1} \bfAPerf(X) \times_{\calS_1 \bfAPerf(X)} \calS_{j-i} \bfAPerf(X) .
	\end{tikzcd}
	\end{align} 
	The bottom horizontal map is an equivalence.
	After evaluating on $S \in \dAff$, we see that the vertical maps are induced by fully faithful functors.
	It is therefore enough to check that the top horizontal functor is essentially surjective.
	Unraveling the definitions, we have to check the following condition.
	Let $\calF \colon T_n \to \catAPerf(X \times S)$ be a semigrid of length $n$ and write $\calF_{a,b}$ for the image of $(a,b) \in T_n$.
	Then if $\calF_{a,b} \in \catCoh_S( X \times S )$ for $a,b \in \{0, 1, \ldots, i, j, j+1, \ldots, n\}$ or for $a,b \in \{i, i+1, \ldots, j\}$, then $\calF_{a,b} \in \catCoh_S( X \times S )$ for all $a,b$.
	A simple induction argument reduces our task to proving the following statement: Suppose that
	\begin{align}  
	\begin{tikzcd}[ampersand replacement = \&]
	\calG_0 \arrow{r} \arrow{d} \& \calG_1\arrow{d} \\ 
	\calG_2 \arrow{r} \& \calG_3
	\end{tikzcd} 
	\end{align} 
	is a pullback square in $\catPerf(X \times S)$.
	Assume that $\calG_0$, $\calG_2$ and $\calG_3$ belong to $\catCoh_S( X \times S )$.
	Then $\calG_1$ belongs to $\catCoh_S( X \times S )$ as well.
	Since $\calG_0$ and $\calG_3$ have tor-amplitude $\le 0$ relative to $S$, we see that, locally on $X$, for every $\calG \in \catCoh^\heartsuit(S)$ one has
	\begin{align}  
	\pi_k( p_\ast( \calG_1 \oplus \calG_2 ) \otimes \calG ) \simeq 0 
	\end{align} 
	for $k \ge 1$, where $p \colon X \times S \to S$ is the canonical projection.
	However, $\pi_k( p_\ast (\calG_2) \otimes \calG ) \simeq 0$ because $\calG_2$ has tor-amplitude $\le 0$ relative to $S$.
	Therefore $\pi_k( p_\ast( \calG_1 ) \otimes \calG ) \simeq 0$ as well.
	The proof is therefore complete.
\end{proof}

Recall now from \cite[Theorem~11.1.6]{Dyckerhoff_Kapranov_Higher_Segal} that if $\cT$ is a presentable $\infty$-category then there is a canonical functor
\begin{align}
\TwoSeg(\cT) \longrightarrow \Alg_{\mathbb E_1} (\Corr^\times(\cT) ) \ .
\end{align}
Here $\Corr^\times(\cT)$ denotes the $(\infty,2)$-category of correspondences equipped with the symmetric monoidal structure induced from the cartesian structure on $\cT$.
See \cite[\S7.2.1 \& \S9.2.1]{Gaitsgory_Rozenblyum_Study_I}.
As $\mathbb E_1$-monoid objects in correspondences play a significant role for us, we introduce the following terminology:

\begin{defin}
	Let $\cT$ be an $\infty$-category with finite products.
	We define the \textit{$\infty$-category of $\mathbb E_1$-monoid objects in $\cT$} as the $\infty$-category $\Alg_{\mathbb E_1}(\Corr^\times(\cT))$. \hfill $\oslash$
\end{defin}

Taking $\cT = \dSt$, we therefore obtain the following result:

\begin{proposition}\label{prop:coh_2_Segal}
	Let $X \in \dSt$ be a derived stack.
	The $2$-Segal object $\calS_\bullet \bfAPerf(X)$ (resp.\ $\calS_\bullet \bfPerf(X)$, $\calS_\bullet \bfCoh(X)$) endows $\bfAPerf(X)$ (resp.\ $\bfPerf(X)$, $\bfCoh(X)$) with the structure of an $\mathbb E_1$-monoid object in $\dSt$.
\end{proposition}

We conclude this section with an analysis of the geometricity of $\calS_n \bfCoh(X)$.
First, we observe that $\calS_n \bfPerf$ canonically coincides with To\"en-Vaqui\'e's moduli of objects:
\begin{align}
	 \calS_n \bfPerf \simeq \calM_{\calS_n \catPerf} \ . 
\end{align}
To show that $\calS_n \catPerf$ is locally geometric and locally of finite presentation, we use the following two lemmas.

We will make use of the following notation: if $\cC$ is an $\infty$-category, $\cC^\omega$ denotes the full subcategory spanned by compact objects of $\cC$.
\begin{lemma} \label{lem:finite_diagrams_in_smooth_and_proper}
	Let $\cC \in \mathscr P\mathsf r^{\mathsf L, \omega}_\C$ be a compactly generated $\C$-linear stable $\infty$-category and let $I$ be a finite category.
	Then:
	\begin{enumerate}\itemsep=0.2cm
		\item \label{lem:finite_diagrams_in_smooth_and_proper-(1)} The canonical map
		\begin{align}
				\Ind( \Fun(I, \cC^\omega) ) \longrightarrow \Fun(I, \cC)
		\end{align}
		is an equivalence.\footnote{When $I$ is a finite poset, this is a consequence of \cite[Proposition~5.3.5.15]{HTT}.
			Notice that Warning 5.3.5.16 does not apply because for us $I$ is a category, and not an arbitrary simplicial set.}
		
		\item \label{lem:finite_diagrams_in_smooth_and_proper-(2)} Assume furthermore that the idempotent completion of $I$ is finite.
		If $\cC$ is of finite type (resp.\ proper), then so is $\Fun(I, \cC)$.
	\end{enumerate}
\end{lemma}

\begin{proof}
	The canonical functor
	\begin{align}
		 \Fun(I, \cC^\omega) \longrightarrow \Fun(I, \cC) 
	\end{align}
	is fully faithful.
	\cite[Proposition~5.3.4.13]{HTT} shows that it lands in the full subcategory $\Fun(I, \cC)^\omega$ of $\Fun(I,\cC)$ spanned by compact objects.
	Therefore, \cite[Proposition~5.3.5.11-(1)]{HTT} shows that the induced map $\Ind(\Fun(I, \cC^\omega)) \to \Fun(I, \cC)$ is fully faithful.
	We now observe that compact objects in $\Fun(I,\cC)$ coincide with $\Fun(I,\cC^\omega)$.
	We already saw one inclusion.
	For the converse, for every $i \in I$ consider the functor given by evaluation at $i$:
	\begin{align}
		 \ev_i \colon \Fun(I, \cC) \longrightarrow \cC \ .
	\end{align}
	Since $\cC$ is presentable, we see that both left and right Kan extensions along $\{i\} \hookrightarrow I$ exist, providing a left adjoint $L_i$ and a right adjoint $R_i$ to $\ev_i$.
	Moreover, since $I$ is finite, the functor $R_i$ is computed by a finite limit, and therefore $R_i$ commutes with filtered colimits.
	Equivalently, $\ev_i$ preserves compact objects.
	This implies that every object in $\Fun(I, \cC)^\omega$ takes values in $\cC^\omega$.
	To complete the proof of statement \eqref{lem:finite_diagrams_in_smooth_and_proper-(1)}, it is enough to prove that $\Fun(I, \cC)$ is compactly generated.
	Let $F \in \Fun(I, \cC)$ be a functor.
	Our goal is to prove that the canonical map
	\begin{align}
			\colim_{G \in \Fun(I, \cC)^\omega_{/F}} G \longrightarrow F
	\end{align}
	is an equivalence.
	Since the functors $\ev_i$ are jointly conservative and they commute with colimits, it is enough to check that for every $i \in I$ the induced map
	\begin{align}
			\colim_{G \in \Fun(I, \cC)^\omega_{/F}} G(i) \longrightarrow F(i) 
	\end{align}
	is an equivalence.
	We can factor this map as
	\begin{align}
			\colim_{G \in \Fun(I, \cC)^\omega_{/F}} G(i) \longrightarrow \colim_{X \in \cC^\omega_{/F(i)}} X \longrightarrow F(i) \ . 
	\end{align}
	Since $\cC$ is compactly generated, the second map is an equivalence.
	Therefore, it is enough to prove that the functor
	\begin{align}
			\ev_i \colon \Fun(I, \cC)^\omega_{/F} \longrightarrow \cC^\omega_{/F(i)}
	\end{align}
	is cofinal.
	Let $\alpha \colon X \to F(i)$ be a morphism, with $X \in \cC^\omega$.
	We will prove that the $\infty$-category
	\begin{align}
			\calE \coloneqq \Fun(I, \cC)^\omega_{/F} \times_{\cC^\omega_{/F(i)}} (\cC^\omega_{/F(i)})_{\alpha/}
	\end{align}
	is filtered, hence contractible.
	Let $J$ be a finite category and let $A \colon J \to \calE$ be a diagram.
	For every $j \in J$, we get a map
	\begin{align}
			X \longrightarrow A_j(i) \longrightarrow F(i) \ .
	\end{align}
	Since $L_i \dashv \ev_i$, we see that $A$ induces a diagram
	\begin{align}
			\widetilde{A} \colon J \to \Fun(I, \cC)^\omega_{L_i(X) // F} \ . 
	\end{align}
	Let $\overline{A} \colon J^\triangleright \to \Fun(I, \cC)_{L_i(X) // F}$ be the colimit extension of $\widetilde{A}$.
	Since $\ev_i$ commutes with filtered colimits, $L_i$ commutes with compact objects, hence $L_i(X)$ is a compact object.
	Since $J$ is a finite category and since compact objects are closed under finite colimits, we deduce that $\overline{A}$ factors through $\Fun(I, \cC)^\omega_{L_i(X)//F}$.
	Applying $\ev_i$, we obtain the required extension $J^\triangleright \to \calE$ of $A$.
	The proof of \eqref{lem:finite_diagrams_in_smooth_and_proper-(1)} is therefore complete.
	
	To prove \eqref{lem:finite_diagrams_in_smooth_and_proper-(2)}, we first observe that
	\begin{align}
		 \Fun(I, \cC) \simeq \Fun^{\mathsf R}( \PSh(I)\op, \cC ) \simeq \PSh(I) \otimes \cC \ ,
	\end{align}
	where the last equivalence follows from \cite[Proposition~4.8.1.17]{Lurie_Higher_algebra}.
	We can further rewrite it as
	\begin{align}
		 \PSh(I) \otimes \cC \simeq (\PSh(I) \otimes \C \Mod) \otimes_{\C \Mod} \cC \simeq \Fun(I, \C\Mod) \otimes_{\C \Mod} \cC \ . 
	\end{align}
	It is therefore enough to prove that $\Fun(I, \C\Mod)$ is smooth and proper.
	Observe that the collection of objects $\{L_i(\C)\}_{i \in I}$ of $\Fun(I, \C\Mod)$ are compact objects and they generate the category, because the evaluation functors $\ev_i$ are jointly conservative.
	Since $I$ is a finite category, the object
	\begin{align}
		 E \coloneqq \bigoplus_{i \in I} L_i(\C) 
	\end{align}
	is a single compact generator for $\Fun(I, \C\Mod)$.
	Moreover, the end formula for the mapping spaces in $\Fun(I, \C\Mod)$ shows that for $F, G \in \Fun(I, \C\Mod)^\omega \simeq \Fun(I, \catPerf(\C))$, $\Map(F,G)$ is perfect.
	In other words, $\Fun(I, \C\Mod)$ is proper.
	To prove that it is smooth as well, it is enough to check that it is of finite type, see \cite[Proposition~2.14]{Toen_Vaquie_Moduli}.
	Combining Lemma 2.11 and Corollary 2.12 in \emph{loc.\ cit.}, it suffices to show that it is a compact object in in $\mathscr P\mathsf r^{\mathsf L, \omega}_\C$.
	Let $\{\calD_\alpha\}$ be a filtered diagram in $\mathscr P\mathsf r^{\mathsf L, \omega}_\C$ and let
	\begin{align}
		\calD \coloneqq \colim_\alpha \calD_\alpha \ . 
	\end{align}
	We have
	\begin{align}
		\Map_{\mathscr P\mathsf r^{\mathsf L, \omega}_\C}( \Fun(I, \C\Mod), \colim_\alpha \calD_\alpha ) & \simeq \Map_{\mathscr P\mathsf r^{\mathsf L, \omega}}( \PSh(I), \colim_\alpha \calD_\alpha ) \\[2pt]
		& \simeq \Map_{\Cat_\infty}( \PSh(I)^\omega, \colim_\alpha \calD_\alpha^\omega ) .
	\end{align}
	Now, $\PSh(I)^\omega$ is the idempotent completion of $I$, which is finite by assumption.
	Therefore, it is a compact object in $\Cat_\infty$, and we can rewrite the above expression as
	\begin{align}
		\Map_{\mathscr P\mathsf r^{\mathsf L, \omega}_\C}( \Fun(I, \C\Mod), \colim_\alpha \calD_\alpha ) & \simeq \colim_\alpha \Map_{\Cat_\infty}( I, \calD_\alpha^\omega ) \\[2pt]
		& \simeq \colim_\alpha \Map_{\mathscr P\mathsf r^{\mathsf L, \omega}_\C}( \Fun(I, \C\Mod), \calD_\alpha )\ .
	\end{align}
	This shows that $\Fun(I, \C\Mod)$ is compact, and the proof is thus complete.
\end{proof}

\begin{lemma}\label{lem:calS-finite-type}
	Let $\cC$ be a $\C$-linear stable $\infty$-category.
	If $\cC$ is of finite type (resp.\ proper) then $\calS_n \cC$ is of finite type (resp.\ proper).
\end{lemma}

\begin{proof}
	There is a natural inclusion $\Delta^{n-1} \hookrightarrow T_n$, sending $[i]$ to the map $(0,i+1) \colon \Delta^1 \to \Delta^n$.
	Left Kan extension along this map provides a canonical map
	\begin{align}
		 \Fun(\Delta^{n-1}, \cC) \longrightarrow \Fun(T_n, \cC) \ , 
	\end{align}
	which factors through $\calS_n \cC$.
	Proceeding by induction on $n$ and applying \cite[Proposition~4.3.2.15]{HTT} we see that the induced functor
	 \begin{align}
		\Fun(\Delta^{n-1}, \cC) \longrightarrow \calS_n \cC
	\end{align} 
	is an equivalence.
	Since $\Delta^{n-1}$ is idempotent complete and finite, the conclusion follows from Lemma~\ref{lem:finite_diagrams_in_smooth_and_proper}.
\end{proof}

\begin{corollary}
	Let $X$ be a derived stack and assume that:
	\begin{enumerate}\itemsep=0.2cm
		\item there exists a flat effective epimorphism $u \colon U \to X$, where $U$ is a smooth geometric stack; 
		\item the derived stack $\bfPerf(X)$ is locally geometric and locally of finite presentation.
	\end{enumerate}
	Then for every $n \ge 0$, the derived stack $\calS_n \bfCoh(X)$ is geometric and locally of finite presentation.
\end{corollary}

\subsection{Categorified Hall algebras} \label{ss:categorification}

Having the $2$-Segal object $\calS_\bullet \bfCoh(X)$ at our disposal, we now explain how to extract a categorified Hall algebra out of it.
As a first step, we endow 
\begin{align}
		\catQCoh(\bfCoh(X))
\end{align} 
with the structure of a $\mathbb E_1$-monoid object.
The main technical idea involved is the universal property of the $(\infty,2)$-category of correspondences proved in in \cite[Theorem~7.3.2.2]{Gaitsgory_Rozenblyum_Study_I} and \cite[Theorem~4.4.6]{MacPherson}, which we will use below.

Since, we are mostly interested in obtaining a convolution algebra structure on the G-theory spectrum of $\bfCoh(X)$, we need to replace $\catQCoh$ with $\catCohb$.
As the stack $\bfCoh(X)$ is typically not quasi-compact, it is important for us to work within the framework of Appendix~\ref{sec:ind_quasi_compact} and to take some extra care in correctly defining the category of sheaves $\catCohb( \bfCoh(X))$.

Let $\Corr^\times( \mathsf{dSch}^{\mathsf{qcqs}} )$ be the symmetric monoidal $(\infty,2)$-category of correspondences on quasi-compact and quasi-separated derived schemes.
Combining \cite[Proposition~1.4]{Toen_Proper} with \cite[Theorem~7.3.2.2]{Gaitsgory_Rozenblyum_Study_I} we obtain a functor
\begin{align}
	\catQCoh \colon \Corr ( \mathsf{dSch}^{\mathsf{qcqs}} ) \longrightarrow \Cat_\infty^{\mathsf{st}} \ . 
\end{align}
Using \cite[Theorem~4.4.6]{MacPherson}, we see that the above functor can be upgraded to a symmetric monoidal functor
\begin{align}
	 \catQCoh \colon \Corr^\times( \mathsf{dSch}^{\mathsf{qcqs}} ) \longrightarrow \Cat_\infty^{\mathsf{st}} \ . 
\end{align}
Finally, using \cite[Proposition 9.2.3.4]{Gaitsgory_Rozenblyum_Study_I} we can extend this to a right-lax symmetric monoidal functor
\begin{align}
	\catQCoh \colon \Corr^\times( \mathsf{dGeom} )_{\mathsf{rep},\mathsf{all}} \longrightarrow \Cat_\infty^{\mathsf{st}} \ ,
\end{align}
where $\Corr^\times( \mathsf{dGeom} )_{\mathsf{rep},\mathsf{all}}$ is the full subcategory of $\Corr^\times( \mathsf{dGeom} )$ where vertical morphisms are representable by derived schemes.
Informally speaking, we can describe this functor as follows:
\begin{itemize}\itemsep0.2cm
	\item it sends a derived geometric stack $F \in \Corr^\times( \dGeom )$ to $\catQCoh(F)$;
	\item it sends a $1$-morphism
	\begin{align}
		\begin{tikzcd}[ampersand replacement = \&]
		X_0 \& Y \arrow{l}[swap]{p} \arrow{d}{q} \\
		{} \& X_1
		\end{tikzcd}
	\end{align}
	to the composition
	\begin{align}
		\begin{tikzcd}[ampersand replacement = \&]
		\catQCoh(X_0) \arrow{r}{p^\ast} \& \catQCoh(Y) \arrow{r}{q_\ast} \& \catQCoh(X_1)
		\end{tikzcd}\ ;
	\end{align}
	\item the right-lax symmetric monoidal structure is given by
	\begin{align}
		\boxtimes \colon \catQCoh(Z) \otimes \catQCoh(Y) \longrightarrow \catQCoh(Z \times Y) \ . 
	\end{align}
	Denoting by $\mathsf{pr}_Z \colon Z \times Y \to Z$ and $\mathsf{pr}_Y \colon Z \times Y \to Y$ the two natural projections, we have
	\begin{align}
		 \calF \boxtimes \calG \coloneqq \mathsf{pr}_Z^\ast \calF \otimes_{\scrO_{Z \times Y}} \mathsf{pr}_Y^\ast \calG \ . 
	\end{align}
\end{itemize}

Let $X$ be a derived stack.
As shown in Proposition~\ref{prop:coh_2_Segal}, the stack $\bfCoh(X)$ defines an $\mathbb E_1$-monoid object in $\Corr^\times( \dSt )$, the algebra structure being canonically encoded in the $2$-Segal object $\calS_\bullet \bfCoh(X)$.
In the main examples considered in this paper, $\bfCoh(X)$ is furthermore geometric, see Propositions~\ref{prop:coh_geometric}, \ref{prop:Betti_coh_geometric}, \ref{prop:dR_coh_geometric}, and \ref{prop:Dol_coh_geometric}.
In this case, we can apply $\catQCoh$ and obtain a stably monoidal $\infty$-category
\begin{align}\label{eq:QCohE1}
	\catQCoh(\bfCoh(X)) \in \mathrm{Alg}_{\mathbb E_1}( \Cat_\infty^{\mathsf{st}} ) \ . 
\end{align}

Now, we would like to define an $\mathbb E_1$-monoidal structure on $\catCohb( \bfCoh(X) )$. This will be achieved by restricting the functor $\catQCoh$ to a right-lax monoidal functor $\catCohb$ from the category of correspondence. As said before, since $\bfCoh(X)$ is typically not quasi-compact, we need to work in the framework developed in Appendix~\ref{sec:ind_quasi_compact}. 

In Corollary~\ref{cor:quasi_compact_ind} we construct a fully faithful and limit-preserving embedding
\begin{align}
	 (-)_{\mathsf{ind}} \colon \dGeom \longrightarrow \Ind(\dGeom^{\mathsf{qc}}) \ . 
\end{align}
Since $(-)_{\mathsf{ind}}$ commutes with limits, we see that the simplicial object
\begin{align}
	 ( \calS_\bullet \bfCoh(X) )_{\mathsf{ind}} \in \Fun( \mathbf \Delta\op, \Ind(\dGeom^{\mathsf{qc}}) ) 
\end{align}
is a $2$-Segal object, and therefore defines an $\mathbb E_1$-monoidal structure on $\bfCoh(X)_{\mathsf{ind}}$ in $\Corr^\times( \Ind(\dGeom^{\mathsf{qc}}) )$.
When the context is clear, we drop the subscript $(-)_{\mathsf{ind}}$ in the above expression.

On the other hand, Corollary~\ref{cor:ind_qcoh} provides a right-lax symmetric monoidal functor
\begin{align}
	\catQCoh_{\mathsf{pro}} \colon \Corr^\times(\Ind(\dGeom^{\mathsf{qc}}))_{\mathsf{rps}, \mathsf{all}} \longrightarrow \mathsf{Pro}(\Cat_\infty^{\mathsf{st}}) 
\end{align}
In particular, we obtain a \textit{refinement} of \eqref{eq:QCohE1}, i.e., the stable pro-category
\begin{align}
	 \flim_{U \Subset \bfCoh(X)} \catQCoh(U) \in \mathsf{Pro}(\Cat_\infty^{\mathsf{st}}) 
\end{align}
acquires a canonical $\mathbb E_1$-monoidal structure.
The colimit is taken over all quasi-compact open substacks of $\bfCoh(X)$ (but an easy cofinality argument shows that one can also employ a chosen quasi-compact exhaustion of $\bfCoh(X)$). 

Now, we see how to replace $\catQCoh$ by $\catCohb$.

\begin{definition}
	A morphism $f \colon X \to Y$ in $\Ind(\dGeom^{\mathsf{qc}})$ is said to be \textit{ind-derived lci} if for every $Z \in \dGeom^{\mathsf{qc}}$ and any morphism $Z \to Y$, the pullback $X \times_Y Z$ is a quasi-compact derived geometric stack and the map $X \times_Y Z \to Z$ is derived lci.
\end{definition}

\begin{lemma} \label{lem:ind_derived_lci}
	Let $f \colon X \to Y \in \dGeom$ be a quasi-compact derived lci morphism.
	Then the induced morphism
	\begin{align}
		 f_{\mathsf{ind}} \colon X_{\mathsf{ind}} \to Y_{\mathsf{ind}} 
	\end{align}
	is ind-derived lci.
\end{lemma}

\begin{proof}
	Using Lemma~\ref{lem:quasi_compact_open_exhaustion}-\eqref{lem:quasi_compact_open_exhaustion-(1)} we can choose an open Zariski exhaustion
	 \begin{align}
		\emptyset = U_0 \hookrightarrow U_1 \hookrightarrow \cdots \hookrightarrow U_\alpha \hookrightarrow U_{\alpha+1} \hookrightarrow \cdots
	\end{align}
	of $Y$, where each $U_\alpha$ is a quasi-compact derived geometric stack. Set
	\begin{align}
		V_\alpha \coloneqq U_\alpha \times_Y X\ .
	\end{align}
	Since $f$ is quasi-compact, the $V_\alpha$ are quasi-compact derived stacks and they form an open Zariski exhaustion of $X$.
	Let $f_\alpha \colon V_\alpha \to U_\alpha$ be the induced morphism, which is lci.
	Therefore, Lemma~\ref{lem:quasi_compact_open_exhaustion}-\eqref{lem:quasi_compact_open_exhaustion-(3)} implies that
	\begin{align}
		X_{\mathsf{ind}} \simeq \fcolim V_\alpha \quad\text{and}\quad Y_{\mathsf{ind}} \simeq \fcolim U_\alpha \ , 
	\end{align}
	and $f_{\mathsf{ind}} \simeq \fcolim f_\alpha$.
	Let $Z \in \dGeom^{\mathsf{qc}}$ be a quasi-compact derived geometric stack and let $g \colon Z \to Y$ be a morphism.
	Using Lemma~\ref{lem:quasi_compact_open_exhaustion}-\eqref{lem:quasi_compact_open_exhaustion-(2)}, we find an index $\alpha$ such that $g$ factors through $U_\alpha$. In particular, the pullback $Z \times_Y X$ fits in the following ladder:
	\begin{align}
		\begin{tikzcd}[ampersand replacement = \&]
			Z \times_Y X \arrow{r} \arrow{d} \& V_\alpha \arrow{d} \arrow{r} \& X \arrow{d} \\
			Z \arrow{r}{g_\alpha} \& U_\alpha \arrow{r} \& Y 
		\end{tikzcd} \ .
	\end{align}
	Since the morphism $V_\alpha \to U_\alpha$ is quasi-compact and derived lci, so is $Z \times_Y X \to Z$.
	This completes the proof.
\end{proof}

Consider now the subcategory $\Corr^\times( \Ind(\dGeom^{\mathsf{qc}}) )_{\mathsf{rps}, \mathsf{lci}}$ of $\Corr^\times(\Ind(\dGeom^{\mathsf{qc}}))_{\mathsf{rep}, \mathsf{all}}$ where the horizontal arrows are taken to be ind-derived lci morphisms and the vertical arrows to be morphisms representable by proper schemes.
Consider the restriction of $\catQCoh$ to this subcategory:
\begin{align}
	\catQCoh_{\mathsf{pro}} \colon \Corr^\times( \Ind(\dGeom^{\mathsf{qc}}) )_{\mathsf{rps}, \mathsf{lci}} \longrightarrow \mathsf{Pro}(\Cat_\infty^{\mathsf{st}} ) \ . 
\end{align}
Let $f \colon X \to Y$ be a morphism in $\Ind(\dGeom^{\mathsf{qc}})$ which is representable by proper schemes.
Using Lemma~\ref{lem:quasi_compact_open_exhaustion}-(1), we can choose an open Zariski exhaustion
\begin{align}
	\emptyset = U_0 \hookrightarrow U_1 \hookrightarrow \cdots \hookrightarrow U_\alpha \hookrightarrow U_{\alpha+1} \hookrightarrow \cdots
\end{align}
of $Y$, where each $U_\alpha$ is a quasi-compact derived stack.
Let $V_\alpha \coloneqq U_\alpha \times_Y X$ and let $f_\alpha \colon V_\alpha \to U_\alpha$ be the induced morphism.
Since $f$ is representable by proper schemes, $V_\alpha$ is again quasi-compact and therefore we obtain a compatible open Zariski exhaustion of $X$.
Thanks to derived base change, we can therefore compute the pushforward in $\mathsf{Pro}(\Cat_\infty^{\mathsf{st}})$ by
\begin{align}
	f_\ast \coloneqq \flim_\alpha f_{\alpha*} \colon \flim_{\alpha} \catQCoh(V_\alpha) \longrightarrow \flim_{\alpha} \catQCoh(U_\alpha) \ .
\end{align}
Since each $f_\alpha$ is representable by proper schemes, this functor restricts to a morphism
\begin{align}
	f_\ast \colon \flim_{\alpha} \catCohb(V_\alpha) \longrightarrow \flim_{\alpha} \catCohb(U_\alpha) \ .
\end{align}
Using \cite[Lemma~2.2]{Toen_Proper}, we similarly deduce that if $f \colon X \to Y$ is a morphism in $\Ind(\dGeom^{\mathsf{qc}})$ which is ind-derived lci, then the pullback functor restricts to a morphism
\begin{align}
	f^\ast \colon \flim_\alpha \catCohb(U_\alpha) \longrightarrow \flim_\alpha \catCohb(V_\alpha) \ .
\end{align}
This implies that $\catQCoh_{\mathsf{pro}}$ admits a right-lax monoidal subfunctor
\begin{align}
	 \catCohb_{\mathsf{pro}} \colon \Corr^\times( \Ind(\dGeom^{\mathsf{qc}}) )_{\mathsf{rps}, \mathsf{lci}} \longrightarrow \mathsf{Pro}(\Cat_\infty^{\mathsf{st}} ) \ .
\end{align}

Applying the tor-amplitude estimates obtained in \S\ref{s:coh_ext}, we obtain the following result:
\begin{theorem} \label{thm:coh_algebra_in_correspondence}
	Let $X$ be one of the following derived stacks:
	\begin{enumerate}\itemsep0.2cm
		\item a smooth proper complex scheme of dimension either one or two;
		\item the Betti, de Rham or Dolbeault stack of a smooth projective curve.
	\end{enumerate}
	Then the composition
	\begin{align}
		\catCohb_{\mathsf{pro}}( \bfCoh(X) ) \times \catCohb_{\mathsf{pro}}( \bfCoh(X) ) \xrightarrow{\boxtimes} \catCohb_{\mathsf{pro}}( \bfCoh(X) \times \bfCoh(X) ) \xrightarrow{q_\ast \circ p^\ast} \catCohb_{\mathsf{pro}}( \bfCoh(X) )\ ,
	\end{align}
	where the map on the right-hand side is induced by the 1-morphism in correspondences:
	\begin{align}\label{eq:convolution_diagram_algebraic}
		\begin{tikzcd}[ampersand replacement = \&]
			{} \& \bfCohext(X) \arrow{dr}{q} \arrow{dl}[swap]{p} \\
			\bfCoh(X) \times \bfCoh(X) \& \& \bfCoh(X) 
		\end{tikzcd}\ ,
	\end{align}
	endows $\catCohb_{\mathsf{pro}}(\bfCoh(X))$ with the structure of an $\mathbb E_1$-monoidal stable $\infty$-pro-category.
\end{theorem}

\begin{proof}
	By Proposition~\ref{prop:coh_2_Segal} we know that $\calS_\bullet \bfCoh(X)$ is a $2$-Segal object in $\dSt$.
	Using Corollary~\ref{cor:quasi_compact_ind}, we see that $(\calS_\bullet \bfCoh(X))_{\mathsf{ind}}$ is a $2$-Segal object in $\Corr^\times( \dSt )$, and therefore it defines an $\mathbb E_1$-monoid object in $\Corr^\times( \Ind(\dGeomqc) )$.
	Proposition~\ref{prop:stack_extensions_is_relatively_affine} shows that the map $p$ is quasi-compact.
	On the other hand Proposition~\ref{prop:extremal_projections_scheme} shows that $p$ is lci when $X$ is a smooth and proper complex scheme of dimension $1$ or $2$, while Corollaries~\ref{cor:extremal_projections_betti}, \ref{cor:dR-lci} and \ref{cor:Dol-lci} show that the same is true when $X$ is the Betti, de Rham or Dolbeault stack of a smooth projective curve.
	Therefore, Lemma~\ref{lem:ind_derived_lci} shows that in all these cases $p_{\mathsf{ind}}$ is ind-derived lci. Moreover, the morphism $q$ is representable by proper schemes: indeed, one can show that $q$ is representable by Quot schemes\footnote{In the de Rham and Dolbeault cases, one has to consider Quot schemes of $\Lambda$-modules à la Simpson, cf.\ the proof of \cite[Theorem~3.8]{Simpson_Moduli_I}. To show the properness of $q$ in the Betti case, one can either use the global quotient description of the Betti moduli stack, described e.g. in \cite[\S1.2]{PT_Local}, or apply the derived Riemann--Hilbert correspondence of \cite{Porta_Derived} and use the invariance of properness under analytification.} and it is known that these are proper schemes.
	The $2$-Segal condition therefore guarantees that $(\calS_\bullet \bfCoh(X))_{\mathsf{ind}}$ endows $\bfCoh(X)_{\mathsf{ind}}$ with the structure of an $\mathbb E_1$-monoid object in $\Corr^\times( \Ind(\dGeom^{\mathsf{qc}}) )_{\mathsf{rps}, \mathsf{lci}}$.
	Applying the right-lax monoidal functor $\catCohb_{\mathsf{pro}} \colon \Corr^\times( \Ind(\dGeom^{\mathsf{qc}}) )_{\mathsf{rps}, \mathsf{lci}} \to \mathsf{Pro}(\Cat_\infty^{\mathsf{st}} )$, we conclude that $\catCohb_{\mathsf{pro}}( \bfCoh(X) )$ inherits the structure of an $\mathbb E_1$-monoid object in $\mathsf{Pro}(\Cat_\infty^{\mathsf{st}} )$.
\end{proof}

Since $\mathbb E_1$-monoid objects in $\mathsf{Pro}(\Cat_\infty^{\mathsf{st}} )$ are (by definition) the same as $\mathbb E_1$-monoidal categories in $\mathsf{Pro}(\Cat_\infty^{\mathsf{st}} )$, we refer to the corresponding tensor structure as the \textit{CoHA tensor structure on $\catCohb_{\mathsf{pro}}( \bfCoh(X) )$}.
We denote this monoidal structure by $\ostar$.

\begin{remark}\label{rem:other-stacks}
	Let $X$ be a smooth projective complex scheme of dimension either one or two. Then the moduli stacks introduced in \S\ref{ss:other-moduli-spaces} are $\mathbb E_1$-monoid objects in $\Corr^\times( \Ind(\dGeom^{\mathsf{qc}}) )_{\mathsf{rps}, \mathsf{lci}}$. If $X$ is quasi-projective, then $\bfCoh_{\mathsf{prop}}(X)$ (resp.\ $\bfCoh_{\mathsf{prop}}^{\leqslant d}(X)$) is an $\mathbb E_1$-monoid object in $\Corr^\times( \Ind(\dGeom^{\mathsf{qc}}) )_{\mathsf{rps}, \mathsf{lci}}$ (resp.\ for any integer $d\leq \dim(X)$).

	Similarly, for the Dolbeault shape, a statement similar to that of Theorem~\ref{thm:coh_algebra_in_correspondence} holds for all the moduli stacks introduced in \S\ref{sec:Dolbeault-moduli-stacks}.
\end{remark}

\subsection{The equivariant case}\label{ss:equivariant}

The main results of \S\ref{ss:convolution_structure} and of \S\ref{ss:categorification} carry over without additional difficulties in the equivariant setting.
Let us sketch how to modify the key constructions.

Let $X \in \dSt$ be a derived stack and let $G \in \mathsf{Mon}_{\mathbb E_1}^{\mathsf{gp}}(\dSt^\times)$ be a grouplike $\mathbb E_1$-monoid in derived stacks acting on $X$.
Typically, $G$ will be an algebraic group.
Since the monoidal structure on $\dSt$ is cartesian, we can use \cite[Proposition~4.2.2.9]{Lurie_Higher_algebra} to reformulate the datum of the $G$-action on $X$ as a diagram
\begin{align}
	A_{G,X} \colon \mathbf \Delta\op \times \Delta^1\longrightarrow \dSt 
\end{align}
satisfying the relative $1$-Segal condition.
Informally speaking, $A_{G,X}$ is the diagram
\begin{align}
	\begin{tikzcd}[ampersand replacement = \&]
	\cdots \arrow[shift left = 1.5ex]{r} \arrow[shift left = .5ex]{r} \arrow[shift left = -.5ex]{r} \arrow[shift left = -1.5ex]{r} \& G^2 \times X \arrow[shift left  = 1ex]{r} \arrow{r} \arrow[shift left = -1ex]{r} \arrow{d} \& G \times X \arrow[shift left = .5ex]{r} \arrow[shift left = -.5ex]{r} \arrow{d} \& X \arrow{d} \\
	\cdots \arrow[shift left = 1.5ex]{r} \arrow[shift left = .5ex]{r} \arrow[shift left = -.5ex]{r} \arrow[shift left = -1.5ex]{r} \& G^2 \arrow[shift left  = 1ex]{r} \arrow{r} \arrow[shift left = -1ex]{r} \& G \arrow[shift left = .5ex]{r} \arrow[shift left = -.5ex]{r} \& \Spec(k)
	\end{tikzcd}\ ,
\end{align}
which encodes at the same time the $\mathbb E_1$-structure on $G$ and the action on $X$.
We denote the geometric realization of the top simplicial object by $[X/G]$, while it is customary to denote the geometric realization of the bottom one by $\sfB G$.

We now define
\begin{align}
	 \calS_\bullet \bfPerf_G(X) \colon \mathbf \Delta\op \longrightarrow \dSt_{/\sfB G}
\end{align}
by setting
\begin{align}
	\calS_\bullet \bfPerf_G(X) \coloneqq \bfMap_{/\sfB G}( [X / G], \calS_\bullet \bfPerf \times \sfB G ) \ .
\end{align}
We also write $\bfPerf_G(X)$ for $\calS_1 \bfPerf_G(X)$.
Notice that
\begin{align}
	\Spec(k) \times_{\sfB G}  \calS_\bullet \bfPerf_G(X) \simeq \bfMap( X, \calS_\bullet \bfPerf ) \ .
\end{align}
We can therefore unpack the datum of the map $\calS_\bullet \bfPerf_G( X ) \to \sfB G$ by saying that $G$ acts canonically on $\calS_\bullet \bfPerf( X )$.
From this point of view, we have a canonical equivalence\footnote{This is nothing but a very special case of the descent for $\infty$-topoi, see \cite[Theorem~6.1.3.9 and Proposition~6.1.3.10]{HTT}.}
\begin{align}
	\calS_\bullet \bfPerf_G( X ) \simeq [ \calS_\bullet \bfPerf(X) / G ] \ .
\end{align}
As an immediate consequence we find that
\begin{align}
	\catCohb( \bfPerf_G( X ) ) \simeq \catCohb_G( \bfPerf( X ) ) \ . 
\end{align}
The right-hand side is the $G$-equivariant stable $\infty$-category of bounded coherent complexes on $\bfPerf(X)$.
Since the functor
\begin{align}
	\bfMap_{/\sfB G}( [X/G], (-) \times \sfB G ) \colon \dSt \longrightarrow \dSt_{/\sfB G}
\end{align}
commutes with limits, we deduce:

\begin{proposition}
	The simplicial derived stack $\calS_\bullet \bfPerf_G(X) \colon \mathbf \Delta\op \to \dSt_{/\sfB G}$ is a $2$-Segal object.
\end{proposition}

Assume now that $G$ is geometric (e.g.\ an affine group scheme) and that there exists a geometric derived stack $U$ equipped with the action of $G$ and a $G$-equivariant, flat effective epimorphism $u \colon U \to X$.
Then the induced morphism $[U/G] \to [X/G]$ is an effective epimorphism which is flat relative to $\sfB G$.
We define $\calS_\bullet \bfCoh_G(X) \in \Fun( \mathbf \Delta\op, \dSt_{/\sfB G} )$ as follows.
Given an affine derived scheme $S = \Spec(A)$ and a morphism $x \colon S \to \sfB G$, we set
\begin{align}
	\mathsf{Map}_{/\sfB G}( S, \calS_\bullet \bfCoh_G(X) ) \coloneqq \left( \calS_\bullet \catCoh_S( S \times_{\sfB G} [X / G] ) \right)^\simeq \in \Fun( \mathbf \Delta\op, \scrS ) \ .
\end{align}
We immediately obtain:

\begin{corollary}
	Let $X$ be a geometric derived stack. Then the simplicial derived stack $\calS_\bullet \bfCoh_G( X ) \colon \mathbf \Delta\op \to \dSt_{/\sfB G}$ is a $2$-Segal object.
\end{corollary}

The above $2$-Segal object endows $\bfCoh(X)$ with the structure of a $G$-equivariant $\mathbb E_1$-convolution algebra in $\dSt$.

\begin{corollary}
	Let $X \in \dSt$ be a derived stack and let $u \colon U \to X$ be a flat effective epimorphism from a geometric derived stack $U$. Assume that the $2$-Segal object $\calS_\bullet \bfCoh(X)$ endows $\bfCoh(X)$ with the structure of an $\mathbb E_1$-monoid object in $\Corr^\times(\dSt)_{\mathsf{rps}, \mathsf{lci}}$.
	Let $G$ be a smooth algebraic group acting on both $U$ and $X$ assume that $u$ has a $G$-equivariant structure.
	Then the $G$-equivariant $2$-Segal object $\calS_\bullet \bfCoh_G(X)$ induces a $\mathbb E_1$-monoidal structure on $\catCohb_{\mathsf{pro}}( \bfCoh_G(X) )\simeq \catCohb_{\mathsf{pro}, G}( \bfCoh(X) )$.
\end{corollary}

\begin{proof}
	Similarly to the proof of Theorem~\ref{thm:coh_algebra_in_correspondence}, all we need to check is that the map
\begin{align}
		\ev_3 \times \ev_1 \colon \bfCohext_G(X) \longrightarrow \bfCoh_G(X) \times \bfCoh_G(X) 
\end{align}
	is quasi-compact and derived lci and that the map
\begin{align}
		 \ev_2\colon \bfCohext_G(X) \longrightarrow \bfCoh_G(X) 
\end{align}
	is representable by proper schemes.
	Observe that for $i = 1, 2, 3$ the right and the outer squares in the commutative diagram
	\begin{align}
		 \begin{tikzcd}[ampersand replacement = \&]
			\bfCohext(X) \arrow{r}{\ev_i} \arrow{d} \& \bfCoh(X) \arrow{d} \arrow{r} \& \Spec(k) \arrow{d} \\
			\bfCohext_G(X) \arrow{r}{\ev_i} \& \bfCoh_G(X) \arrow{r} \& \sfB G
		\end{tikzcd}
	\end{align}
	are pullback squares.
	Therefore so is the left one.
	The conclusion now follows because $\Spec(k) \to \sfB G$ is a smooth atlas and from the analogous statements for $\bfCoh(X)$, which have been proven in the proof of Theorem~\ref{thm:coh_algebra_in_correspondence}.
\end{proof}

\bigskip\section{Decategorification}\label{s:cohas}

Now, we investigate what happens to our construction when we \textit{decategorify}, i.e., when we pass to the G-theory (introduced in \S\ref{ss:g-theory-non-quasi-compact}).
A first consequence of our Theorem~\ref{thm:coh_algebra_in_correspondence} is the following:

\begin{proposition}
	Let $Y$ be one of the following derived stacks:
	\begin{enumerate}\itemsep0.2cm
		\item a smooth proper complex scheme of dimension either one or two;
		\item the Betti, de Rham or Dolbeault stack of a smooth projective curve.
	\end{enumerate}
	The CoHA tensor structure on $\catCohb_{\mathsf{pro}}(\bfCoh(Y))$ endows $G(\bfCoh(Y))$ with the structure of a $\mathbb E_1$-monoid object in $\mathsf{Sp}$.
\end{proposition}

\begin{remark}
	Up to our knowledge, the above result provides the first construction of a Hall algebra structure on the full G-theory spectrum of $\bfCoh(Y)$. 
	Furthermore, the above results hold also for the stack $\bfCoh_{\mathsf{prop}}^{\leqslant d}(S)$, where $S$ is a smooth (quasi-)projective complex surface and $0 \leq d \leq 2$ is an integer. 
\end{remark}

Taking $\pi_0$ of $G(\bfCoh(Y))$, we obtain an associative algebra structure on $G_0(\bfCoh(Y))$.
When $Y$ is the de Rham shape of a curve, this is a $K$-theoretical Hall algebra associated to flat vector bundles on the curve, and it has not been previously considered in the literature.
On the other hand, in \cite{Zhao_Hall,KV_Hall} and in \cite{Sala_Schiffmann} the authors considered the cases of $Y$ being a surface or being the Dolbeault shape of a curve, respectively.
Below, we briefly review the construction in \cite{KV_Hall} and prove that the two algebra structures on $G_0(\bfCoh(Y))$ obtained using our method or theirs agree.

Let $S$ be a smooth (quasi-)projective complex surface and let $0 \leq d \leq 2$ be an integer.
To lighten the notation, write
\begin{align}
	Y \coloneqq \bfCoh_{\mathsf{prop}}^{\le d}(S) , \qquad Y^{\mathsf{ext}} \coloneqq \bfCoh_{\mathsf{prop}}^{\le d, \mathsf{ext}}(S) 
\end{align}
and
\begin{align}
	Y_0 \coloneqq \trunc{Y}, \qquad Y_0^{\mathsf{ext}} \coloneqq \trunc{Y}^{\mathsf{ext}} \ .
\end{align}
Proposition~\ref{prop:stack_extensions_is_relatively_affine} implies that
\begin{align}
	Y^{\mathsf{ext}} = \Spec_{Y\times Y}( \Sym_{\scrO_{Y \times Y}}( \calE^\vee ) ) \ , 
\end{align}
where $\calE \in \catPerf(Y \times Y)$ is a certain perfect complex on $Y \times Y$.
Let
\begin{align}
	i \colon Y_0 \times Y_0 \longrightarrow Y \times Y 
\end{align}
be the natural inclusion and let $\calE_0 \coloneqq i^\ast(\calE)$. Set
\begin{align}
	\widetilde{Y}^{\mathsf{ext}} \coloneqq \Spec_{Y_0 \times Y_0}( \Sym_{\scrO_{Y_0 \times Y_0}}( \calE_0^\vee ) ) \ .
\end{align}
Consider the commutative diagram
\begin{align}
	\begin{tikzcd}[ampersand replacement = \&]
		Y_0^{\mathsf{ext}} \arrow{r}{\iota} \arrow{dr}[swap]{p_0} \& \widetilde{Y}^{\mathsf{ext}} \arrow{r}{j} \arrow{d}{\widetilde{p}} \& Y^{\mathsf{ext}} \arrow{d}{p} \\
		{} \& Y_0 \times Y_0 \arrow{r}{i} \& Y \times Y 
	\end{tikzcd} \ .
\end{align}
The right square is a pullback, by construction. Therefore, the diagram
\begin{align}
	 \begin{tikzcd}[ampersand replacement = \&]
		\catCohb( \widetilde{Y}^{\mathsf{ext}} ) \arrow{r}{j_\ast} \& \catCohb(Y^{\mathsf{ext}}) \\
		\catCohb(Y_0 \times Y_0) \arrow{u}{\widetilde{p}^\ast} \arrow{r}{i_\ast} \& \catCohb(Y \times Y) \arrow{u}{p^\ast}
	\end{tikzcd} 
\end{align}
canonically commutes.
Passing to $G$-theory, the functors $i_\ast$ and $j_\ast$ induce equivalences, thereby identifying $p^\ast$ et $\widetilde{p}^\ast$.

We now compare $\iota_\ast^{-1} \circ \widetilde{p}^\ast \colon G( Y_0 \times Y_0 ) \to G( Y^{\mathsf{ext}}_0 )$ with the construction of the virtual pullback $p_0!$ by Kapranov--Vasserot.
In \cite[\S3.3]{KV_Hall}, they take as additional input an explicit resolution of $\calE_0$ as a $3$-terms complex
\begin{align}
	\calE^\bullet_0 \coloneqq \cdots \longrightarrow 0 \longrightarrow \calE^{-1}_0 \xrightarrow{d^0} \calE^0_0 \xrightarrow{d^1} \calE^1_0 \longrightarrow 0 \longrightarrow \cdots \ . 
\end{align}
Let $\calE_0^{\le 0}$ be the $2$-terms complex $\calE_0^{-1} \to \calE_0^0$ and set
\begin{align}
	E^{\le 0} \coloneqq \mathbb V_{Y_0 \times Y_0}( (\calE_0^{\le 0})^\vee ) \quad \text{and} \quad E^1 \coloneqq \mathbb V_{Y_0 \times Y_0}( (\calE_0^1)^\vee ) \ . 
\end{align}
The canonical projection $\pi \colon E^{\le 0} \to Y_0 \times Y_0$ is smooth and the differential $d^1$ induces a section
\begin{align}
	s \colon E^{\le 0} \longrightarrow \pi^\ast E^1 \coloneqq E^1 \times_{Y_0 \times Y_0} E^{\le 0} \ , 
\end{align}
such that
\begin{align}
	\begin{tikzcd}[ampersand replacement = \&]
		\widetilde{Y}^{\mathsf{ext}} \arrow{r}{t} \arrow{d}{t} \& E^{\le 0} \arrow{d}{s} \\
		E^{\le 0} \arrow{r}{0} \& \pi^\ast E^1
	\end{tikzcd} 
\end{align}
is a derived pullback square.
Therefore, we can factor $\widetilde{p} \colon \widetilde{Y}^{\mathsf{ext}} \to Y_0 \times Y_0$ as
\begin{align}
	\begin{tikzcd}[ampersand replacement = \&]
		\widetilde{Y}^{\mathsf{ext}} \arrow{r}{t} \arrow{dr}[swap]{\widetilde{p}} \& E^{\le 0} \arrow{d}{\pi} \\
		{} \& Y_0 \times Y_0 
	\end{tikzcd} \ .
\end{align}
As in \textit{loc.cit.} the operation $p_0^!$ is defined as the composition $s^! \circ \pi^\ast \colon G_0(Y_0 \times Y_0) \to G_0(Y_0^{\mathsf{ext}})$, to compare the two constructions it is enough to verify that
\begin{align}
	 s^! = \iota_\ast^{-1} \circ t^\ast 
\end{align}
as functions $G_0(E^{\le 0}) \to G_0(Y_0^{\mathsf{ext}})$. This follows at once by unraveling the definition of $s^!$. Thus our construction of the Hall product on $G_0(Y)\simeq G_0(Y_0)$ coincides with theirs, 
and we obtain:
\begin{theorem}
	Let $S$ be a smooth (quasi-)projective complex surface and let $0\leq d\leq 2$ be an integer. There exists an algebra isomorphism between
	\begin{align}
		\pi_0\big( \lim K( \catCohb_{\mathsf{pro}}( \bfCoh_{\mathsf{prop}}^{\leqslant d}(S) )) \big)
	\end{align}
	and the K-theoretical Hall algebra of $S$ as defined in \cite{Zhao_Hall, KV_Hall}. Thus, the CoHA tensor structure on the stable $\infty$-category $\catCohb_{\mathsf{pro}}( \bfCoh^{\leqslant d}(S) )$ is a categorification of the latter. 
	
	Finally, if in addition $S$ is toric, similar results holds in the equivariant setting.
\end{theorem}

\subsection{The equivariant case}\label{ss:cohas-equivariant}

Let $\bfCoh_0(\C^2)\coloneqq \bfCoh_{\mathsf{prop}}^{\leqslant 0}(\C^2)$ be the geometric derived stack of zero-dimen\-sio\-nal coherent sheaves on $\C^2$. Note that the natural $\C^\ast\times\C^\ast$-action on $\C^2$ lifts to an action on $\bfCoh_0(\C^2)$.

A convolution algebra structure on the Grothendieck group $G_0^{\C^\ast\times \C^\ast}( \trunc{\bfCoh_{0}(\C^2)} )$ of the truncation of $\bfCoh_0(\C^2)$ has been defined in \cite{SV_elliptic,SV_Langlands}. In those papers, the convolution product is defined by using an explicit presentation of $\trunc{\bfCoh_0(\C^2)}$ as disjoint union of quotient stacks. Moreover, as proved in those papers, the convolution algebra on $G_0^{\C^\ast\times \C^\ast}( \trunc{\bfCoh_0(\C^2)} )$ is isomorphic to a positive nilpotent part $\mathbf{U}_{q, t}^+(\ddot{\mathfrak{gl}}_1)$ of the elliptic Hall algebra $\mathbf{U}_{q, t}(\ddot{\mathfrak{gl}}_1)$ of Burban and Schiffmann \cite{BS_elliptic}.

In \cite[Proposition~6.1.5]{KV_Hall}, the authors showed that the convolution product defined by using virtual pullbacks coincides with the convolution product defined by using the explicit description of $\trunc{\bfCoh_0(\C^2)}$ in terms of quotient stacks. Thanks to this result (which holds also equivariantly), by arguing as in the previous section, one can show the following.
\begin{proposition}
	There exists a $\Z[q,t]$-algebra isomorphism 
	\begin{align}
		\pi_0 K( \catCohb_{\mathsf{pro},\,\C^\ast\times \C^\ast}( \bfCoh_0(\C^2) ) ) \simeq \mathbf{U}_{q, t}^+(\ddot{\mathfrak{gl}}_1)\ .	
	\end{align}
	Thus, $\big(\catCohb_{\mathsf{pro},\,\C^\ast\times \C^\ast}( \bfCoh_0(\C^2) ), \ostar \big)$ is a categorification of $\mathbf{U}_{q, t}^+(\ddot{\mathfrak{gl}}_1)$. Here, $\bfCoh_0(\C^2)\coloneqq \bfCoh_{\mathsf{prop}}^{\leqslant 0}(\C^2)$.
\end{proposition}

Let $X$ be a smooth projective complex curve and let $\bfHiggs^{\textrm{na\"if}}(X)\coloneqq \mathbf{T}^\ast[0]\bfCoh(X)$ (cf.\ Remark~\ref{rem:naive_Higgs}). Recall that $\C^\ast$ acts by ``scaling the Higgs fields". 

The Gro\-then\-dieck group $G_0^{\C^\ast}( \trunc{\bfHiggs^{\textrm{na\"if}}(X)} )$ of the truncation of $\bfHiggs^{\textrm{na\"if}}(X)$ is endowed with a convolution algebra structure as constructed in \cite{Sala_Schiffmann} and in \cite{Minets_Hall} for the rank zero case. In the rank zero case, the construction of the product follows the one in \cite{SV_elliptic,SV_Langlands} discussed above, while in the higher rank case one uses a \textit{local} description of $\bfHiggs^{\textrm{na\"if}}(X)$ as a quotient stack; then the construction of the product is performed locally and one glues suitably to get a global convolution product. By similar arguments as above and thanks to Remark~\ref{rem:naive_Higgs}, we have the following. 
\begin{proposition}
	Let $X$ be a smooth projective complex curve. There exists an algebra isomorphism between 
	\begin{align}
		\pi_0\big( \lim K( \catCohb_{\mathsf{pro},\,\C^\ast}( \bfCoh(X_\Dol)  ) \big)
	\end{align}
	and the K-theoretical Hall algebra of Higgs sheaves on $X$ introduced in \cite{Sala_Schiffmann, Minets_Hall}. Thus, the CoHA tensor structure on the stable $\infty$-category $\catCohb_{\mathsf{pro},\,\C^\ast }( \bfCoh(X_\Dol) )$ is a categorification of the latter. 
\end{proposition}
\begin{remark}
	The Betti K-theoretical Hall algebra of a smooth projective complex curve $X$ can be defined by using a K-theoretic analog of the Kontsevich-Soibelman CoHA formalism due to Pădurariu in \cite{Padurariu_CoHA} for the quiver with potential defined by Davison in \cite{Davison_character}. We expect that this algebra is isomorphic to our decategorification of the Betti Cat-HA.
\end{remark}

Finally, it is relevant to mention that our approach defines the de Rham K-theoretical Hall algebra of a smooth projective curve $X$. The nature of this algebra is at the moment mysterious. Note that \cite{SV_Langlands} gives an indication that the algebra should at least contain the K-theoretical Hall algebra of the preprojective algebra of the $g$-loop quiver, where $g$ is the genus of $X$.

\begin{remark}
	By using the formalism of Borel--Moore homology of higher stacks developed in \cite{KV_Hall} and their construction of the Hall product via virtual pullbacks, we obtain equivalent realizations of the COHA of a surface by \cite{KV_Hall} and of the Dolbeault CoHA of a curve \cite{Sala_Schiffmann, Minets_Hall}. Moreover, we define the de Rham cohomological Hall algebra of a curve.
\end{remark}

\bigskip\section{A Cat-HA version of the Hodge filtration}\label{s:Hodge_filtration}

In this section, we shall present a relation between the de Rham categorified Hall algebra and the Dolbeault categorified Hall algebra, which is induced by the Deligne categorified Hall algebra $( \catCohb_{\C^\ast}( \bfCoh(X_\Del) ) , \ostar_\Del )$. Deligne's $\lambda$-connections interpolate Higgs bundles with vector bundles with flat connections, and they were used by Simpson \cite{Simpson_Hodge} to prove the non-abelian Hodge correspondence. For this reason, the relation we prove in this section can be interpreted as a version of the Hodge filtration in the setting of categorified Hall algebras.

\subsection{Categorical filtrations}

We let
\begin{align}
	\catPerffilt \coloneqq \catPerf( [\A^1_\C / \G_m] ) \ , \quad \catPerfgr \coloneqq \catPerf( \sfB \G_m ) \ .
\end{align}
The two morphisms
\begin{align}
	\begin{tikzcd}[ampersand replacement = \&]
	\sfB \G_m \arrow{r}{j} \& {[ \A^1_\C / \G_m]} \& \Spec(\C) \simeq [\G_m / \G_m] \arrow{l}[swap]{i}
	\end{tikzcd}
\end{align}	 
induce canonical morphisms
\begin{align}
	j^\ast \colon \catPerffilt \longrightarrow \catPerfgr \ , \quad \catPerffilt \longrightarrow \catPerf \ .
\end{align}
The group structure on $\sfB \G_m$ endows $\catPerfgr$ with a Künneth monoidal structure. The same holds for $\catPerffilt$.
With respect to these monoidal structures, the above functors are symmetric monoidal.

\begin{definition}
	Let $\calC$ be a stable $\C$-linear $\infty$-category.
	A \textit{lax filtered structure on $\calC$} is the given of an $\infty$-category $\calC^\bullet \in \catPerffilt \Mod( \Cat_\infty^{\mathsf{st}} )$ equipped with a functor
	\begin{align}
		\Phi \colon \calC^\bullet \otimes_{\catPerffilt} \catPerf \longrightarrow \calC \ . 
	\end{align}
	We refer to the datum $(\calC, \calC^\bullet, \Phi)$ as the datum of a \textit{lax filtered stable ($\C$-linear) $\infty$-category}.
	We say that a lax filtered $\infty$-category is \textit{filtered} if $\Phi$ is an equivalence.
\end{definition}

\begin{definition}
	Let $(\calC, \calC^\bullet, \Phi)$ be a lax filtered stable $\infty$-category.
	A \textit{lax associated graded category} is the given of an $\infty$-category $\calG \in\catPerfgr \Mod( \Cat_\infty^{\mathsf{st}} )$ together with a morphism
	\begin{align}
		 \Psi \colon \calC^\bullet \otimes_{\catPerffilt} \catPerfgr \longrightarrow \calG \ . 
	\end{align}
	We say that $(\calG, \Psi)$ is the \textit{associated graded} if the morphism $\Psi$ is an equivalence.
\end{definition}

\subsection{Hodge filtration}

Let $X$ be a smooth projective complex curve. We will apply the formalism in the previous section with $\calC=\catCohb_{\mathsf{pro}}(\bfCoh(X_\dR))$ and $\calG=\catCohb_{\mathsf{pro},\,\C^\ast}(\bfCoh^{\mathsf{ss},\, 0}(X_\Dol))$.

Let $X_\Del$ be the deformation to the normal bundle of the map $X \to X_\dR$ as constructed in \cite[\S 9.2.4]{Gaitsgory_Rozenblyum_Study_II}.
Then $X_\Del$ admits a canonical $\G_m$-action and it is equipped with a canonical $\G_m$-equivariant map $X_\Del \to \A^1$.
We refer to $X_\Del$ as \emph{Deligne's shape} of $X$.
Furthermore, we let
\begin{align}
	X_{\Del, \, \G_m} \coloneqq [X_\Del / \G_m] 
\end{align}
be the quotient by the action of $\G_m$.
We refer to $X_{\Del, \, \G_m}$ as the \emph{equivariant} Deligne shape of $X$.
See also \cite[\S\ref*{shapes-s:Del-definition}]{Porta_Sala_Shapes} for a more in-depth treatment of the Deligne shape.
We define $\bfCoh_{/\A^1}(X_\Del)$ as the functor
\begin{align}
	(\dAff_{/\A^1})\op \longrightarrow \cS 
\end{align}
sending $S \to \A^1$ to the maximal $\infty$-groupoid $\catCoh_S( S \times_{\A^1} X_\Del )^\simeq$ contained in the $\infty$-category of families of coherent sheaves on $S \times_{\A^1} X_\Del$ that are flat relative to $S$.
Similarly, we define $\bfCoh_{/[\A^1 / \G_m]}( X_{\Del, \G_m} )$ as the functor
\begin{align}
	 (\dAff_{/[\A^1 / \G_m]})\op \longrightarrow \cS 
\end{align}
sending $S \to [\A^1 / \G_m]$ to the maximal $\infty$-groupoid $\catCoh_S( S \times_{[\A^1 / \G_m]} X_{\Del, \G_m} )^\simeq$ contained in the $\infty$-category of families of coherent sheaves on $S \times_{[\A^1 / \G_m]} X_{\Del, \G_m}$ that are flat relative to $S$.

\begin{proposition}
	The derived stack $\bfCoh_{/\A^1}(X_\Del)$ (resp.\ $\bfCoh_{/[\A^1 / \G_m]}(X_{\Del, \, \G_m} )$) is a geometric derived stack, locally of finite presentation in $\dSt_{/ \A^1}$ (resp.\ $\dSt_{/[\A^1 / \G_m]}$).
\end{proposition}

\begin{proof}
	Given a morphism of derived stacks $Y \to S$ we write
\begin{align}
		\bfPerf_{/S}(Y) \coloneqq \bfMap_{/S}( Y, \bfPerf \times S) \ . 
\end{align}
	
	The canonical map $X\times \A^1\to X_\Del$ (resp.\ $X\times [\A^1/\G_m]\to X_{\Del, \, \G_m}$) is a flat effective epimorphism as a $\A^1$-map (resp.\ $[\A^1/\G_m]$-map) (cf.\  \cite[\S\ref*{shapes-s:Del-definition}]{Porta_Sala_Shapes}). Thus, $\bfCoh_{/\A^1}(X_\Del)$ and $\bfCoh_{/[\A^1 / \G_m]}(X_{\Del, \, \G_m})$ fit into the pullback squares (cf.\ Lemma~\ref{lem:universally_flat_atlas} and Corollary~\ref{cor:coherent_over_smooth_are_perfect})
	\begin{align}
		\begin{tikzcd}[ampersand replacement = \&]
			\bfCoh_{/\A^1}(X_\Del) \arrow{r} \arrow{d} \& \bfCoh(X) \times \A^1 \arrow{d} \\
			\bfPerf_{/\A^1}(X_\Del) \arrow{r} \& \bfPerf(X) \times \A^1
		\end{tikzcd}\quad\textrm{and}\quad 
		\begin{tikzcd}[ampersand replacement = \&]
			\bfCoh_{/[\A^1 / \G_m]}(X_{\Del, \, \G_m}) \arrow{r} \arrow{d} \& \bfCoh(X) \times [\A^1/\G_m] \arrow{d} \\
			\bfPerf_{/[\A^1 / \G_m]}( X_{\Del, \, \G_m} ) \arrow{r} \& \bfPerf(X) \times [\A^1/\G_m]
		\end{tikzcd}\ .
	\end{align}
		Since $\bfPerf_{/\A^1}(X_\Del)$ and $\bfPerf(X)\times \A^1$ (resp.\ $\bfPerf_{/[\A^1 / \G_m]}(X_{\Del, \, \G_m})$ and $\bfPerf(X)\times [\A^1/\G_m]$) are geometric as derived stacks over $\A^1$ (resp.\ $[\A^1 / \G_m]$) by  \cite[\S\ref*{shapes-s:Deligne_representability}]{Porta_Sala_Shapes} and \cite[Corollary~3.29]{Toen_Vaquie_Moduli} respectively, and $\bfCoh(X) \times \A^1$ (resp.\ $\bfCoh(X) \times [\A^1/\G_m]$) is geometric in $\dSt_{/ \A^1}$ (resp.\ $\dSt_{/[\A^1 / \G_m]}$) because of Proposition~\ref{prop:coh_geometric}, we conclude that $\bfCoh_{/\A^1}(X_\Del)$ and $\bfCoh_{/[\A^1 / \G_m]}(X_{\Del, \, \G_m})$ are geometric as well.
\end{proof}

We have canonical maps $\bfCoh_{/\A^1}(X_\Del)\to \A^1$ and $\bfCoh_{/[\A^1 / \G_m]}(X_{\Del,\, \G_m})\to [\A^1/\G_m]$. Unraveling the definitions, we see that
\begin{align}
	\bfCoh_{/\A^1}(X_\Del) \times_{\A^1} \{0\} \simeq \bfCoh(X_\Dol) \quad \text{and} \quad \bfCoh_{/\A^1}(X_\Del) \times_{\A^1} \{1\} \simeq \bfCoh(X_\dR) \ ,
\end{align} 
while
\begin{align}\label{eq:fibers}
	\begin{split}
		\bfCoh_{/[\A^1 / \G_m]}(X_{\Del,\, \G_m}) &\times_{[\A^1/\G_m]} \sfB \G_m \simeq \bfCoh_{\C^\ast}(X_\Dol) \\ 
		\bfCoh(X_{\Del, \, \G_m}) &\times_{[\A^1 / \G_m]} [\G_m/\G_m] \simeq \bfCoh(X_\dR)\times \sfB \G_m \ .
	\end{split}
\end{align} 

We also consider the open substack $\bfCoh_{/\A^1}^\ast(X_\Del)\subset \bfCoh_{/\A^1}(X_\Del)$ for which the fiber at zero is the derived moduli stack $\bfCoh^{\mathsf{ss}, \,0}(X_\Dol)$ of semistable Higgs bundles on $X$ of degree zero (cf.\ \cite[\S7]{Simpson_Geometricity}). 

Similarly, we can define the derived moduli stacks of extensions of Deligne's $\lambda$-connections. Thus, we have the convolution diagram in $\dSt_{/ \A^1}$:
\begin{align}
\begin{tikzcd}[ampersand replacement = \&]
{} \& \bfCohext_{/\A^1}(X_\Del) \arrow{dl}[swap]{p} \arrow{dr}{q} \\
\bfCoh_{/\A^1}(X_\Del)\times_{\A^1} \bfCoh_{/\A^1}(X_\Del) \& \& \bfCoh_{/\A^1}(X_\Del)
\end{tikzcd}
\end{align}
and the convolution diagram in $\dSt_{/[\A^1 / \G_m]}$:
\begin{align}
\begin{tikzcd}[ampersand replacement = \&,column sep = small]
{} \& \bfCohext_{/[\A^1/ \G_m]}(X_{\Del,\, \G_m}) \arrow{dl}[swap]{p} \arrow{dr}{q} \\
\bfCoh_{/[\A^1/ \G_m]}(X_{\Del,\, \G_m})\times_{[\A^1/\G_m]} \bfCoh_{/[\A^1/ \G_m]}(X_{\Del,\, \G_m}) \& \& \bfCoh_{/[\A^1/ \G_m]}(X_{\Del,\, \G_m})
\end{tikzcd}
\end{align}
Because of Corollaries~\ref{cor:dR-lci} and \ref{cor:Dol-lci}, it follows that the map $p$ above is derived lci. A similar result holds when we restrict to the open substack $\bfCoh^\ast_{/\A^1}(X_\Del)$ and the corresponding open substack of extensions. Following the same arguments as in \S\ref{s:categorifiedHall}, we can encode such convolution diagrams into $2$-Segal objects, and obtain the following: 

\begin{proposition}
	Let $X$ be a smooth projective complex curve. Then
	\begin{itemize}\itemsep0.2cm
		\item there is a $2$-Segal object $\calS_\bullet \bfCoh_{/\A^1}(X_\Del)$ which endows $\bfCoh_{/\A^1}(X_\Del)$ with the structure of an $\mathbb E_1$-monoid object in $\Corr^\times\big(\dGeom_{/\A^1}\big)_{\mathsf{lci,rps}}$;
		\item there is a $2$-Segal object $\calS_\bullet \bfCoh_{/[\A^1/ \G_m]}(X_{\Del,\, \G_m})$ which endows $\bfCoh_{/[\A^1/ \G_m]}(X_{\Del,\, \G_m})$ with the structure of an $\mathbb E_1$-monoid object in $\Corr^\times\big(\dGeom_{/[\A^1/\G_m]}\big)_{\mathsf{lci,rps}}$.
	\end{itemize}
	A similar result holds for $\bfCoh^\ast_{/\A^1}(X_\Del)$ and $\bfCoh^\ast_{/[\A^1/ \G_m]}(X_{\Del,\, \G_m})$.
\end{proposition}

\begin{corollary}
	$\catCohb_{\mathsf{pro}}( \bfCoh_{/\A^1}(X_\Del) )$ and $\catCohb_{\mathsf{pro}}( \bfCoh_{/[\A^1/ \G_m]}(X_{\Del,\, \G_m}) )$ are $\mathbb E_1$-monoid objects in $\mathsf{Pro}(\Cat_\infty^{\mathsf{st}} )$. A similar result holds for $\bfCoh^\ast_{/\A^1}(X_\Del)$ and $\bfCoh^\ast_{/[\A^1/ \G_m]}(X_{\Del,\, \G_m})$.
\end{corollary}

By combining the results above with \eqref{eq:fibers}, we get:
\begin{theorem}
	Let $X$ be a smooth projective complex curve. Then
	\begin{align}
		\catCohb_{\mathsf{pro},\,\C^\ast}( \bfCoh^{(\ast)}_{/\A^1}(X_\Del) )\simeq \catCohb_{\mathsf{pro}}( \bfCoh^{(\ast)}_{/[\A^1/ \G_m]}(X_{\Del,\, \G_m}) )
	\end{align}
	is a module over $\catPerffilt$ and we have $\mathbb E_1$-monoidal functors:
	\begin{align}
	\Phi^{(\ast)} \colon \catCohb_{\mathsf{pro},\,\C^\ast}( \bfCoh^{(\ast)}_{/\A^1}(X_\Del) ) & \otimes_{\catPerffilt} \catPerf_\C \longrightarrow \catCohb( \bfCoh(X_\dR) )\ ,\\[3pt]
	\Psi^{(\ast)} \colon \catCohb_{\mathsf{pro},\,\C^\ast}( \bfCoh_{/\A^1}^{(\ast)}(X_\Del) ) & \otimes_{\catPerffilt} \catPerfgr \longrightarrow \catCohb_{\C^\ast}( \bfCoh^{(\mathsf{ss}, \, 0)}(X_\Dol) )\ .
	\end{align}
\end{theorem}

Following Simpson \cite[\S7]{Simpson_Geometricity}, we expect the following to be true:
\begin{conjecture}[Cat-HA version of the non-abelian Hodge correspondence]
	The morphisms $\Phi^\ast$ and $\Psi^\ast$ are equivalences, i.e., $\catCohb_{\mathsf{pro}}(\bfCoh(X_\dR))$ is filtered by $\catCohb_{\mathsf{pro},\,\C^\ast}(\bfCoh^\ast_{/\A^1}(X_\Del))$ with associated graded $\catCohb_{\mathsf{pro},\,\C^\ast}(\bfCoh^{\mathsf{ss}, \, 0}(X_\Dol))$.
\end{conjecture}

\bigskip\section{A Cat-HA version of the Riemann--Hilbert correspondence}\label{s:RH}

In this section we briefly consider a complex analytic analogue of the theory developed so far.
Thanks to the foundational work on derived analytic geometry \cite{DAG-IX,Porta_Yu_Higher_analytic_stacks_2014,Porta_GAGA,Porta_Holstein_Mapping} most of the constructions and results obtained so far carry over in the analytic setting.
After sketching how to define the derived analytic stack of coherent sheaves, we focus on two main results.
The first, is the construction of a monoidal functor between the algebraic and the analytic categorified Hall algebras coming from nonabelian Hodge theory.
The second is to provide an equivalence between the analytic categorified Betti algebra and the de Rham one.
This equivalence is an instance of the Riemann--Hilbert correspondence, and it is indeed induced by the main results of \cite{Porta_Derived,Porta_Holstein_Mapping}.

\subsection{The analytic stack of coherent sheaves}

We refer to \cite[\S2]{Porta_Holstein_Mapping} for a review of derived analytic geometry.
Using the notations introduced there, we denote by $\bfAnPerf$ the complex analytic stack of perfect complexes (see \S4 in \textit{loc.\ cit.}).
Similarly, given derived analytic stacks $X$ and $Y$, we let $\bfAnMap(X,Y)$ be the derived analytic stack of morphisms between them.

Fix a derived geometric analytic stack $X$.
We wish to define a substack of $\bfAnPerf(X) \coloneqq \bfAnMap(X, \bfAnPerf)$ classifying families of coherent sheaves on $X$.
The same ideas of \S\ref{s:coh} apply, but as usual some extra care to deal with the notion of flatness in analytic geometry is needed.

\begin{definition}
	Let $S$ be a derived Stein space and let $f \colon X \to S$ be a morphism of derived analytic stacks.
	We say that an almost perfect complex $\calF \in \catAPerf(X)$ has \textit{tor-amplitude within $[a,b]$ relative to $S$} (resp.\ \textit{tor-amplitude $\le n$ relative to $S$}) if for every $\calG \in \catAPerf^\heartsuit(S)$ one has
	\begin{align}
		\pi_i( \calF \otimes f^\ast \calG ) \simeq 0 \quad i \notin [a,b] \quad (\text{resp.\ } i > n) \ .
	\end{align}
	We let $\catAPerf_S^{\le n}(X)$ denote the full subcategory of $\catAPerf(X)$ spanned by those sheaves of almost perfect modules $\calF$ on $X$ having tor-amplitude $\le n$ relative to $S$.
	We write
	\begin{align}
		\catCoh_S(X) \coloneqq \catAPerf_S^{\le 0}(X) \ , 
	\end{align}
	and we refer to it as the $\infty$-category of flat families of coherent sheaves on $X$ relative to $Y$.
\end{definition}

The above definition differs from \cite[Definitions~7.1 \& 7.2]{Porta_Yu_Mapping}.
We prove in Lemma~\ref{lem:relative_tor_amplitude_geometric_case_analytic} that they are equivalent.

\begin{lemma} \label{lem:flat_effective_epimorphism_analytic}
	Let $X$ be a derived analytic stack, let $S \in \dStn$ and let $f \colon X \to S$ be a morphism in $\dAnSt$.
	Assume that there exists a flat\footnote{A morphism $f \colon U \to X$ of derived analytic stacks is said to be flat if the pullback functor $f^\ast \colon \catAPerf(X) \to \catAPerf(U)$ is $t$-exact.} effective epimorphism $u \colon U \to X$.
	Then $\calF \in \catAPerf(X)$ has tor-amplitude within $[a,b]$ relative to $S$ if and only if $u^\ast(\calF)$ has tor-amplitude within $[a,b]$ relative to $S$.
\end{lemma}

\begin{proof}
	Let $\calG \in \catAPerf^\heartsuit(S)$.
	Then since $u$ is a flat effective epimorphism, we see that the pullback functor
	\begin{align}  
		u^\ast \colon \catAPerf(X) \longrightarrow \catAPerf(U)
	\end{align} 
	is $t$-exact and conservative.
	Therefore $\pi_i( \calF \otimes f^\ast \calG ) \simeq 0$ if and only if
	\begin{align}  
		u^\ast( \pi_i(\calF \otimes f^\ast \calG) ) \simeq \pi_i( u^\ast(\calF) \otimes u^\ast f^\ast \calG ) \simeq 0 \ . 
	\end{align} 
	The conclusion follows.
\end{proof}

\begin{lemma} \label{lem:relative_tor_amplitude_geometric_case_analytic}
	Let $f \colon X \to S$ be a morphism of derived analytic stacks.
	Assume that $X$ is geometric and that $S$ is a derived Stein space.
	Then $\calF \in \catQCoh(X)$ has tor-amplitude within $[a,b]$ relative to $S$ if and only if there exists a smooth Stein covering $\{u_i \colon U_i \to X\}$ such that $\Gamma(U_i; u_i^\ast \calF)$ has tor-amplitude within $[a,b]$ as $\Gamma(S; \scrO_S^{\mathsf{alg}})$-module.
\end{lemma}

\begin{proof}
	Using Lemma~\ref{lem:flat_effective_epimorphism_analytic}, we can reduce ourselves to the case where $X$ is a derived Stein space.
	Notice that $\calF \otimes_{\scrO_X} f^\ast \calG \in \catAPerf(X)$.
	Therefore, Cartan's theorem B applies and shows that $\pi_i( \calF \otimes_{\scrO_X} f^\ast \calG ) = 0$ if and only if $\pi_i( f_\ast( \calF \otimes_{\scrO_X} f^\ast \calG ) ) = 0$.
	Observe now that there is a canonical morphism
	\begin{align}
	\eta_{\calF, \calG} \colon f_\ast(\calF) \otimes_{\scrO_S} \calG \longrightarrow f_\ast( \calF \otimes_{\scrO_X} f^\ast \calG ) \ .
	\end{align}
	When $\calG = \scrO_S$ this morphism is obviously an equivalence.
	We claim that it is an equivalence for any $\calG \in \catAPerf(S)$.
	
	This question is local on $S$.
	Write $A_S \coloneqq \Gamma(S; \scrO_S^{\mathsf{alg}})$.
	Using \cite[Lemma~4.12]{Porta_Holstein_Mapping} we can restrict ourselves to the case where $\calG \simeq \varepsilon_S^\ast(M)$ for some $M \in \catAPerf(A_S)$.
	Here $\varepsilon_S^\ast \colon A_S \Mod \to \scrO_S \Mod$ is the functor introduced in \cite[\S4.2]{Porta_Holstein_Mapping}.
	In this case, we see that since $\eta_{\calF, \calG}$ is an equivalence when $\calG = \scrO_S$, it is also an equivalence whenever $M$ (and hence $\calG$) is perfect.
	In the general case, we use \cite[7.2.4.11(5)]{Lurie_Higher_algebra} to find a simplicial object $P^\bullet \in \Fun(\mathbf \Delta\op, \catAPerf(A_S))$ such that
	\begin{align}
		\vert P^\bullet \vert \simeq M \ . 
	\end{align}
	Write $\calP^\bullet \coloneqq \varepsilon_S^\ast(P^\bullet)$.
	Reasoning as in \cite[Corollary 3.5]{Porta_Yu_Mapping}, we deduce that
	\begin{align}
		\vert \calP^\bullet \vert \simeq \varepsilon_S^\ast( M ) \simeq \calG \ . 
	\end{align}
	It immediately follows that
	\begin{align}
		\calF \otimes_{\scrO_X} f^\ast \calG \simeq \vert \calF \otimes_{\scrO_X} f^\ast \calP^\bullet \vert \ ,
	\end{align}
	and proving that $\eta_{\calF, \calG}$ is an equivalence is reduced to checking that $f_\ast$ preserves the above colimit.
	Since the above diagram as well as its colimit takes values in $\catAPerf(X)$, we can apply Cartan's theorem B.
	The descent spectral sequence degenerates, and therefore the conclusion follows.
\end{proof}

\begin{corollary} \label{cor:relative_tor_amplitude_truncation_analytic}
	Let $f \colon X \to S$ be a morphism as in the previous lemma.
	Let $j \colon \trunc{S} \to S$ be the canonical morphism and consider the pullback diagram
	\begin{align}
		 \begin{tikzcd}[ampersand replacement = \&]
			X_0 \arrow{d}{f_0} \arrow{r}{i} \& X \arrow{d}{f} \\
			\trunc{S} \arrow{r}{j} \& S 
		\end{tikzcd}\ .
	\end{align}
	Then an almost perfect complex $\calF \in \catAPerf(X)$ has tor-amplitude within $[a,b]$ relative to $S$ if and only if $i^\ast \calF$ has tor-amplitude within $[a,b]$ relative to $\trunc{S}$.
\end{corollary}

\begin{proof}
	The map $j$ is a closed immersion and therefore so is $i$.
	In particular, for any $\calG \in \catAPerf(\trunc{S})$ the canonical map
	\begin{align}
		f^\ast j_\ast (\calG) \longrightarrow i_\ast f_0^\ast(\calG) 
	\end{align}
	is an equivalence.\footnote{Ultimately, this can be traced back to the unramifiedness of the analytic pregeometry $\calT_{\mathsf{an}}(\C)$. See \cite[Lemma~6.1]{Porta_Yu_Mapping} for an argument in the non-archimedean case.}
	Moreover, the projection formula holds for $i$ and $i_\ast$ is $t$-exact.
	Suppose that $\calF$ has tor-amplitude within $[a,b]$ relative to $S$.
	Let $\calG \in \catAPerf^\heartsuit(\trunc{S})$.
	Then
	\begin{align}
		 i_\ast ( i^\ast \calF \otimes_{\scrO_{X_0}} f_0^\ast \calG ) \simeq \calF \otimes_{\scrO_X} i_\ast f_0^\ast \calG \simeq \calF \otimes_{\scrO_X} f^\ast j_\ast \calG\ .
	\end{align}
	Since $j_\ast$ is $t$-exact, $j_\ast \calG \in \catAPerf^\heartsuit(S)$, and therefore the above tensor product is concentrated in homological degree $[a,b]$.
	In other words, $i^\ast \calF$ has tor-amplitude within $[a,b]$ relative to $\trunc{S}$.
	For the converse, it is enough to observe that $j_\ast$ induces an equivalence $\catAPerf^\heartsuit(\trunc{S}) \simeq \catAPerf^\heartsuit(S)$.
\end{proof}

\begin{definition}
	Let $S \in \dStn$ and let $f \colon X \to S$ be a morphism in $\dAnSt$.
	A morphism $u \colon U \to X$ is said to be \textit{universally flat relative to $S$} if for every derived Stein space $S' \in \dSt$ and every morphism $S' \to S$ the induced map $S' \times_S U \to S' \times_S X$ is flat.
	We say that a morphism $u \colon U \to X$ is \textit{universally flat} if it is universally flat relative to $\Spec(\C)$.
\end{definition}

\begin{rem}
	Let $S$ be an affine derived scheme and let $f \colon X \to S$ and $u \colon U \to X$ be morphisms of derived stacks.
	If $f$ is flat, then for every morphism $S' \to S$ of affine derived schemes, the morphism $S' \times_S U \to S' \times_S X$ is flat.
	See \cite[Proposition~\ref*{shapes-prop:local_tor_amplitude}]{Porta_Sala_Shapes}.
	In the analytic setting, it is difficult to prove a similar result, because it essentially relies on base change for maps between derived affine schemes (see \cite[Proposition~\ref*{shapes-prop:representable_by_affine}]{Porta_Sala_Shapes}), which is not available in the analytic setting.
\end{rem}

\begin{corollary} \label{cor:base_change_coherent_analytic}
	Let $X$ be a derived analytic stack and let $S$ be a derived Stein space.
	Assume that there exists a universally flat effective epimorphism $u \colon U \to X$ where $U$ is geometric and underived.
	Let $f \colon S' \to S$ be a morphism of derived Stein spaces and consider the pullback
	\begin{align}
		\begin{tikzcd}[ampersand replacement = \&]
			X \times S' \arrow{r}{g} \arrow{d}{q} \& X \times S \arrow{d}{p} \\
			S' \arrow{r}{f} \& S 
		\end{tikzcd}\ .
	\end{align}
	If $\calF \in \catAPerf(X \times S)$ has tor-amplitude within $[0,0]$ relative to $S$, then $g^\ast \calF$ has tor-amplitude within $[0,0]$ relative to $S'$.
\end{corollary}

\begin{proof}
	Since $u \colon U \to X$ is universally flat, the morphism $U \times S \to X \times S$ and $U \times S' \to X \times S'$ are flat.
	Therefore Lemma~\ref{lem:flat_effective_epimorphism_analytic} shows that we can restrict ourselves to the case $X = U$.
	Using Corollary~\ref{cor:relative_tor_amplitude_truncation_analytic}, we can reduce the problem to the case where $S$ and $S'$ are underived.
	Since the question is local on $X$, we can furthermore assume that $X$ is a Stein space.
	At this point, the conclusion follows directly from \cite[\S8.3, Proposition~3]{Douady_Hilbert_scheme}.
\end{proof}

Using the above corollary, we can therefore define a derived analytic stack $\bfAnCoh(X)$, which is a substack of $\bfAnPerf(X)$.

In what follows, we will often restrict ourselves to the study of $\bfAnCoh(X\an)$, where now $X$ is an algebraic variety.
Combining \cite[Proposition~5.2 \& Theorem~5.5]{Porta_Holstein_Mapping} we see that if $X$ is a proper complex scheme, then there is a natural equivalence\footnote{The derived analytification functor has been firstly introduced in \cite[Remark 12.26]{DAG-IX} and studied extensively in \cite[\S4]{Porta_GAGA}. For a review, see \cite[\S3.1]{Porta_Holstein_Mapping}.}
\begin{align}\label{eq:isomuX}
	\bfPerf(X)\an \simeq \bfAnPerf(X\an) \ .
\end{align}
We wish to extend this result to $\bfCoh(X)\an$ and $\bfAnCoh(X\an)$.
Let us start by constructing the map between them.
The map $\bfPerf(X)\an \to \bfAnPerf(X\an)$ is obtained by adjunction from the map
\begin{align}
	\bfPerf(X) \longrightarrow \bfAnPerf(X\an) \circ (-)\an  \ ,
\end{align}
which, for $S \in \dAff^{\mathsf{afp}}$, is induced by applying $(-)^\simeq \colon \Cat_\infty \to \cS$ to the analytification functor
\begin{align}
	\catPerf(X \times S) \longrightarrow \catPerf(X\an \times S\an) \ .
\end{align}
It is therefore enough to check that this functor respects the two subcategories of families of coherent sheaves relative to $S$ and $S\an$, respectively.

\begin{lemma} \label{lem:analytification_relative_tor_amplitude}
	Let $f \colon X \to S$ be a morphism of derived complex stacks locally almost of finite presentation.
	Suppose that $X$ is geometric and that $S$ is affine.
	Then $\calF \in \catAPerf(S)$ has tor-amplitude within $[a,b]$ relative to $S$ if and only if $\calF\an \in \catAPerf(X\an)$ has tor-amplitude within $[a,b]$ relative to $S\an$.
\end{lemma}

\begin{proof}
	Suppose first that $\calF\an$ has tor-amplitude within $[a,b]$ relative to $S\an$.
	Let $\calG \in \catAPerf^\heartsuit(S)$.
	Then we have to check that $\pi_i( \calF \otimes_{\scrO_X} f^\ast \calG ) = 0$ for $i \notin [a,b]$.
	As the analytification functor $(-)\an$ is $t$-exact and conservative, this is equivalent to checking that we have $\pi_i( (\calF \otimes_{\scrO_X} f^\ast \calG)\an ) = 0$.
	But,
	\begin{align} \label{eq:analytification_relative_tor_amplitude}
		( \calF\otimes_{\scrO_X} f^\ast \calG )\an \simeq \calF\an \otimes_{\scrO_{X\an}} f^{\mathsf{an} \ast}(\calG\an) \ ,
	\end{align}
	and the conclusion follows from the fact that $\calG\an \in \catAPerf^\heartsuit(S\an)$.
	
	Suppose now that $\calF$ has tor-amplitude within $[a,b]$ relative to $S = \Spec(A)$.
	We can check that $\calF\an$ has tor-amplitude within $[a,b]$ relative to $S\an$ locally on $S\an$.
	For every derived Stein open subspace $j_U \colon U \subset S\an$, write
	\begin{align}
		A_U \coloneqq \Gamma(U; \scrO_{S\an}^{\mathsf{alg}} \vert_U) \ . 
	\end{align}
	Write $a_U \colon \Spec(A_U) \to S$ for the morphism induced by the canonical map $A \to A_U$.
	Consider the two pullback squares
	\begin{align}
		\begin{tikzcd}[ampersand replacement = \&]
			X_U \arrow{r}{b_U} \arrow{d}{f_U} \& X \arrow{d}{f} \\
			\Spec(A_U) \arrow{r}{a_U} \& S
			\end{tikzcd} \quad , \quad \begin{tikzcd}[ampersand replacement = \&]
			X\an_U \arrow{r}{i_U} \arrow{d}{f_U\an} \& X\an \arrow{d}{f\an} \\
			U \arrow{r}{j_U} \& S\an 
		\end{tikzcd}\ .
	\end{align}
	There is a natural analytification functor relative to $U$
	\begin{align}
		(-)\an_U \colon \catAPerf(X_U) \longrightarrow \catAPerf(X\an_U) \ .
	\end{align}
	Moreover, the canonical map
	\begin{align}
		 i_U^\ast( \calH\an ) \longrightarrow ( b_U^\ast( \calH ) )\an_U
	\end{align}
	is an equivalence for every $\calH \in \catAPerf(X)$.
	
	Fix now $\calG \in \catAPerf(S\an)$.
	If $\calG \simeq (\widetilde{\calG})\an$ for some $\widetilde{\calG} \in \catAPerf^\heartsuit(S)$, then the equivalence \eqref{eq:analytification_relative_tor_amplitude} shows that
	\begin{align}
		\pi_i( \calF\an \otimes_{\scrO_{X\an}} f^{\mathsf{an}*}(\calG) ) = 0 
	\end{align}
	for $i \notin [a,b]$.
	In the general case, we choose a double covering $\{V_i \Subset U_i \Subset S\an\}$ by relatively compact derived Stein open subspaces of $S\an$.
	Using \cite[Lemma~4.12]{Porta_Holstein_Mapping} we can find $\widetilde{\calG}_i \in \catAPerf(A_{V_i})$ such that $\calG \vert_{V_i} \simeq \varepsilon_{V_i}^\ast( \widetilde{\calG}_i )$.
	Here $\varepsilon_{V_i}^\ast$ is the functor introduced in \cite[\S4.2]{Porta_Holstein_Mapping}.
	At this point, we observe that Lemma \ref{lem:base_change_tor_amplitude} guarantees that $b_U^\ast(\calF)$ has tor-amplitude within $[a,b]$ relative to $\Spec(A_U)$.
	The conclusion then follows from the argument given in the first case.
\end{proof}

As a consequence, we find a morphism
\begin{align} \label{eq:comparing_analytified_coh_with_ancoh}
	\bfCoh(X) \longrightarrow \bfAnCoh(X\an) \circ (-)\an \ ,
\end{align}
which by adjunction induces
\begin{align}
	\mu_X \colon \bfCoh(X)\an \longrightarrow \bfAnCoh(X\an) \ ,
\end{align}
which is compatible with the morphism $\bfPerf(X)\an \to \bfAnPerf(X\an)$.

\begin{proposition} \label{prop:analytification_stack_coherent_sheaves}
	If $X$ is a proper complex scheme, the natural transformation
	\begin{align}
		\mu_X \colon \bfCoh(X)\an \to \bfAnCoh(X\an)
	\end{align}
	is an equivalence.
\end{proposition}

\begin{proof}
	Reasoning as in the proof of the equivalence \eqref{eq:isomuX} in \cite[Proposition~5.2]{Porta_Holstein_Mapping}, we reduce ourselves to check that for every derived Stein space $U \in \dStn$ and every compact derived Stein subspace $K$ of $U$, the natural morphism
	\begin{align}
		\fcolim_{K \subset V \subset U} \catCoh_{\Spec(A_V)}( \Spec(A_V) \times X ) \longrightarrow \fcolim_{K \subset V \subset U} \catCoh_V( V \times X\an )
	\end{align}
	is an equivalence in $\Ind( \Cat_\infty^{\mathsf{st}} )$.
	Here the colimit is taken over the family of open Stein neighborhoods $V$ of $K$ inside $U$.
	Using \cite[Lemma~5.13]{Porta_Holstein_Mapping} we see that for every $V$, the functor
	\begin{align}
		\catCoh_{\Spec(A_V)}( \Spec(A_V) \times X ) \longrightarrow \catCoh_V( V \times X\an )
	\end{align}
	is fully faithful.
	The conclusion now follows by combining \cite[Proposition~5.15]{Porta_Holstein_Mapping} and the ``only if'' direction of Lemma~\ref{lem:analytification_relative_tor_amplitude}.
\end{proof}

\subsection{Categorical Hall algebras in the $\C$-analytic setting}

Let $X\in \dAnSt$ be a derived analytic stack. In the previous section, we have introduced the analytic stack $\bfAnCoh(X)$ parameterizing families of sheaves of almost perfect modules over $X$ of tor-amplitude $\le 0$ relative to the base.
Similarly, we can define the derived analytic stacks $\bfAnPerf^{\mathsf{ext}}$, $\bfAnPerf^\mathsf{ext}(X)$, and $\bfAnCohext(X)$.
We deal directly with the Waldhausen construction.

We define the simplicial derived analytic stack
\begin{align}
	\calS_\bullet \bfAnPerf \colon \dStn\op \longrightarrow \Fun(\mathbf \Delta\op, \cS)
\end{align}
by sending an object $[n] \in \mathbf \Delta$ and a derived Stein space $S$ to the full subcategory of\footnote{See \S\ref{ss:convolution_structure} for the notations used here.}
\begin{align}
	\calS_n \catPerf(S) \hookrightarrow \Fun(T_n, \catPerf(S)) \ .
\end{align}
Since each $T_n$ is a finite category, \cite[Corollary~7.2]{Porta_Holstein_Mapping} and the flatness of the relative analytification proven in \cite[Proposition~4.17]{Porta_Yu_Representability} imply that the natural map
\begin{align}
	(\calS_\bullet \bfPerf)\an \longrightarrow \calS_\bullet \bfAnPerf
\end{align}
is an equivalence.
Moreover, \cite[Proposition~7.3]{Porta_Holstein_Mapping} implies that the analytification commutes with the limits appearing in the $2$-Segal condition.
We can therefore deduce that $\calS_\bullet \bfAnPerf$ is a $2$-Segal object in $\dAnSt$.
From this, we deduce immediately that for every derived analytic stack $X$, $\bfAnMap(X, \calS_\bullet \bfAnPerf)$ is again a $2$-Segal object.
At this point, the same reasoning of Lemma~\ref{lem:coh_2_Segal_geometric} yields:

\begin{proposition} \label{prop:analytic_convolution_tensor_product}
	Let $X\in \dAnSt$ be a derived geometric analytic stack. Then $\calS_\bullet \bfAnCoh(X)$ is a $2$-Segal object in $\dAnSt$, and therefore it endows the derived analytic stack $\bfAnCoh(X)$ with the structure of an $\mathbb E_1$-monoid object.
\end{proposition}

The morphism \eqref{eq:analytification_relative_tor_amplitude} can be naturally upgraded to a natural transformation
\begin{align}
	\calS_\bullet \bfCoh(X) \longrightarrow \calS_\bullet \bfAnCoh(X\an) \circ (-)\an
\end{align}
in $\Fun(\mathbf \Delta\op, \dSt)$.
By adjunction, we therefore find a morphism of simplicial objects
\begin{align} \label{eq:comparison_analytification_Segal_objects}
	\left( \calS_\bullet \bfCoh(X) \right)\an \longrightarrow \calS_\bullet \bfAnCoh(X\an) \ .
\end{align}

\begin{remark}
	Suppose that $X$ is such that each $\calS_n \bfCoh(X)$ is geometric.
	Then \cite[Proposition 7.3]{Porta_Holstein_Mapping} implies that $(\calS_\bullet \bfCoh(X))\an$ is a $2$-Segal object in $\dAnSt$.
\end{remark}

Let $Y \in \dAnSt$ be a derived analytic stack and let $u \colon U \to Y$ be a flat effective epimorphism from an underived geometric analytic stack $U$. As above, we are able to define the derived stack $\bfAnCoh(Y)$.
Notice that $\bfAnCoh(Y)$ only depends on $Y$ and not on $U$.
However, as in the algebraic case, the proof of the functoriality of $\bfAnCoh(Y)$ relies on the existence of $U$ and on Lemma~\ref{cor:base_change_coherent_analytic}.
In addition, we have
\begin{align}\label{eq:coh-shapes_analytic}
	\bfAnCoh(Y) \simeq \bfAnPerf(Y) \times_{\bfAnPerf(U)} \bfAnCoh(U) \ .
\end{align}
This is the analytic counterpart of Lemma~\ref{lem:universally_flat_atlas}.

Similarly, we can define $\bfAnCohext(Y)$ and $\bfAnBun^\mathsf{ext}(Y)$ and more generally their Waldhausen analogues $\calS_\bullet \bfAnCoh(Y)$ and $\calS_\bullet \bfAnBun(Y)$. 
We immediately obtain:

\begin{proposition}
	Let $Y \in \dAnSt$ be a derived analytic stack and let $u \colon U \to Y$ be a flat effective epimorphism from an underived geometric analytic stack $U$.
	Then $\calS_\bullet \bfAnCoh(Y)$ is a $2$-Segal object and it endows $\bfAnCoh(Y)$ with the structure of an $\mathbb E_1$-monoid object in $\dAnSt$.
\end{proposition}

As a particular case, let $X$ be a smooth proper connected analytic space. Simpson's shapes $X_\B$, $X_\dR$, $X_\Dol$, and $X_\Del$ also exist in derived analytic geometry (as introduced e.g. in \cite[\S~5.2]{Porta_Holstein_Mapping}). We have the following analytic analog of Proposition~\ref{prop:coh_2_Segal}.

\begin{corollary}
	Let $X\in \dAnSt$ be a derived geometric analytic stack and let  $Y$ be one of the following stacks: $X_\B$, $X_\dR$, or $X_\Dol$. Then $\calS_\bullet \bfAnCoh(Y)$ is a $2$-Segal object in $\dAnSt$, and therefore it endows the derived analytic stack $\bfAnCoh(Y)$ with the structure of an $\mathbb E_1$-monoid object.
\end{corollary}

Our next step is to construct the categorified Hall algebras in the analytic setting.
The lack of quasi-coherent sheaves in analytic geometry forces us to consider a variation of the construction considered in \S \ref{ss:categorification}.
We start with the following construction:

\begin{construction}
	Let $\cTdisc(\C)$ be the full subcategory of $\mathsf{Sch}_\C$ spanned by finite dimensional affine spaces $\A^n_\C$.
	Given an $\infty$-topos $\calX$, sheaves on $\calX$ with values in $\CAlg_\C$ can be canonically identified with product preserving functors $\cTdisc(\C) \to \calX$.
	We let $\RTop(\cTdisc(\C))$ denote the $\infty$-category of $\infty$-topoi equipped with a sheaf of derived commutative $\C$-algebras.
	The construction performed in \cite[Notation~2.2.1]{DAG-V} provides us with a functor
	\begin{align}
		\Gamma \colon \big( \RTop( \cTdisc(\C) ) \big)\op \longrightarrow \CAlg_\C \ .
	\end{align}
	Equipping both $\infty$-categories with the cocartesian monoidal structure, we see that $\Gamma$ can be upgraded to a right-lax symmetric monoidal structure.
	Composing with the symmetric monoidal functor $\catQCoh \colon \CAlg_\C \to \Cat_\infty^{\mathsf{st}}$ we therefore obtain a right-lax symmetric monoidal functor
	\begin{align}
		\big( \RTop( \cTdisc(\C) ) \big)\op \longrightarrow \Cat_\infty^{\mathsf{st}} \ .
	\end{align}
	We denote the sheafification of this functor with respect to the \'etale topology on $\RTop(\cTdisc(\C))$ (see \cite[Definition~2.3.1]{DAG-V}) by
	\begin{align}
		\scrO \Mod \colon ( \RTop(\cTdisc(\C)) )\op \longrightarrow \Cat_\infty^{\mathsf{st}} \ .
	\end{align}
	Observe that $\scrO \Mod$ is canonically endowed with a right-lax symmetric monoidal structure.
\end{construction}

Consider the natural forgetful functor
\begin{align}
	(-)^{\mathsf{alg}} \colon \dAn \longrightarrow \RTop(\cTdisc(\C)) \ . 
\end{align}
Equipping both $\infty$-categories with the cartesian monoidal structure, we see that $(-)^{\mathsf{alg}}$ can be upgraded to a left-lax monoidal functor.
We still denote by $\scrO \Mod$ the composition
\begin{align}
\begin{tikzcd}[ampersand replacement = \&, column sep = large]
	(\dAn)\op \arrow{r}{(-)^{\mathsf{alg}}} \& (\RTop(\cTdisc(\C)))\op \arrow{r}{\scrO \Mod} \& \Cat_\infty^{\mathsf{st}}
\end{tikzcd}  \ ,
\end{align}
which canonically inherits the structure of a right-lax monoidal functor.
Given $X \in \dAn$, we denote by $\scrO_X \Mod$ its image via this functor.

This functor admits a canonical subfunctor
\begin{align}
	\catAPerf \colon \dAn\op \longrightarrow \Cat_\infty^{\mathsf{st}}\ ,
\end{align}
which sends a derived $\C$-analytic space to the full subcategory of $\scrO_X \Mod$ spanned by sheaves of almost perfect modules.
Observe that sheaves of almost perfect modules are closed under exterior product, and therefore $\catAPerf$ inherits the structure of a right-lax monoidal functor.
Moreover, if $f \colon X \to Y$ is proper, then \cite[Theorem~6.5]{Porta_GAGA} implies that the functor
\begin{align}
	f_\ast \colon \scrO_X \Mod \longrightarrow \scrO_Y \Mod
\end{align}
restricts to a functor
\begin{align}
	f_\ast \colon \catAPerf(X) \longrightarrow \catAPerf(Y)\ ,
\end{align}
which is right adjoint to $f^\ast$.

\begin{lemma}\label{lem:rightadjointable}
	Let
	\begin{align}
		\begin{tikzcd}[ampersand replacement = \&]
			X' \arrow{r}{u} \arrow{d}{g} \& X \arrow{d}{f} \\
			Y' \arrow{r}{v} \& Y
		\end{tikzcd}
	\end{align}
	be a pullback square in $\dAn$.
	Assume that the truncations of $X, X', Y$ and $Y'$ are separated analytic spaces.
		If $f$ is proper then the commutative diagram
	\begin{align}
		\begin{tikzcd}[ampersand replacement = \&]
		\catAPerf(Y) \arrow{r}{v^\ast} \arrow{d}{f^\ast} \& \catAPerf(Y') \arrow{d}{g^\ast} \\
		\catAPerf(X) \arrow{r}{u^\ast} \& \catAPerf(X')
		\end{tikzcd}
	\end{align}
	is vertically right adjointable.
\end{lemma}

\begin{proof}
	We adapt the proof of \cite[Theorem~6.8]{Porta_Yu_Mapping} to the complex analytic setting.
	The key input is unramifiedness for the pregeometry $\cT_{\mathsf{an}}(\C)$, proven in \cite[Proposition~11.6]{DAG-IX}, which has as a consequence Proposition 11.12(3) in \textit{loc.\ cit.}
	In turn, this implies that the statement of this lemma holds true when $g$ is a closed immersion.
	Knowing this, Steps 1 and 2 of the proof of \cite[Theorem~6.8]{Porta_Yu_Mapping} apply without changes.
	Step 3 applies as well, with the difference that in the $\C$-analytic setting we can reduce to the case where $Y' = \mathsf{Sp}(\C)$ is the $\C$-analytic space associated to a point.
	In particular, the map $Y' = \mathsf{Sp}(\C) \to Y$ is now automatically a closed immersion, and therefore the conclusion follows.
\end{proof}
Let $\dAn^{\mathrm{sep}}$ denote the full subcategory of $\dAn$ spanned by derived $\C$-analytic spaces whose truncation is a separated analytic space. Lemma \ref{lem:rightadjointable} shows that the assumptions of \cite[Theorem~3.2.2(b)]{Gaitsgory_Rozenblyum_Study_I} are satisfied with $\mathsf{horiz} = \mathsf{all}$ and $\mathsf{vert} = \mathsf{proper}$. As a consequence, we can extend $\catAPerf$ to a functor
\begin{align}
	\catAPerf \colon \Corr( \dAn^{\mathrm{sep}} )_{\mathsf{proper}, \mathsf{all}}^{\mathsf{isom}} \longrightarrow \Cat_\infty^{\mathsf{st}} \ .
\end{align}
Moreover, the considerations in \cite[\S~3.3.1]{Gaitsgory_Rozenblyum_Study_I} show that this functor inherits a canonical right-lax monoidal structure.
Using \cite[Theorem~8.6.1.5]{Gaitsgory_Rozenblyum_Study_I} we obtain (via right Kan extension) a right-lax monoidal functor
\begin{align}
	\catAPerf \colon \Corr( \dAnSt )_{\mathsf{all}, \mathsf{rps}}^{\mathsf{isom}} \longrightarrow \Cat_\infty^{\mathsf{st}} \ . 
\end{align}
Here $\mathsf{rps}$ denotes the class of $1$-morphisms representable by proper derived $\C$-analytic spaces.

Given a derived $\C$-analytic space $X$, we denote by $\catCohb(X)$ the full subcategory of $\catAPerf(X)$ spanned by locally cohomologically bounded sheaves of almost perfect modules.

\begin{lemma}
	Let $f \colon X \to Y$ be a morphism of derived geometric analytic stacks.
	If $f$ is lci\footnote{In this setting, it means that the \textit{analytic} cotangent complex $\mathbb L^{\mathsf{an}}_{X/Y}$ introduced in \cite{Porta_Yu_Representability} is perfect and has tor-amplitude within $[0,1]$.} then it has finite tor-amplitude and in particular it induces a functor
	\begin{align}
		f^\ast \colon \catCohb(Y) \longrightarrow \catCohb(X)\ .
	\end{align}
\end{lemma}

\begin{proof}
	The argument of \cite[Corollary~2.9]{Porta_Yu_NQK} applies.
\end{proof}

As a consequence, we obtain a right-lax monoidal functor
\begin{align}
	\catCohb \colon \Corr^\times( \dAnSt )_{\mathsf{rps}, \mathsf{lci}} \longrightarrow \Cat_\infty^{\mathsf{st}} \ . 
\end{align}

Finally, we want to restrict ourselves to derived geometric analytic stacks. In particular, we need that $\bfAnCoh(Y)$ and the corresponding $2$-Segal space to be geometric. So, first note that if $Y \in \dSt$ is a derived stack, then we obtain as before a natural transformation
\begin{align}\label{eq:an_An}
	\calS_\bullet \bfCoh(Y)\an \longrightarrow \calS_\bullet \bfAnCoh(Y)
\end{align}
in $\Fun( \mathbf \Delta\op, \dAnSt )$.

Let $X$ be a smooth and proper complex scheme. By \cite[Proposition~5.2]{Porta_Holstein_Mapping}, $\bfAnPerf(X)$ is equivalent to the analytification $\bfPerf(X)\an$ of the derived stack $\bfPerf(X)=\bfMap(X, \bfPerf)$. Thus, $\bfAnPerf(X)$ is a locally geometric derived stack, locally of finite presentation.

\begin{lemma}\label{lem:equivalence_1}
	The map \eqref{eq:an_An} induces an equivalence $(\calS_\bullet \bfCoh(X))\an \simeq \calS_\bullet \bfAnCoh(X\an)$. In particular, for each $n \ge 0$ the derived analytic stack $\calS_n \bfAnCoh(X\an)$ is locally geometric and locally of finite presentation.
\end{lemma}

\begin{proof}
	When $n = 1$, this is exactly the statement of Proposition \ref{prop:analytification_stack_coherent_sheaves}.
	The proof of the general case is similar, and there are no additional subtleties.
\end{proof}

Let $X$ be a smooth proper connected complex scheme. As proved in \cite[\S5.2]{Porta_Holstein_Mapping}, the analytification functor commutes with Simpson's shape functor, i.e., we have the following canonical equivalences:
\begin{align}
	(X_\dR)\an \simeq (X\an)_\dR\ , \ (X_\B)\an \simeq (X_\B)\an \ , \ (X_\Dol)\an \simeq (X\an)_\Dol\ .
\end{align}
\begin{lemma}\label{lem:equivalence_2}
	Let $\ast\in \{\B, \dR, \Dol\}$. Then the map \eqref{eq:an_An} induces an equivalence $(\calS_\bullet \bfCoh(X_\ast))\an \simeq \calS_\bullet \bfAnCoh((X\an)_\ast)$. In particular, for each $n \ge 0$ the derived analytic stack $\calS_n \bfAnCoh((X\an)_\ast)$ is locally geometric and locally of finite presentation.
\end{lemma}

\begin{proof}
	The proof of Proposition \ref{prop:analytification_stack_coherent_sheaves} applies, with the following caveat: rather than invoking \cite[Lemma~5.13 \& Proposition~5.15]{Porta_Holstein_Mapping}, we instead use Propositions 5.26 (for the de Rham case), 5.28 (for the Betti case) and 5.32 (for the Dolbeault case) in \cite{Porta_Holstein_Mapping}.
\end{proof}

Finally, we are able to give the analytic counterpart of Theorem~\ref{thm:coh_algebra_in_correspondence}:
\begin{theorem} \label{thm:coh_algebra_in_correspondence_analytic}
	Let $Y$ be one of the following derived stacks:
	\begin{enumerate}\itemsep0.2cm
		\item a smooth proper complex scheme of dimension either one or two;
		\item the Betti, de Rham or Dolbeault stack of a smooth projective curve.
	\end{enumerate}
	Then the composition
	\begin{align}
		\catCohb( \bfAnCoh(Y\an) ) \times \catCohb( \bfAnCoh(Y\an) ) &\xrightarrow{\boxtimes} \catCohb( \bfAnCoh(Y\an) \times \bfAnCoh(Y\an) )\\[3pt]
		 &\xrightarrow{q_\ast \circ p^\ast} \catCohb( \bfAnCoh(Y\an) )\ ,
	\end{align}
	where the map on the right-hand side is induced by the 1-morphism in correspondences:
	\begin{align}
		\begin{tikzcd}[ampersand replacement = \&]
		{} \& \bfAnCohext(Y\an) \arrow{dr}{q} \arrow{dl}[swap]{p} \\
		\bfAnCoh(Y\an) \times \bfAnCoh(Y\an) \& \& \bfAnCoh(Y\an) 
		\end{tikzcd}\ ,
	\end{align}
	endows $\catCohb(\bfAnCoh(Y))$ with the structure of an $\mathbb E_1$-monoidal stable $\infty$-category.
\end{theorem}

\begin{proof}
	The only main point to emphasize is how to use the tor-amplitude estimates for the map $p$ in the algebraic case (i.e., Proposition~\ref{prop:extremal_projections_scheme} and Corollaries~\ref{cor:extremal_projections_betti}, \ref{cor:dR-lci}, and \ref{cor:Dol-lci}) in the analytic setting.
	First of all, we use Lemmas \ref{lem:equivalence_1} and \ref{lem:equivalence_2} to identify the $2$-Segal object $\calS_\bullet \bfAnCoh(Y\an)$ with $( \calS_\bullet \bfCoh(Y) )\an$.
	Then it remains to check that $p\an$ is derived lci, where now $p$ is the map appearing in \eqref{eq:convolution_diagram_algebraic}.
	This follows by combining Lemma~\ref{lem:analytification_relative_tor_amplitude} and \cite[Theorem~5.21]{Porta_Yu_Representability}.
\end{proof}

\begin{corollary} \label{cor:comparison_algebraic_vs_analytic_CoHA}
	Let $Y$ be as in Theorem~\ref{thm:coh_algebra_in_correspondence_analytic}.
	Then the derived analytification functor induces a morphism in $\Alg_{\mathbb E_1}( \Cat_\infty^{\mathsf{st}} )$
	\begin{align}
		\catCohb( \bfCoh(Y) ) \longrightarrow \catCohb( \bfAnCoh(Y\an) ) \ . 
	\end{align}
\end{corollary}

\begin{proof}
	By using Lemmas~\ref{lem:equivalence_1} and \ref{lem:equivalence_2}, we have $\catCohb( \bfCoh(Y)\an ) \simeq \catCohb( \bfAnCoh(Y\an) )$ as $\mathbb E_1$-monoid objects. 
	The analytification functor $(-)\an$ promotes to a symmetric monoidal functor
	\begin{align}
		(-)\an \colon \Corr^\times( \dSt ) \longrightarrow \Corr^\times( \dAnSt ) \ .
	\end{align}
	Combining Lemma \ref{lem:analytification_relative_tor_amplitude} and \cite[Theorem 5.21]{Porta_Yu_Representability}, we conclude that $(-)\an$ preserves lci morphisms.
	Moreover, \cite[Lemma 3.1(3)]{Porta_Yu_Representability} and \cite[Proposition 6.3]{Porta_Yu_Higher_analytic_stacks_2014}, we see that $(-)\an$ also preserves proper morphisms.
	Finally, using the derived GAGA theorems \cite[Theorems 7.1 \& 7.2]{Porta_GAGA} we see that $(-)\an$ takes morphisms which are representable by proper schemes to morphisms which are representable by proper analytic spaces\footnote{Using \cite[Proposition 6.3]{Porta_Yu_Higher_analytic_stacks_2014} it is enough to prove that the analytification takes representable morphisms with geometric target to representable morphisms. This immediately follows from \cite[Proposition 2.25]{Porta_Yu_Higher_analytic_stacks_2014}.}.
	Therefore, it restricts to a symmetric monoidal functor
	\begin{align}
		(-)\an \colon \Corr^\times( \dSt )_{\mathsf{rps}, \mathsf{lci}} \longrightarrow \Corr^\times( \dAnSt )_{\mathsf{rps}, \mathsf{lci}} \ .
	\end{align}
	The analytification functor for coherent sheaves induces a natural transformation of right-lax symmetric monoidal functors
	\begin{align}
		\catCohb \longrightarrow \catCohb \circ (-)\an \ .
	\end{align}
	Here both functors are considered as functors $\dSt \to \Cat_\infty$.
	Using the universal property of the category of correspondences, we can extend this natural transformation of right-lax symmetric monoidal functors defined over the category of correspondences.
	The key point is to verify that if $p \colon X \to Y$ is a proper morphism of geometric derived stacks locally almost of finite presentation, then the diagram
	\begin{align}
		\begin{tikzcd}[ampersand replacement = \&]
			\catCohb(X) \arrow{r}{(-)\an} \arrow{d}{p_\ast} \& \catCohb( X\an ) \arrow{d}{p\an_\ast} \\
			\catCohb(Y) \arrow{r}{(-)\an} \& \catCohb( Y\an ) 
		\end{tikzcd}
	\end{align}
	commutes.
	This is a particular case of \cite[Theorem~7.1]{Porta_GAGA}.
	The conclusion follows.
\end{proof}

\subsection{The derived Riemann--Hilbert correspondence}

Let $X$ be a smooth proper connected complex scheme.
In \cite[\S3]{Porta_Derived} there is constructed a natural transformation
\begin{align}
	\eta_{\mathsf{RH}} \colon X\an_\dR \longrightarrow X\an_\B \ ,
\end{align}
which induces for every derived analytic stack $Y \in \dAnSt$ a morphism
\begin{align}
	 \eta_{\mathsf{RH}}^\ast \colon \bfAnMap( X\an_\dR, Y ) \longrightarrow \bfAnMap( X\an_\B, Y ) \ . 
\end{align}
It is then shown in \cite[Theorem 6.11]{Porta_Derived} that this map is an equivalence when $Y = \bfAnPerf$.\footnote{See \cite[Corollary~7.6]{Porta_Holstein_Mapping} for a discussion of which other derived analytic stacks $Y$ see $\eta_{\mathsf{RH}}$ as an equivalence.}
Taking $Y = \calS_\bullet \bfAnPerf$, we see that $\eta_{\mathsf{RH}}$ induces a morphism of $2$-Segal objects
\begin{align}
	 \eta_{\mathsf{RH}}^\ast \colon \calS_\bullet \bfAnPerf( X\an_\dR ) \longrightarrow \calS_\bullet \bfAnPerf( X\an_\B ) \ .
\end{align}
By applying the functor $\TwoSeg( \dAnSt ) \to \Alg_{\mathbb E_1}( \Corr^\times( \dAnSt ) )$, we therefore conclude that
\begin{align}
	 \eta_{\mathsf{RH}}^\ast \colon \bfAnPerf( X\an_\dR ) \longrightarrow \bfAnPerf( X\an_\B )
\end{align}
acquires a natural structure of morphism between $\mathbb E_1$-convolution algebras:

\begin{proposition}
	The morphism
	\begin{align}
		\eta_{\mathsf{RH}}^\ast \colon \calS_\bullet \bfAnPerf( X\an_\dR ) \longrightarrow \calS_\bullet \bfAnPerf( X\an_\B )
	\end{align}
	is an equivalence.
	Moreover, it restricts to an equivalence
	\begin{align}
		\eta_{\mathsf{RH}}^\ast \colon \calS_\bullet \bfAnCoh_\dR( X ) \longrightarrow \scrS\bfAnCoh_\B( X ) \ . 
	\end{align}
\end{proposition}

\begin{proof}
	Fix a derived Stein space $S \in \dStn$.
	Then \cite[Theorem~6.11]{Porta_Derived} provides an equivalence of stable $\infty$-categories
	\begin{align}
		\catPerf( X\an_\dR \times S ) \simeq \catPerf( X\an_\B \times S ) \ . 
	\end{align}
	Therefore, for every $n \ge 0$ we obtain an equivalence
	\begin{align}
		\calS_n \bfAnPerf( X\an_\dR )(S) \simeq \Fun( T_n , \catPerf( X\an_\dR \times S ) ) \simeq \Fun( T_n, \catPerf( X\an_\B \times S ) ) \simeq \calS_n \bfAnPerf( X\an_\B )(S) \ .
	\end{align}
	The first statement follows at once.
	The second statement follows automatically given the commutativity of the natural diagram
	\begin{align}
		\begin{tikzcd}[ampersand replacement = \&]
		{} \& X\an \arrow{dl}[swap]{\lambda_X} \arrow{dr} \\
		X\an_\dR \arrow{rr}{\eta_{\mathsf{RH}}} \& \& X\an_\B .
		\end{tikzcd}
	\end{align}
	\end{proof}

\begin{theorem}[CoHA version of the derived Riemann--Hilbert correspondence]\label{thm:RH}
	There is an equivalence of stable $\mathbb E_1$-monoidal $\infty$-categories
	\begin{align}
		( \catCohb( \bfAnCoh_\dR(X) ) , \ostar_\dR\an ) \simeq ( \catCohb( \bfAnCoh_\B(X) ), \ostar_\B\an ) \ .
	\end{align}
\end{theorem}

\begin{rem}
	In the algebraic setting we considered the finer invariant $\catCohb_{\mathsf{pro}}$, which is more adapted to the study of non-quasi-compact stacks.
	Among its features, there is the fact that for every derived stack $Y$ there is a canonical equivalence (cf.\ Proposition \ref{prop:quasi_compact_G_truncation})
\begin{align}
		 K( \catCohb_{\mathsf{pro}}(Y) ) \simeq K( \catCohb_{\mathsf{pro}}(\trunc{Y}) ) . 
\end{align}
	In the $\C$-analytic setting, a similar treatment is possible, but it is more technically involved.
	In the algebraic setting, the construction of $\catCohb_{\mathsf{pro}}$ relies on the machinery developed in \S\ref{sec:ind_quasi_compact}, which provides a canonical way of organizing exhaustion by quasi-compact substacks into a canonical ind-object.
	In the $\C$-analytic setting, one cannot proceed verbatim, because quasi-compact $\C$-analytic substacks are extremely rare and it is not true that every geometric derived analytic stack admits an open exhaustion by quasi-compact ones.
	Rather, one would have to use compact Stein subsets, see \cite[Definition 2.14]{Porta_Holstein_Mapping}.
	Combining \cite[Corollary 4.5.1.10]{Lurie_SAG} and \cite[Theorem 4.13]{Porta_Holstein_Mapping}, it would then be possible to compare the $K$-theory of the resulting pro-category of bounded coherent sheaves on a derived analytic stack $Y$ with the one of the classical trucation $\trunc{Y}$.
	We will not develop the full details here.
\end{rem}

\appendix

\bigskip\section{Ind quasi-compact stacks} \label{sec:ind_quasi_compact}

The main object of study of the paper is the derived stack $\bfCoh(S)$ of coherent sheaves on $S$, where $S$ is a smooth and proper scheme or one of Simpson's shapes of a smooth and proper scheme.
This stack is typically not quasi-compact, and this requires some care when studying its invariants, such as the $G$-theory.
For example, when $X$ is a quasi-compact geometric derived stack, the inclusion $i \colon \trunc{X} \hookrightarrow X$ induces a canonical equivalence
\begin{align}
	i_\ast \colon G(\trunc{X}) \xrightarrow{\sim} G(X) \ .
\end{align}
This relies on Quillen's theorem of the heart and the equivalence $\catCoh^\heartsuit(\trunc{X}) \simeq \catCoh^\heartsuit(X)$ induced by $i_\ast$.
In particular, one needs quasi-compactness of $X$ to ensure that the \textit{t}-structure on $\catCohb(X)$ is (globally) bounded.
In this appendix, we set up a general framework to deal with geometric derived stacks that are not necessarily quasi-compact.

\subsection{Open exhaustions}

Let $j \colon \dGeomqc \hookrightarrow \dSt$ be the inclusion of the full subcategory of $\dSt$ spanned by quasi-compact geometric derived stacks.
Left Kan extension along $j$ induces a functor
\begin{align}
	\Psi \colon \dSt \longrightarrow \PSh( \dGeomqc ) \ .
\end{align}
We have:

\begin{lemma} \label{lem:quasi_compact_open_exhaustion}
	Let $X \in \dGeom$ be a locally geometric derived stack.
	Then:
	\begin{enumerate}\itemsep=0.2cm
		\item \label{lem:quasi_compact_open_exhaustion-(1)} There exists a (possibly transfinite) sequence 
		\begin{align}
				\emptyset = U_0 \hookrightarrow U_1 \hookrightarrow \cdots U_\alpha \hookrightarrow U_{\alpha+1} \hookrightarrow \cdots 
		\end{align}
		of quasi-compact Zariski open substacks of $X$ whose colimit is $X$.
		
		\item \label{lem:quasi_compact_open_exhaustion-(2)} Let $Y \in \dGeomqc$ be a quasi-compact geometric derived stack.
		For any exhaustion of $X$ by quasi-compact Zariski open substacks of $X$ as in the previous point, the canonical morphism
		\begin{align}
			\colim_\alpha \Map_{\dSt}( Y, U_\alpha ) \longrightarrow \Map_{\dSt}( Y, X )
		\end{align}
		is an equivalence.
		
		\item \label{lem:quasi_compact_open_exhaustion-(3)} The object $\Psi(X)$ belongs to the full subcategory $\Ind( \dGeomqc )$ of $\PSh( \dGeomqc )$.
	\end{enumerate}
\end{lemma}

\begin{proof}
	Let $V \to X$ be a smooth atlas, where $V$ is a scheme.
	Let $V' \hookrightarrow V$ be the inclusion of a quasi-compact open Zariski subset.
	Let
	\begin{align}
		V'_\bullet \coloneqq \check{\calC}( V' \to X ) 
	\end{align}
	be the \v{C}ech nerve of $V' \to X$ and set
	\begin{align}
		U' \coloneqq | V'_\bullet | \ . 
	\end{align}
	The canonical map $U' \to X$ is representable by open Zariski immersions, and $U'$ is a quasi-compact stack.
	Since this construction is obviously functorial in $V'$, we see that any exhaustion of $V$ by quasi-compact Zariski open subschemes induces a similar exhaustion of $X$, thus completing the proof of point \eqref{lem:quasi_compact_open_exhaustion-(1)}.
	
	We now prove point \eqref{lem:quasi_compact_open_exhaustion-(2)}.
	Fix an exhaustion of $X$ by quasi-compact Zariski open substacks of $X$ as in point \eqref{lem:quasi_compact_open_exhaustion-(1)}.
	For every index $\alpha$, the map $U_\alpha \to U_{\alpha+1}$ is an open immersion and therefore it is $(-1)$-truncated in $\dSt$.
	Using \cite[Proposition~5.5.6.16]{HTT}, we see that
	\begin{align}
		\Map_{\dSt}(Y, U_\alpha) \longrightarrow \Map_{\dSt}(Y, U_{\alpha + 1}) 
	\end{align}
	is $(-1)$-truncated as well.
	The same goes for the maps $\Map_{\dSt}(Y, U_\alpha) \to \Map_{\dSt}(Y, X)$.
	As a consequence, the map
	\begin{align}
		 \colim_\alpha \Map_{\dSt}( Y, U_\alpha ) \longrightarrow \Map_{\dSt}( Y, X ) 
	\end{align}
	is $(-1)$-truncated.
	To prove that it is an equivalence, we are left to check that it is surjective on $\pi_0$.
	Let $f \colon Y \to X$ be a morphism.
	Write $Y_\alpha \coloneqq U_\alpha \times_X Y$.
	Then the sequence $\{Y_\alpha\}$ is an open exhaustion of $Y$, and since $Y$ is quasi-compact there must exist an index $\alpha$ such that $Y_\alpha = Y$.
	This implies that $f$ factors through $U_\alpha$, and therefore the proof of \eqref{lem:quasi_compact_open_exhaustion-(2)} is achieved.
	
	As for point \eqref{lem:quasi_compact_open_exhaustion-(3)}, this immediately follows from \eqref{lem:quasi_compact_open_exhaustion-(2)} and \cite[Corollary~5.3.5.4-(1)]{HTT}.
\end{proof}

\begin{corollary} \label{cor:quasi_compact_ind}
	The functor $\Psi$ restricts to a limit preserving functor
	\begin{align}
		(-)_{\mathsf{ind}} \colon \dGeom \longrightarrow \Ind( \dGeomqc ) \ .
	\end{align}
\end{corollary}

\subsection{G-theory of non-quasi-compact stacks}\label{ss:g-theory-non-quasi-compact}

As a consequence, when $X$ is a locally geometric derived stack, we can canonically promote $\catCohb(X)$ to a \textit{pro-category}
\begin{align}
	\catCohb_{\mathsf{pro}}(X) \coloneqq \catCohb( X_{\mathsf{ind}} ) \simeq \flim_\alpha \catCohb(U_\alpha) \in \mathsf{Pro}(\Cat_\infty^{\mathsf{st}}) . 
\end{align}
In particular, we can give the following definition:

\begin{definition}
	Let $X \in \dGeom$ be a locally geometric derived stack.
	The pro-spectrum of $G$-theory of $X$ is
	\begin{align}
		G_{\mathsf{pro}}(X) \coloneqq K(\catCohb_{\mathsf{pro}}(X)) \in \mathsf{Pro}( \mathsf{Sp} ) \ .
	\end{align}
	The spectrum of $G$-theory of $X$ is the realization of $G_{\mathsf{pro}}(X)$:
	\begin{align}
		G(X) \coloneqq \lim G_{\mathsf{pro}}(X) \in \mathsf{Sp}\ . 
	\end{align}
\end{definition}

\begin{remark}
	If $X$ is quasi-compact, then $X_{\mathsf{ind}}$ is equivalent to a constant ind-object.
	As a consequence, both $\catCohb_{\mathsf{pro}}(X)$ and $G_{\mathsf{pro}}(X)$ are equivalent to constant pro-objects and $G(X)$ simply coincides with the spectrum $K( \catCohb(X) )$.
\end{remark}

\begin{proposition} \label{prop:quasi_compact_G_truncation}
	Let $X \in \dGeom$ be a locally geometric derived stack.
	The inclusion $i \colon \trunc{X} \hookrightarrow X$ induces a canonical equivalence
	\begin{align}
		 i_\ast \colon G_{\mathsf{pro}}(\trunc{X}) \xrightarrow{\sim} G_{\mathsf{pro}}(X)\ , 
	\end{align}
	and therefore an equivalence
	 \begin{align}
		 i_\ast \colon G(\trunc{X}) \xrightarrow{\sim} G(X) \ . 
	 \end{align}
\end{proposition}

\begin{proof}
	Choose an exhaustion $\{U_\alpha\}$ of $X$ by quasi-compact open Zariski subsets as in Lemma~\ref{lem:quasi_compact_open_exhaustion}-\eqref{lem:quasi_compact_open_exhaustion-(1)}.
	Then $\{ \trunc{U}_\alpha \}$ is an exhaustion of $\trunc{X}$, and the map $i_\ast \colon G_{\mathsf{pro}}(\trunc{X}) \to G_{\mathsf{pro}}(X)$ can be computed as
	\begin{align}
		 \flim_\alpha K( \catCohb(\trunc{U}_\alpha) ) \longrightarrow \flim_\alpha K( \catCohb( U_\alpha ) ) \ .
	\end{align}
	Since each $U_\alpha$ is quasi-compact, this is a level-wise equivalence.
	Therefore, it is also an equivalence at the level of pro-objects.
	The second statement follows by passing to realizations.
\end{proof}

\begin{definition}
	Let $X \in \dGeom$ be a locally geometric derived stack. We define
	\begin{align}
		G_0(X)\coloneqq \pi_0 G(X) \ .
	\end{align}
\end{definition}

\begin{remark}
	In \cite{Sala_Schiffmann, KV_Hall}, the authors defined $G_0$ of a non-quasi-compact geometric classical stack $Y$ as the limit of the $G_0(V_\alpha)$ for an exhaustion $\{V_\alpha\}$ of $Y$ by quasi-compact Zariski open substacks. 
	
	The relation between the above two definitions is given as follows. Let $X \in \dGeom$ be a locally geometric derived stack and let $\{U_\alpha\}$ be an exhaustion of $X$ by quasi-compact Zariski open substacks. Then there exists a short exact sequence
	\begin{align}
		0\longrightarrow \lim_\alpha{}^{1}\,  \pi_1 G(U_\alpha) \longrightarrow G_0(X) \longrightarrow \lim_\alpha G_0(U_\alpha)\longrightarrow 0
	\end{align} 
	in the abelian category of abelian groups.
\end{remark}

\begin{remark}\label{rem:quasi-compactness-Coh}
	Now, we discuss the quasi-compactness of the moduli stacks of coherent sheaves we deal with in the main body of the paper.
	
	Let $Y$ be a smooth projective complex variety. First, recall that the classical stack $\trunc{\bfCoh(Y)}$ decomposes into the disjoint union
	\begin{align}
		\trunc{\bfCoh(Y)} = \bigsqcup_{P}\, \trunc{\bfCoh^{P}(Y)}\ ,
	\end{align}
	with respect to the Hilbert polynomials of coherent sheaves. Here, the stacks on the right-hand-side are introduced in \S~\ref{ss:other-moduli-spaces}. We have a corresponding decomposition at the level of the derived enhancements
	\begin{align}\label{eq:decomposition}
		\bfCoh(Y) = \bigsqcup_{P}\, \bfCoh^{P}(Y)\ .
	\end{align}
	
	Now, fix $P(m)\in \Q[m]$. As shown e.g. in the proof of \cite[Théorè\-me~4.6.2.1]{Laumon_Champs}, there exists an exhaustion $\{U_n^P\}_{n\in N}$ of $\bfCoh^P(Y)$ by quasi-compact open substacks, such that the truncation $\trunc{U}_n^P$ is a quotient stack by a certain open subset of a Quot scheme for any $n$. 
	By \cite[Proposition~4.1.1]{KV_Hall}, a similar description holds for the moduli stack $\bfCoh_{\mathsf{prop}}(Y)$ of coherent sheaves with proper support on a smooth quasi-projective complex variety $Y$. Note that the decomposition \eqref{eq:decomposition} for the stack $\bfCoh_0(Y)\coloneqq \bfCoh_{\mathsf{prop}}^{\leq 0}(Y)$ reduces to
	\begin{align}
		\bfCoh_0(Y) = \bigsqcup_{k\in \Z_{\geq 0}}\, \bfCoh_0^{k}(Y)\ .
	\end{align}
	By using the explicit description of the $U_n^k$'s, one can show that $\trunc{\bfCoh_0^{k}(Y)}$ is a quasi-compact quotient stack, hence the stack $\bfCoh_0^{k}(Y)$ is quasi-compact.
	
	Now, let $Y$ be a smooth projective complex curve. The moduli stack $\bfCoh_\Dol(Y)$ is not quasi-compact. On the other hand, the moduli stacks $\bfBun^n_\B(Y)$ and $\bfBun^n_\dR(Y)$ are quasi-compact quotient stacks. The truncations of these stacks are quotients by the Betti and de Rham \textit{representation spaces} respectively (cf.\ \cite{Simpson_Moduli_II}). The derived stacks are quotients by the derived versions of these representation spaces (see e.g.  \cite[\S1.2]{PT_Local}).
\end{remark}

\subsection{Correspondences}\label{ss:correspondences-ind}

We finish this section by providing a formal extension of Gaitsgory--Rozenblyum correspondence machine in the setting of not necessarily quasi-compact stacks.
Let $\mathbb S$ be an $(\infty,2)$-category, seen as an $(\infty,1)$-category weakly enriched in $\Cat_\infty$, in the sense of \cite{Gepner_Haugseng_Enriched,Hinich_Yoneda}.
We write $\Cat_\infty^{(2)}$ for $\Cat_\infty$ thought as weakly enriched over itself in the natural way (i.e.\ for the $(\infty,2)$-category of $(\infty,1)$-categories).
Consider the $2$-categorical Yoneda embedding
\begin{align}
	y \colon \mathbb S \longrightarrow 2\textrm{-}\Fun\Big(\mathbb S^{1\textrm{-}\mathsf{op}}, \Cat_\infty^{(2)}\Big) \ . 
\end{align}
Then \cite[Corollary~6.2.7]{Hinich_Yoneda} guarantees that is $2$-fully faithful.
We let $2\textrm{-}\Ind(\mathbb S)$ be the full $2$-subcategory of $2\textrm{-}\Fun\Big(\mathbb S^{1\textrm{-}\mathsf{op}}, \Cat_\infty^{(2)}\Big)$ spanned by those functors that commute with finite $\Cat_\infty$-limits.
The fully faithful functor $\scrS \to \Cat_\infty$ induces a fully faithful embedding
\begin{align}
	\Ind(\mathbb S^{1\textrm{-}\mathsf{cat}}) \to \left(2\textrm{-}\Fun\Big(\mathbb S^{1\textrm{-}\mathsf{op}}, \Cat_\infty^{(2)}\Big)\right)^{1\textrm{-}\mathsf{cat}} \ . 
\end{align}
We let $2 \textrm{-} \Ind(\mathbb S)$ be the full $2$-subcategory of $2\textrm{-} \Fun\Big(\mathbb S^{1\textrm{-}\mathsf{op}}, \Cat_\infty^{(2)}\Big)$ spanned by the essential image of the above functor. \footnote{We thank Andrea Gagna e Ivan Di Liberti in helping us conceiving this definition.}
By construction, the $1$-category underlying $2 \textrm{-} \Ind(\mathbb S)$ coincides with $\Ind(\mathbb S^{1\textrm{-}\mathsf{cat}})$.

\begin{remark}
	We are not sure about the intrinsic meaning of the above definition of $2\textrm{-}\Ind(\mathbb S)$ from a $2$-categorical perspective.
	It is probably too little to be the correct $2\textrm{-}\Ind$ construction, and we are not aware whether it satisfies some $2$-categorical universal property.
	If $\fcolim_\alpha x_\alpha$ and $\fcolim_\beta y_\beta$ are two objects in $2 \textrm{-} \Ind(\mathbb S)$ in the above sense, the mapping category is given by the formula
	\begin{align}
	 	\Fun_{2\textrm{-}\Ind(\mathbb S)}\left( \fcolim_\alpha x_\alpha, \fcolim_\beta y_\beta \right) = \lim_\alpha \colim_\beta \Fun_{\mathbb S}(x_\alpha, y_\beta) \ .
	\end{align} 
	This is all that is needed in what follows.
\end{remark}

We now equip $\calC$ with three markings $\calC_\mathsf{horiz}$, $\calC_\mathsf{vert}$ and $\calC_\mathsf{adm}$ satisfying the conditions of \cite[\S7.1.1.1]{Gaitsgory_Rozenblyum_Study_I}.
We define three markings $\Ind(\calC)_\mathsf{horiz}$, $\Ind(\calC)_\mathsf{vert}$ and $\Ind(\calC)_\mathsf{adm}$ on $\Ind(\calC)$ by declaring that a morphism $f \colon X \to Y$ in $\Ind(\calC)$ belongs to $\Ind(\calC)_\mathsf{horiz}$ (resp.\ $\Ind(\calC)_\mathsf{vert}$, $\Ind(\calC)_\mathsf{adm}$) if it is representable by morphisms in $\calC_\mathsf{horiz}$ (resp.\ $\calC_\mathsf{vert}$, $\calC_\mathsf{adm}$).
It is then straightforward to check that the conditions of \cite[\S7.1.1.1]{Gaitsgory_Rozenblyum_Study_I} are again satisfied.
Next, we let
\begin{align}
	 \Phi \colon \calC\op \longrightarrow \Cat_\infty
\end{align}
be a functor and let $\Phi_{\mathsf{horiz}}$ be its restriction to $(\calC_{\mathsf{horiz}})\op$.
Passing to ind-objects, we obtain a functor
\begin{align}
	\Phi^{\mathsf{pro}} \colon \Ind(\calC)\op \longrightarrow \mathsf{Pro}(\Cat_\infty) \ .
\end{align}
We let $\Phi^{\mathsf{pro}}_{\mathsf{horiz}}$ be its restriction to $(\Ind(\calC)_{\mathsf{horiz}})\op$.
Before stating the next key proposition, we recall the definition of the bivariance property:

\begin{definition}
	A functor $\Phi \colon \calC\op \to \Cat_\infty$ is said to be \emph{$\calC_{\mathsf{vert}}$-right bivariant} if for every morphism $f \colon X \to Y$ in $\calC_{\mathsf{vert}}$, the functor $\Phi(f) \colon \Phi(Y) \to \Phi(X)$ admit a right adjoint $\Phi_\ast(f)$.
	A $\calC_{\mathsf{vert}}$-right bivariant functor is said to have \emph{base change with respect to $\calC_{\mathsf{horiz}}$} if for every pullback diagram
	\begin{align}
		 \begin{tikzcd}[ampersand replacement = \&]
			W \arrow{r}{g'} \arrow{d}{f'} \& X \arrow{d}{f} \\
			Z \arrow{r}{g} \& Y
		\end{tikzcd} 
	\end{align}
	where $f \in \calC_{\mathsf{vert}}$ and $g \in \calC_{\mathsf{horiz}}$, the square
	\begin{align}
		 \begin{tikzcd}[ampersand replacement = \&]
			\Phi(Y) \arrow{r}{\Phi(g)} \arrow{d}{\Phi(f)} \& \Phi(Z) \arrow{d}{\Phi(f')} \\
			\Phi(X) \arrow{r}{\Phi(g')} \& \Phi(W)
		\end{tikzcd}
	\end{align}
	is vertically right adjointable.
\end{definition}

\begin{remark}
	In \cite[Chapter~7]{Gaitsgory_Rozenblyum_Study_I}, the above property is not directly considered.
	It rather corresponds to the right Beck--Chevalley property (see Definition~7.3.2.2 in \textit{loc.\ cit.}) for functors with values in $(\Cat_\infty^{(2)})^{2\textrm{-}\mathsf{op}}$.
\end{remark}

\begin{proposition} \label{prop:bivariant_functors_ind}
	Keeping the above notation and assumptions, suppose furthermore that:
	\begin{enumerate}\itemsep=0.2cm
		\item $\Phi$ is $\calC_{\mathsf{vert}}$-right bivariant and has base change with respect to $\calC_{\mathsf{horiz}}$.
		
		\item Every object in $\Ind(\cC)$ can be represented as a filtered colimit whose transition maps belong to $\calC_{\mathsf{horiz}}$.
			\end{enumerate}
	Then $\Phi^{\mathsf{pro}}_{\mathsf{horiz}}$ is $\Ind(\calC)_{\mathsf{vert}}$-right bivariant and has base change with respect to $\Ind(\calC)_{\mathsf{horiz}}$.
\end{proposition}

\begin{proof}
	Let $f \colon X \to Y$ be a morphism in $\Ind(\cC)$.
	Choose a representation $Y \simeq \fcolim_\alpha Y_\alpha$, where the transition maps belong to $\calC_{\mathsf{horiz}}$.
	For every index $\alpha$, we let
	\begin{align}
		 X_\alpha \coloneqq Y_\alpha \times_Y X 
	\end{align}
	and we let $f_\alpha \colon X_\alpha \to Y_\alpha$ be the induced morphism.
	By definition of $\Ind(\calC)_{\mathsf{vert}}$, $X_\alpha$ belongs to $\calC$ and $f_\alpha$ is a morphism in $\calC_{\mathsf{vert}}$.
	The morphisms
	\begin{align}
		 \Phi(f_\alpha) \colon \Phi(Y_\alpha) \longrightarrow \Phi(X_\alpha) 
	\end{align}
	admit a left adjoint $\Phi_!(f_\alpha)$.
	Since $\Phi_{\mathsf{horiz}}$ satisfies the right Beck--Chevalley property with respect to $\calC_{\mathsf{vert}}$, the morphisms $\Phi_!(f_\alpha)$ assemble into a morphism
	\begin{align}
		\Phi_!(f) \colon \flim_\alpha \Phi(X_\alpha) \longrightarrow \flim_\alpha \Phi(Y_\alpha) 
	\end{align}
	in $2\textrm{-}\mathsf{Pro}(\Cat_\infty^{(2)})$.
	The triangular identities for the adjunction $\Phi_!(f_\alpha) \dashv \Phi(f_\alpha)$ induce triangular identities exhibiting $\Phi_!(f)$ as a left adjoint to $\Phi(f)$ in the $(\infty,2)$-category $2\textrm{-}\mathsf{Pro}(\Cat_\infty^{(2)})$.
	For every morphism $Z \to Y$ in $\Ind(\calC)_{\mathsf{horiz}}$, we let $Z_\alpha \coloneqq Y_\alpha \times_Y Z$.
	The induced morphism $Z_\alpha \to Y_\alpha$ belongs to $\calC_{\mathsf{horiz}}$ by definition.
	In this way, we can describe the Beck--Chevalley transformation for the diagram
	\begin{align}
		\begin{tikzcd}[ampersand replacement = \&]
		\Phi(Y) \arrow{r} \arrow{d} \& \Phi(X) \arrow{d} \\
		\Phi(Z) \arrow{r} \& \Phi(X \times_Y Z)
		\end{tikzcd} 
	\end{align}
	in terms of the Beck--Chevalley transformation for the diagram
	\begin{align}
		\begin{tikzcd}[ampersand replacement = \&]
		\Phi(Y_\alpha) \arrow{r} \arrow{d} \& \Phi(X_\alpha) \arrow{d} \\
		\Phi(Z_\alpha) \arrow{r} \& \Phi(X_\alpha \times_{Y_\alpha} Z_\alpha) 
		\end{tikzcd}\ ,
	\end{align}
	which holds by assumption.
\end{proof}

Seeing $\mathsf{Pro}(\Cat_\infty)$ as the underlying $1$-category of $2\textrm{-}\mathsf{Pro}(\Cat_\infty^{(2)})$, we obtain:

\begin{corollary} \label{cor:bivariant_extension_ind}
	Keeping the above notation, assume $\calC_{\mathsf{vert}} \subset \calC_{\mathsf{horiz}}$ and $\calC_{\mathsf{vert}} = \calC_{\mathsf{adm}}$.
	Then under the same assumptions of Proposition~\ref{prop:bivariant_functors_ind}, the functor
	\begin{align}
		\Phi^{\mathsf{pro}} \colon \Ind(\calC)\op \longrightarrow \mathsf{Pro}(\Cat_\infty) 
	\end{align}
	uniquely extends to a functor
	\begin{align}
		\Phi^{\mathsf{pro}}_{\mathsf{corr}} \colon \Corr( \Ind(\calC) )_{\mathsf{vert}, \mathsf{horiz}}^{\mathsf{vert}} \longrightarrow 2 \textrm{-} \mathsf{Pro}(\Cat_\infty^{(2)}) \ , 
	\end{align}
	and its restriction to $\Corr(\Ind(\calC))_{\mathsf{vert}, \mathsf{horiz}}$ factors through the maximal $(\infty,1)$-category $\mathsf{Pro}(\Cat_\infty)$ of $2\textrm{-}\mathsf{Pro}(\Cat_\infty^{(2)})$.
\end{corollary}

\begin{proof}
	It is enough to apply \cite[Theorem~7.3.2.2-(b)]{Gaitsgory_Rozenblyum_Study_I} to the $(\infty,2)$-category
	\begin{align}
			\mathbb S = 2\textrm{-}\mathsf{Pro}(\Cat_\infty^{(2)})^{2\textrm{-}\mathsf{op}} \ .
	\end{align}
	See also \cite[Theorem~4.2.6]{MacPherson}.
\end{proof}

\begin{corollary} \label{cor:ind_qcoh}
	There exists a uniquely defined right lax symmetric monoidal functor
	\begin{align}
		 \catQCoh_{\mathsf{pro}} \colon \Corr^\times(\Ind(\dGeom^{\mathsf{qc}}))_{\mathsf{rps}, \mathsf{all}} \longrightarrow \mathsf{Pro}(\Cat_\infty) \ .
	\end{align}
\end{corollary}

\begin{proof}
	Take $\calC = \mathsf{dSch}^{\mathsf{qc}}$ and $\calD = \dGeom^{\mathsf{qc}}$.
	For $\calC$, we take $\mathsf{horiz} = \mathsf{all}$ and $\mathsf{adm} = \mathsf{vert} = \mathsf{proper}$.
	Observe that condition (5) in \cite[\S7.1.1.1]{Gaitsgory_Rozenblyum_Study_I} is satisfied.
	Consider the functor
	\begin{align}
	 	\catQCoh \colon (\mathsf{dSch}^{\mathsf{qc}})\op \longrightarrow \Cat_\infty \ . 
	\end{align}
	Applying \cite[Theorem~7.3.2.2]{Gaitsgory_Rozenblyum_Study_I}, we obtain a functor
	\begin{align}
		\catQCoh \colon \Corr( \mathsf{dSch}^{\mathsf{qc}} )_{\mathsf{proper}, \mathsf{all}}^{\mathsf{proper}} \longrightarrow \Cat_\infty^{(2)} \ .
	\end{align}
	For $\calC, \calD$ we now take $\mathsf{horiz} = \mathsf{all}$ and $\mathsf{vert} = \mathsf{rps}$ (morphisms that are representable by proper schemes) and $\mathsf{adm} = \mathsf{isom}$.
	Then \cite[Theorem~8.6.1.5]{Gaitsgory_Rozenblyum_Study_I} supplies an extension of $\catQCoh$ as
	\begin{align}
		\catQCoh \colon \Corr( \dGeom^{\mathsf{qc}} )_{\mathsf{rps}, \mathsf{all}}^{\mathsf{rps}}  \longrightarrow \Cat_\infty^{(2)}\ . 
	\end{align}
	Thanks again to \cite[Theorem~7.3.2.2]{Gaitsgory_Rozenblyum_Study_I}, the above functor is uniquely determined by its restriction
	\begin{align}
		\catQCoh \colon( \dGeom^{\mathsf{qc}} )\op  \longrightarrow \Cat_\infty\ . 
	\end{align}
	This is our $\Phi$. It satisfies the hypothesis of Proposition~\ref{prop:bivariant_functors_ind}. Thus, applying Corollary~\ref{cor:bivariant_extension_ind}, we obtain a functor
	\begin{align}
		\catQCoh \colon \Corr( \Ind( \mathsf{dGeom}^{\mathsf{qc}} ) )_{\mathsf{rps}, \mathsf{all}}^{\mathsf{rps}} \longrightarrow 2 \textrm{-} \mathsf{Pro}(\Cat_\infty^{(2)}) \ , 
	\end{align}
	which we restrict to a functor
	\begin{align}
		\catQCoh_{\mathsf{pro}} \colon \Corr( \Ind( \mathsf{dGeom}^{\mathsf{qc}} ))_{\mathsf{rps}, \mathsf{all}} \longrightarrow \mathsf{Pro}(\Cat_\infty) \ . 
	\end{align}
	Combining \cite[Theorem~4.4.6]{MacPherson} and \cite[Proposition~9.3.2.4]{Gaitsgory_Rozenblyum_Study_I} we conclude that we can upgrade these constructions to right lax symmetric monoidal functors.
\end{proof}

\bigskip

\providecommand{\bysame}{\leavevmode\hbox to3em{\hrulefill}\thinspace}
\providecommand{\href}[2]{#2}

\end{document}